\documentclass[a4paper]{amsart}
\usepackage{amssymb, amsmath, amsthm, anysize, color, float, longtable}

\theoremstyle{plain}
\newtheorem{theorem}{Theorem}[section]
\newtheorem{cor}[theorem]{Corollary}
\newtheorem{lemma}[theorem]{Lemma}

\newtheorem{proposition}[theorem]{Proposition}
\newtheorem{definition}[theorem]{Definition}

\newcommand{\PGL}{\mathop{\mathrm{PGL}}}
\newcommand{\PSp}{\mathop{\mathrm{PSp}}}

\newcommand{\GL}{\mathop{\mathrm{GL}}}
\newcommand{\GU}{\mathop{\mathrm{GU}}}

\newcommand{\SL}{\mathop{\mathrm{SL}}}
\newcommand{\Sp}{\mathop{\mathrm{Sp}}}

\newcommand{\Syl}{\mathop{\mathrm{Syl}}}
\newcommand{\AGL}{\mathop{\mathrm{AGL}}}
\newcommand{\GF}{\mathop{\mathrm{GF}}}
\newcommand{\PSL}{\mathop{\mathrm{PSL}}}
\newcommand{\PSU}{\mathop{\mathrm{PSU}}}
\newcommand{\PGU}{\mathop{\mathrm{PGU}}}
\newcommand{\SU}{\mathop{\mathrm{SU}}}
\newcommand{\SO}{\mathop{\mathrm{SO}}}
\newcommand{\PSO}{\mathop{\mathrm{PSO}}}
\newcommand{\PO}{\mathop{\mathrm{PO}}}
\newcommand{\POmega}{\mathop{\mathrm{P}\Omega}}

\newcommand{\Sym}{\mathop{\mathrm{Sym}}}
\newcommand{\Alt}{\mathop{\mathrm{Alt}}}

\DeclareMathOperator{\Aut}{Aut}
\DeclareMathOperator{\Inn}{Inn}
\DeclareMathOperator{\soc}{soc}
\DeclareMathOperator{\supp}{supp}

\newcommand{\la}{\langle}
\newcommand{\ra}{\rangle}
\DeclareMathOperator{\Wr}{wr}
\newcommand{\PGammaL}{\mathop{\mathrm{P}\Gamma\mathrm{L}}}

\newcommand{\sdp}{\rtimes}

\newcommand{\norml}{\vartriangleleft}
\newcommand{\Out}{\mathrm{Out}}
\newcommand{\normleq}{\trianglelefteq}

\definecolor{darkblue}{rgb}{0,0,0.8}

\newcommand{\tont}{$\frac{3}{2}$-transitive}

\renewcommand\le{\leqslant}
\renewcommand\ge{\geqslant}
\renewcommand\leq{\leqslant}
\renewcommand\geq{\geqslant}

\begin{document}

\title{The classification of almost simple $\tfrac{3}{2}$-transitive groups}

\author{John Bamberg}
\email{John.Bamberg@uwa.edu.au}

\author{Michael Giudici}
\email{Michael.Giudici@uwa.edu.au}

\author{Martin W. Liebeck}
\email{m.liebeck@imperial.ac.uk}

\author{Cheryl E. Praeger}
\email{Cheryl.Praeger@uwa.edu.au}

\author{Jan Saxl}
\email{saxl@dpmms.cam.ac.uk}

\thanks{This paper forms part of an Australian Research Council Discovery Project. The second author was supported by an Australian Research Fellowship while the fourth author was supported by a Federation Fellowship of the Australian Research Council.}

\address[John Bamberg, Michael Giudici, Cheryl E. Praeger]{Centre for the Mathematics of Symmetry and Computation\\
School of Mathematics and Statistics\\ 
University of Western Australia, 35 Stirling Highway\\ 
Crawley, Western Australia 6009}

\address[Martin W. Liebeck]{Department of Mathematics, Imperial College, London SW7 2BZ, England}
\address[Jan Saxl]{DPMMS, CMS, University of Cambridge, Wilberforce Road, Cambridge CB3 0WB, England }

\date{}


\begin{abstract}
A finite transitive permutation group is said to be \tont\ if all the 
nontrivial orbits of a point stabiliser have the same size greater than 1. Examples include the
2-transitive groups, Frobenius groups and several other less obvious ones.
We prove that \tont\ groups are either affine or almost simple, and 
classify the latter. One of the main steps in the proof is 
an arithmetic result on the subdegrees of groups of Lie type in 
characteristic $p$: with some explicitly listed exceptions, every 
primitive action of such a group is either 2-transitive, or has a 
subdegree divisible by $p$.
\end{abstract}

\maketitle

%
%

\section{Introduction}

Burnside, in his 1897 book \cite[p 192, Theorem IX]{Bur1897}, investigated the structure of 
finite 2-transitive permutation groups. He proved that any such group 
is either affine or almost simple; in other words, the group has a
unique minimal normal subgroup which is either elementary abelian and regular, or
nonabelian simple. A transitive permutation group $G$ on a set $\Omega$ is said to be \textit{\tont} if all orbits of $G_{\alpha}$ on $\Omega\backslash\{\alpha\}$ have the same size, with this size being greater than 1. For convenience, we also count $S_2$ as a \tont\ group. A nontrivial, nonregular, normal subgroup of a 2-transitive group has to be \tont. The term {\tont} was first used in Wielandt's book
\cite{Wielandt}, where he extended Burnside's proof to show that any
{\tont} group is either primitive or a Frobenius group (that is, every 
two point stabiliser is trivial). The classification of 2-transitive
groups is a notable consequence of the classification of finite simple
groups (see \cite{CKS,Hering,Howlett,lierank3,Maillet}, and \cite{Cameron81} for 
an overview).

In this paper, we obtain two results along the way towards the 
classification of {\tont} groups. First, in parallel to Burnside's 
structure theorem, we prove

\begin{theorem}\label{tontred}
Every finite primitive {\tont} group is either affine or almost simple.
\end{theorem}

Our second result deals with the almost simple case. A {\it subdegree} of 
a transitive permutation group is a size of an orbit of the point stabiliser.

\begin{theorem}\label{tontas}
Let $G$ be a finite almost simple {\tont} group of degree $n$ 
on a set $\Omega$. Then one of the following holds:

{\rm (i)} $G$ is $2$-transitive on $\Omega$.

{\rm (ii)} $n=21$ and $G$ is $A_7$ or $S_7$ acting on the set of 
pairs of elements of $\{1,\ldots ,7\}$; the size of the nontrivial 
subdegrees is $10$.

{\rm (iii)} $n=\frac{1}{2}q(q-1)$ where $q=2^f\geq 8$, and either $G = \PSL_2(q)$
or $G=\PGammaL_2(q)$ with $f$ prime; the size of the nontrivial subdegrees
is $q+1$ or $f(q+1)$, respectively.
\end{theorem}

Theorem \ref{tontas} is a combination of Theorems \ref{antont}, \ref{thm:classgroups}(B), \ref{32excep} and \ref{sporadtont}.
The examples in parts (ii) and (iii) can be found in Lemmas \ref{lem:intrans} and \ref{lem:L2}. The groups in (iii) were first investigated in the context of 
$\frac{3}{2}$-transitivity by McDermott \cite{McDermott71}, and 
Camina and McDermott \cite{cammcd}. A characterisation of the groups in (i) and (iii) as the only \tont\ groups with trivial Fitting subgroup and all two-point stabilisers conjugate was given by Zieschang  \cite{Zieschang92}.

Affine \tont\ groups will be the subject of a future paper. The soluble 
case was handled completely by Passman in \cite{Passman67,Passman67b,
Passman69}.

Our proof of Theorem \ref{tontas}, in the main case where $G$ has socle
of Lie type, uses the following result concerning the arithmetic nature
of subdegrees of such groups. 

\begin{theorem}\label{pTheorem} 
Let $G$ be an almost simple group
  with socle $L$ of Lie type of characteristic $p$.
 Let $G$ act primitively on a set $\Omega$,  
  and let $H$ be the stabiliser of a point.  
Assume that $p$ divides $|H|$. Then one of the following holds:

\begin{enumerate}
\item[(i)] $G$ has a subdegree divisible by $p$.

\item[(ii)] $G$ is $2$-transitive on $\Omega$: here either $G = \Sp_{2d}(2)$ for $d\geq 3$ with $H = O^{\pm}_{2d}(2)$ and $|\Omega| = 2^{d-1}(2^d\mp 1)$, or $G$ is detailed in Table \ref{tab:2trans}.

\item[(iii)]  $L=\PSL_2(q)$ with $q=2^f\geq 8$, $|\Omega|=\frac{1}{2}q(q-1)$, 
$H\cap L = D_{2(q+1)}$, and $|G:L|$ is odd.
\item[(iv)] $G$ is detailed in Table \ref{tab:otherexcep}.
\end{enumerate}
\end{theorem}

\begin{table}
\caption{2-transitive cases for Theorem \ref{pTheorem}}\label{tab:2trans}
 \begin{tabular}{ccc}
\hline
  $L$ & $|\Omega|$  & Comment(s)\\
\hline
$\PSL_2(4)$ & 6 &\\
$\Sp_4(2)'$  & 6 &\\
$\Sp_4(2)'$  & 10&\\
$\PSL_2(9)$  & 6 &\\
$\PSL_4(2)$ & 8&\\
$\PSL_3(2)$ &  8 & $G=L.2$\\
$G_2(2)'$ &  $28$ & \\
$^2\!G_2(3)'$  & 9 & $G=L.3$\\
\hline
 \end{tabular}
\end{table}

\begin{table}
\caption{Other exceptions to Theorem \ref{pTheorem}}\label{tab:otherexcep}
 \begin{tabular}{cccll}
\hline
  $L$  &$|\Omega|$ & $H$ & Subdegrees & Comment(s)\\
\hline
$\PSU_3(5)$ &  $50$ & $N_G(A_7)$ & 1, 7, 42 \\
$\PSp_4(3)$ & $27$ & $N_G(2^4.A_5)$ & 1, 10, 16 &\\
$G_2(2)'$ & $36$ & $\PSL_3(2)$ & 1, 7, 7, 21 & $G=L$\\
\hline
\end{tabular}
\end{table}

We remark that the examples in Table \ref{tab:2trans} and  lines 2 and 3 of Table \ref{tab:otherexcep} are in some sense degenerate as they arise due to exceptional isomorphisms with either alternating groups or groups of Lie type of different characteristic, namely $\PSL_2(4)\cong \PSL_2(5)$, $\Sp_4(2)'\cong\PSL_2(9)\cong A_6$, $\PSL_4(2)\cong A_8$, $\PSL_3(2)\cong\PSL_2(7)$, $G_2(2)'\cong\PSU_3(3)$, ${}^2\!G_2(3)'\cong\PSL_2(8)$ and $\PSp_4(3)\cong\PSU_4(2)$. The conditions $G=L.2$ and $L.3$ in lines 6 and 8 of Table \ref{tab:2trans} are required so that $p$ divides $|H|$.

Theorem \ref{pTheorem} follows from Theorems \ref{thm:classgroups}(A) and \ref{psubdegexcep}. Theorem \ref{thm:psubdeg} in Section 3 is a result of a similar flavour 
for general primitive permutation groups.

We shall obtain several consequences of the above results. The first 
is immediate from Theorems \ref{tontred} and \ref{tontas}.

\begin{cor}\label{soctont}
The socle of a primitive {\tont} group is either regular or {\tont}.
\end{cor}

The orbitals of \tont\ groups form Schurian equivalenced non-regular schemes (see \cite{pseudo}), which form a class of \emph{pseudocyclic} association schemes. To date, there are not many known constructions of pseudocyclic schemes, and Theorem \ref{tontas} implies that there are no new examples in this subclass.

In the next result, we call a transitive permutation group 
{\it strongly} {\tont} if all non-principal
constituents of the permutation character are distinct and 
have the same degree.
Theorem 30.2 of \cite{Wielandt} states that strongly {\tont} groups
are either abelian and regular, or {\tont}. Hence Theorem \ref{tontas}, together with Lemma \ref{lem:L2}, implies the following. 

\begin{cor}\label{strongtont}
The strongly {\tont} almost simple permutation groups are precisely the 
groups in parts {\rm (i), (ii)} and {\rm (iii)} of Theorem $\ref{tontas}$.  
\end{cor}

A related notion arose in a paper of Dixon \cite{Dixon05}.
He defines a {\it $\mathbb{Q}I$-group} to be a finite transitive 
permutation group for which
the permutation character is $1+\theta$ where $\theta$ is irreducible 
over the rationals. All such groups are primitive and strongly \tont. Moreover, Dixon reduced their study to the almost simple case. The groups in (iii) of Theorem \ref{tontas} satisfy the 
$\mathbb{Q}I$-condition if and only if $q-1 = 2^f-1$ is prime  
(see \cite[Theorem 11]{Dixon05}). Hence we have

\begin{cor}\label{qigroups}
The almost simple $\mathbb{Q}I$-groups which are not $2$-transitive 
are precisely the 
groups in part {\rm (iii)} of Theorem $\ref{tontas}$ with $2^f-1$ 
prime.  
\end{cor}

The final consequence concerns a connection between Theorem \ref{pTheorem}
and {\it triple factorizations}. These are factorizations of a group $G$
as a product $ABA$ for subgroups $A,B$; such factorizations have been 
of interest since the paper of Higman and McLaughlin \cite{HM} linking
them with incidence geometries. 

It is an elementary observation that if $G$ is a transitive 
permutation group with point stabiliser $H$, and $p$ is a prime 
dividing $|H|$, then all the subdegrees of $G$ are coprime to $p$
if and only if $G$ admits the triple factorization 
\[
G = HN_G(P)H,
\]
where $P$ is a Sylow $p$-subgroup of $H$. Indeed, for $g \in G\backslash H$ 
the condition on the 
subdegrees implies that there is an $H$-conjugate $P^h$ of $P$ contained
in the two-point stabiliser $H \cap H^g$. Then $P^h,P^{hg^{-1}}$ are
contained in $H$, hence are $H$-conjugate, and so $hg^{-1}h' \in N_G(P)$
for some $h'\in H$, giving $g \in HN_G(P)H$. Hence the subdegree condition
implies the triple factorization, and the converse implication is proved by
reversing the argument. 

In view of Theorem \ref{pTheorem}, this gives the following result concerning
triple factorizations of groups of Lie type. We denote the set of Sylow $p$-subgroups of a group $G$ by $\Syl_p(G)$. 

\begin{cor}\label{triple}
Let $G$ be an almost simple group with socle $L$ of Lie type in characteristic
$p$, and let $H$ be a maximal subgroup of $G$. Assume that $p$ divides $|H|$.
Then $G = HN_G(P)H$ for $P \in Syl_p(H)$ if and only if 
$G,H$ are as in Theorem $\ref{pTheorem}${\rm (ii)-(iv)}.
\end{cor}

We note that in \cite{alavi1,alavi2} it is shown that there are many more triple factorizations $G=HAH$ with $A$ a maximal subgroup properly containing $N_G(P)$, where $H$ is a maximal subgroup of order divisible by $p$ and $P\in\Syl_p(H)$.

\section{Basic lemmas}

In this section, we provide some of the lemmas that we repeatedly use in the course of
this work. Let $G$ be a transitive permutation group on a set $\Omega$. We refer to $|\Omega|$ as the \emph{degree} of $G$ and the orbits of a point stabiliser as \emph{suborbits}.  Recall from the Introduction the result of Wielandt 
that a finite \tont\ permutation group is
primitive or a Frobenius group.  

\begin{lemma}\label{lem:pnorml}\label{lem:overgroup}\label{lem:rsubdegnottont}
Let $G$ be a finite transitive permutation group with a point stabiliser $H$.
\begin{enumerate}
\item[(i)] Suppose the degree of $G$ is divisible by $r>1$. If $G$ has a subdegree
  divisible by $r$ then $G$ is not \tont.
\item[(ii)] Suppose that $H<K<G$ and in the action of $K$ on the set of right cosets of
  $H$, $K$ has a suborbit of length $\ell$.  Then $G$ has a suborbit of length $\ell$.
\item[(iii)] Let $T$ be a normal subgroup of $G$. Let $g\in G$ and suppose
  $|T\cap H:T\cap H\cap H^g|$ is divisible by a positive integer $k$. Then
  $|H:H\cap H^g|$ is divisible by $k$.
\item[(iv)] Let $p$ be a prime such that $H$ has a nontrivial
normal $p$-subgroup $P$.
  Then $G$ has a subdegree divisible by $p$.
\end{enumerate}
\end{lemma}

\begin{proof}
(i) This is clear.

\medskip\noindent(ii) Since $K$ has a suborbit of length $\ell$, there exists $g\in K$
such that $|H:H\cap H^g|=\ell$. Since $g\in G$, it follows that $G$ has a suborbit of
length $\ell$.

\medskip\noindent(iii) Consider the group action of $G$ on the right cosets of $H$.  Since
$T\cap H$ is a normal subgroup of $H$, the orbit lengths of $T\cap H$ divide the orbit
lengths of $H$. Therefore, for all $g\in G$, we have that $|T\cap H:T\cap H\cap H^g|$
divides $|H:H\cap H^g|$.

\medskip\noindent(iv) There must be some $H$-orbit $\Delta$ upon which $P$ acts
nontrivially. Then as $P\norml H$, all orbits of $P$ on $\Delta$ have the same size. It follows that $p$
divides $|\Delta|$.
\end{proof}

Let $G$ be a transitive permutation group with point stabiliser $H$. We say that a subgroup $H_0$ of $H$ is \emph{weakly closed in $G$} if whenever $H_0^g\leqslant H$ for $g\in G$ there exists $h\in H$ such that $H_0^g=H_0^h$.

\begin{lemma}\label{lem:weaklyclosed}
  Let $G$ be a transitive permutation group with point stabiliser $H$ and let $p$ be a
  prime. Suppose there exists $T\leqslant H$ such that
\begin{enumerate}
\item[(i)] $N_{G}(T)\not\leqslant H$, and
\item[(ii)] for all $S\in\Syl_p(H)$, the group $\la T,S\ra$ contains a normal
  subgroup $H_0$ of $H$ which is weakly closed in $G$ such that
  $H=N_G(H_0)$.
\end{enumerate}
Then $G$ has a subdegree which is divisible by $p$.
\end{lemma}

\begin{proof}
  Let $g\in N_{G}(T)\backslash H$. Then $T\leqslant H\cap H^g$. Suppose that
  $|H:H^g\cap H|$ is coprime to $p$. Then there exists $S\in \Syl_p(H)$ such that
  $S\leqslant H^g\cap H$. Thus $H^g\cap H\geqslant \la S,T\ra\geqslant H_0$ and so
  $H_0^{g^{-1}}\leqslant H$. Since $H_0$ is weakly closed in $G$, it follows that there
  exists $h\in H$ such that $H_0^{g^{-1}}=H_0^h$.  Thus $hg\in N_{G}(H_0)=H$. Hence
  $g\in H$, a contradiction, and so $G$ has a subdegree divisible by $p$.
\end{proof}

The next result is commonly known as Tits' Lemma. A proof can be found 
in \cite[(1.6)]{seitz73}.

\begin{lemma}
\label{lem:Tits}
Let $L$ be a quasisimple group of Lie type in characteristic $p$, and let $H$ be a
maximal subgroup of $L$ which has index coprime to $p$ and does not contain the unique quasisimple normal subgroup of $L$.  Then $H$ is a parabolic subgroup.
\end{lemma}

\begin{lemma}\label{lem:23/4}
  Let $G$ be a group of Lie type of characteristic $p$, with $p$ odd. Let $G$ act
  primitively with point stabiliser $H$. Suppose that there exists $H_0\normleq H$, weakly
  closed in $G$ and $H_0$ a central product of quasisimple Lie type groups of characteristic $p$, excluding types $\PSL_3(q)$, 
$\PSL_d(q)$ ($d\ge 5$) and $E_6(q)$. Assume $|G:H|$ is even. Then $G$ has a subdegree divisible by $p$.
\end{lemma}

\begin{proof}
 Let $T$ be a Sylow 2-subgroup of $H$. As $|G:H|$ is even, $T$ is not a Sylow 2-subgroup of $G$ and so 
  $N_{H}(T)<N_{G}(T)$. Let $H_0$ be a central product of quasisimple Lie type groups $H_i$ of characteristic $p$, with each $H_i$ subject to the conditions in the lemma. Let $S$ be a Sylow $p$-subgroup of $H$. Then $|H:\la T,S\ra|$ is coprime to $p$ and hence, by Lemma \ref{lem:Tits}, for each $i$, the projection of $M:=H_0\cap \la T,S\ra$ to $H_i$ has image $\pi_i(M)$ equal to either $H_i$ or a parabolic subgroup of $H_i$. By our assumption on $H_i$, each parabolic subgroup of $H_i$ has even index \cite{odddegree}. Since $T$ is a Sylow 2-subgroup of $H$ and $H_0\norml H$, it follows that $M$ contains a Sylow 2-subgroup of $H_0$, and hence that $\pi_i(M)=H_i$ for each $i$. Then since $M$ contains a Sylow $p$-subgroup $S\cap H_0$ of $H_0$, we have $M=H_0$, so $H_0\leqslant \la T,S\ra$. By the maximality of $H$, $N_G(H_0)=H$ and hence the result follows from Lemma \ref{lem:weaklyclosed}.
\end{proof}

Some care is required when applying Lemma \ref{lem:23/4} in the case where $H_0$ is a
central product of classical groups, as factors isomorphic to $\PSL_d(q)$ may be hidden due
to isomorphisms in the low dimensional cases. However, by \cite[Prop. 2.9.1]{KL}, if a
classical group in characteristic $p$ is isomorphic to $\PSL_d(p^f)^k$ for some $d,p,f$ and $k$ then $d=2$ or 4, and
so Lemma \ref{lem:23/4} does apply. 

Next we note the following result of Neumann and Praeger 
\cite[Corollary 1]{restmvmt}.

\begin{lemma}\label{thm:restmvmt}
   Let $G\leqslant \Sym(\Omega)$. If there exists a $k$-set $\Gamma$ of 
$\Omega$ such that
   there is no $g\in G$ with $\Gamma^g\cap \Gamma=\varnothing$ then 
$\Gamma$ intersects
   nontrivially a $G$-orbit of length at most $k^2-k+1$.
\end{lemma}

\begin{lemma}\label{lem:mvSylp}\label{lem:mvconjclass}
  Let $G$ be a permutation group with point stabiliser $H$, let $S$ be a Sylow
  $p$-subgroup of $H$ for some prime $p$, and let $x\in H$ have order a power of $p$.
\begin{enumerate}
\item[(i)]  If $|S^G|>|S^H|^2-|S^H|+1$ then $G$ has a subdegree divisible by $p$.
\item[(ii)] If $|x^G|> |x^G\cap H|^2-|x^G\cap H|+1$ then $G$ has a
  subdegree divisible by $p$.
\end{enumerate}
\end{lemma}

\begin{proof} 
  First consider the action of $G$ on $S^G$. If
  $|S^G|>|S^H|^2-|S^H|+1$, it follows from Lemma \ref{thm:restmvmt} applied to the action of $G$ on $\Syl_p(G)$ by conjugation that there exists $g\in G$ such that
  $S^H\cap (S^H)^g=\varnothing$.  Since $S^H$ is the set of Sylow $p$-subgroups of $H$, it
  follows that $H\cap H^g$ does not contain a Sylow $p$-subgroup of $H$. Hence the subdegree $|H:H\cap H^g|$ is divisible by $p$.

Now assume $|x^G|> |x^G\cap H|^2-|x^G\cap H|+1$.  By Sylow's Theorem, $x$ is contained in some Sylow $p$-subgroup $S$ of $H$, and hence
  each Sylow $p$-subgroup of $H$ contains a conjugate of $x$. By Lemma
  \ref{thm:restmvmt} applied to the action of $G$ on itself by conjugation, there exists $g\in G$ such that $(x^G\cap H)\cap 
  (x^G\cap H)^g=\varnothing$. Hence
  $S^H\cap (S^H)^g=\varnothing$ and the argument of the previous paragraph applies.
\end{proof}

We remark that \cite[Lemma 9]{BGS} is a weaker but useful version of Lemma \ref{lem:mvconjclass}.

%
%

\section{The reduction to affine and almost simple groups}

In this section we prove Theorem \ref{tontred}, which states that a \tont\ group is either almost simple or affine. First we consider primitive groups preserving a Cartesian decomposition.

The group $\Sym(\Delta)\Wr S_k$ acts primitively in its natural product action on the set $\Omega=\Delta^k$ when $|\Delta|\ge 3$. Suppose that $G$ is a primitive subgroup of $\Sym(\Delta)\Wr S_k$ in this action. Then the image of $G$ 
under the natural homomorphism to $S_k$ is
a transitive permutation group $K$ on the set $\{1,\ldots,k\}$. Let $G_1$ be the stabiliser in $G$ of 1 in this action. Then $G_1\leqslant \Sym(\Delta)\times (\Sym(\Delta)\Wr S_{k-1})$ and so we have a natural homomorphism $\pi_1:G_1\rightarrow \Sym(\Delta)$ to the first direct factor. Let $H=\pi_1(G_1)\leqslant\Sym(\Delta)$. By \cite[(2.2)]{kovacs}, there exists $g\in\Sym(\Delta)^k\cap \ker(\pi_1)$ such that $G^g\leqslant H\Wr K$ and so we may assume that $G\leqslant H\Wr K$. Then $G_1$ induces $H$ on the set of first entries of the points of $\Omega$. Since $G$ is primitive on $\Omega$, $H$ must be primitive on $\Delta$ and we refer to $H$ as the \emph{primitive component of $G$ relative to the decomposition $\Omega=\Delta^k$}. Note that it may be possible to write $\Omega=\Lambda^r$ for some $r\neq k$ and such a decomposition would give rise to a different primitive component.

\begin{lemma}\label{lem:prodactsubdeg}
 Let $G$ be a finite primitive permutation group such that $G\leqslant H\Wr S_k$ acting in product action on $\Omega=\Delta^k$ with primitive component $H$ on $\Delta$ and $k\geq 2$. Let $K\le S_k$ be the 
image of the natural homomorphism $G \rightarrow S_k$.
\begin{enumerate}
 \item[(i)] If $H$ has a subdegree of length $\ell$ on $\Delta$ then $G$ has a subdegree of length $\ell k$ on $\Omega$.
 \item[(ii)] If $K$ has an orbit of length $\ell$ on the set of $r$-subsets of $\{1,\ldots,k\}$ then $G$ has a subdegree divisible by $\ell$.
\end{enumerate}
\end{lemma}
\begin{proof}
Let $\delta\in\Delta$ and let $\alpha=(\delta,\ldots,\delta)\in\Omega$. Then $G_{\alpha}=G\cap (H_{\delta}\Wr S_k)$. Let $\delta'\in\Delta\backslash\{\delta\}$ and let $\gamma=(\delta',\delta,\ldots,\delta)$. Then $|\gamma^{G_{\alpha}}|=|(\delta')^{H_{\delta}}|k$ and so part (i) holds. 

Let $J\subseteq\{1,\ldots,k\}$ such that $|J^K|=\ell$. Let $\gamma$ be the element of $\Omega$ such that for each  $j\in J$ we have $(\gamma)_j=\delta'$ and for $j\notin J$, $(\gamma)_j=\delta$ (where $(\gamma)_j$ denotes the $j^{th}$ coordinate). For $\gamma'\in\Omega$, let $\pi(\gamma')=\{i\mid (\gamma')_i\neq \delta\}$. Then $\{\pi(\gamma')\mid \gamma'\in \gamma^{G_{\alpha}}\}=J^K$. Moreover, the set of all $\gamma'\in\gamma^{G_{\alpha}}$ such that $\pi(\gamma')=J$ forms a block of imprimitivity for $G_{\alpha}$ on $\gamma^{G_{\alpha}}$. Hence part (ii) follows.
\end{proof}

If a primitive group $G$ on a set $\Omega$ cannot be embedded into a wreath product in product action (with $k\ge 2$) then we call $G$ a \emph{basic primitive permutation group}. One interpretation of the O'Nan-Scott Theorem for primitive groups is that a basic primitive permutation group is either affine, almost simple or of diagonal type \cite[Theorem 4.6]{peterpermgps}. By choosing an appropriate representation of $\Omega$ as a Cartesian power we may assume that the primitive component of a nonbasic primitive permutation group is one of the following three basic types:

\begin{itemize}
 \item \textit{Affine Type:}
Here $\Omega=\GF(p)^k$ and $G=N\sdp G_0$ where $N$ is the group of all translations and $G_0$ is
an irreducible subgroup of $\GL_k(p)$. If we also insist that $G$ is basic then $G_0$ is a primitive linear group, that is, does not preserve a nontrivial direct sum decomposition of the vector space.

\item \textit{Almost Simple Type:}
Here $G$ is isomorphic to a subgroup of $\Aut(T)$ containing $\Inn(T)$ for some nonabelian simple group $T$. 

\item \textit{Diagonal Type:}
Here $G$ has socle $N=T^k$ for some nonabelian simple group $T$ and $k\geq 2$. Moreover, $N_\alpha$ is a full diagonal subgroup of $N$.  
Identify $T$ with ${\rm Inn}(T)$.
Then $\Omega$ can be identified with $T^{k-1}$, and $T^k\le G\le A$
  where
$$A=\{(a_1,\ldots,a_k)\pi\mid\pi\in S_k,a_1\in\Aut(T),a_i\in\Inn(T)a_1\}$$ 
and the action of
  $A$ on $T^{k-1}$ is given by 
\begin{align*}
  (t_2,\ldots,t_k)^{(a_1,\ldots,a_k)}&=(a_1^{-1}t_2a_2,\ldots,a_1^{-1}t_ka_k)\\ 
(t_2,\ldots,t_k)^{\pi^{-1}}&=(t_{1\pi}^{-1}t_{2\pi},\ldots,t_{1\pi}^{-1}t_{k\pi})
\end{align*}
for all $(a_1,\ldots,a_k)\pi\in A$ and $(t_2,\ldots,t_k)\in\Inn(T)^{k-1}$ where $t_1=1$.
Also, either $G$ acts primitively on the set of simple direct factors of $N$, or $k=2$ and $G$ fixes each simple direct factor setwise.
\end{itemize}

 Before dealing with groups of diagonal type we need the following lemma.
\begin{lemma}\label{lem:evenlength}
  Let $T$ be a nonabelian simple group and let $p$ be a prime dividing $|T|$. Then $T$ has a conjugacy class of size divisible by $p$.
\end{lemma}

\begin{proof}
  Suppose that every conjugacy class in $T$ has size coprime to $p$. Then every element of $T$ is
  centralised by some Sylow $p$-subgroup of $T$ and so given a Sylow $p$-subgroup $S$ of $T$, $C_T(S)$ meets each conjugacy class nontrivially. Hence in the action of $T$ on the set of cosets of $C_T(S)$ every element of $T$ fixes some coset. Since every transitive group of degree at least 2 has a fixed point free element \cite[p. 173]{peterpermgps}, it follows that $C_T(S)=T$ and so $S\norml T$. Since $T$ is nonabelian simple it follows that $S=1$, a contradiction. Hence $T$ has a conjugacy class with size divisible by $p$.
\end{proof}

\begin{lemma}\label{lem:SD}
 Let $G$ be a primitive permutation group of diagonal type and let $p$ be a prime dividing $|\Omega|$. Then $G$ has a subdegree divisible by $p$.
\end{lemma}

\begin{proof}
  Let $N=T^k$ be the socle of $G$. Then $|\Omega|=|T|^{k-1}$ and so $p$ divides $|T|$. Moreover, $\Inn(T)\norml
  G_{\alpha}\leqslant \Aut(T)\times S_k$. By Lemma \ref{lem:evenlength}, we can choose $t\in T$
  such that $|t^T|$ is divisible by $p$. Then as $|(t,1,\ldots,1)^{\Inn(T)}|=|t^T|$ and $\Inn(T)\norml
  G_{\alpha}$ it follows that $G$ has a subdegree divisible by $p$. 
\end{proof}

\begin{theorem}
\label{thm:psubdeg}
Let $G$ be a primitive permutation group on a set $\Omega$ and let $p$ be a prime such that $p$ divides $|\Omega|$. Then one of the following holds:
\begin{enumerate}
 \item[(i)] $G$ has a subdegree divisible by $p$,
 \item[(ii)] $G$ is almost simple or affine,
 \item[(iii)] $G\leqslant H\Wr K$ acting in product action on $\Omega=\Delta^k$ with primitive component $H$ such that $H$ is almost simple, $H$ does not have a subdegree divisible by $p$ and for each $r\leqslant k$, each orbit of $K$ on $r$-subsets has size coprime to $p$.
\end{enumerate}
 \end{theorem}
 
\begin{proof}
 Suppose that $G$ does not have a subdegree divisible by $p$. If $G$ is a basic primitive permutation group then by Lemma \ref{lem:SD} and the preceding remarks it follows that $G$ is either almost simple or affine. If $G\leqslant H\Wr K$ in product action on $\Omega=\Delta^k$ then by our remarks above we may assume that $H$ is affine, almost simple or of diagonal type. By Lemma \ref{lem:prodactsubdeg}, $H$ does not have a subdegree divisible by $p$ and for each $r\leqslant k$, each orbit of $K$ on $r$-subsets has size coprime to $p$. Lemma \ref{lem:SD} implies that $H$ is not of diagonal type and so $H$ is almost simple or affine. Since the wreath product of an affine group with a subgroup of $S_k$ is an affine group the result follows.
\end{proof}

We can now prove Theorem \ref{tontred}.

\vspace{4mm}
\textit{Proof of Theorem \ref{tontred}:}
 Let $G$ be a finite primitive \tont\ group and let $p$ be a prime dividing $|\Omega|$.  Then $G$ does not have a subdegree divisible by $p$ and so by Theorem \ref{thm:psubdeg}, $G$ is either affine, almost simple or $G\leqslant H\Wr S_k$ acting on $\Delta^k$ with $H$ almost simple and $k\ge 2$. Suppose the latter holds. Let $\ell$ be a subdegree of $H$ on $\Delta$. Then since $G$ is \tont, Lemma \ref{lem:prodactsubdeg}(i) implies that all subdegrees of $G$ have size $\ell k$ and that all suborbits of $H$ have length $\ell$. Let $\alpha=(\delta,\ldots,\delta)\in\Delta^k$ and let $\gamma=(\delta',\delta',\delta,\ldots,\delta)$ for some $\delta'\neq \delta$. Since $|\gamma^{G_{\alpha}}|=\ell k$, \cite[Theorem 1.1]{GLPST} implies that $(H,|\Delta|)$ is one of $(\PGL_2(7),21)$, $(\PGL_2(9),45)$, $(M_{10},45)$ or
  $(\PGammaL_2(9),45)$. None of these groups is \tont, which implies that $G$ is also not {\tont}.
\qed

\section{Alternating Groups}\label{an}

In this section we prove

\begin{theorem}\label{antont}
  The only primitive {\tont} permutation groups with socle
  $A_n$ ($n\ge 5$) that are not 2-transitive are the groups $A_7$ and $S_7$ acting on the set of $2$-subsets of a $7$ element set.
\end{theorem}

Let $G$ be an almost simple group with socle $A_n\,(n\ge 5)$. Then 
either $G$ is $A_n$ or $S_n$, or $n=6$ when there are three further 
possibilities. We shall handle the latter at the end of the section, so 
suppose now that $G = A_n$ or $S_n$.

The following result on the maximal subgroups of $G$ can be found 
in \cite[\S4.6]{peterpermgps}.

\begin{proposition}\label{onanscott}
The maximal subgroups of $G$ are the intersections with $G$ of the following:

\medskip\noindent\textsc{Intransitive:} subgroups $S_k\times S_{n-k}$ with $1\leq k < n/2$.

\medskip\noindent\textsc{Imprimitive:} subgroups $S_k\Wr S_{n/k}$ in imprimitive 
action on $\{1,\ldots, n\}$.

\medskip\noindent\textsc{Primitive:} 
\begin{enumerate}
\item[(i)] $S_k\Wr S_r$ ($n=k^r,\,k\ge 5,\,r\geq 2$) in product action 
on $\{1,\ldots, n\}$;
\item[(ii)] $\AGL_d(p)$ ($n=p^d$);
\item[(iii)] Diagonal groups $T^k.(\Out(T)\times S_k)$, where $T$ is
nonabelian simple, $n=|T|^{k-1}$;
\item[(iv)] Almost simple groups acting primitively on $\{1,\ldots, n\}$.
\end{enumerate}
\end{proposition}

We now embark on the proof of Theorem \ref{antont}.
We first treat the case where the point stabiliser in $G$ is a subgroup of
intransitive type.

\begin{lemma}\label{lem:intrans}
  Let $G=A_n$ or $S_n$, where $n\ge 5$, in its natural action on $k$-sets 
  with $k<n/2$. Then
  $G$ is \tont\ if and only if either $k=1$, or $n=7$ and $k=2$.
\end{lemma} 

\begin{proof}
  Let $\Omega$ be the set of all $k$-sets of $\{1,\ldots,n\}$ and let
  $\alpha=\{1,\ldots,k\}\in\Omega$. If $k=1$ then $G$ is 2-transitive on $\Omega$ and
  hence \tont. Now assume that $k\ge 2$. Then $G_{\alpha}$
 has $k+1$ orbits on
  $\Omega$. These orbits are the sets $\Omega_i=\{\beta\in\Omega : |\beta \cap
  \alpha|=i\}$ for $0\le i \le k$. Since for each $i$, there are $\binom{k}{i}$
  subsets of $\alpha$ of size $i$, and $\binom{n-k}{k-i}$ subsets in the complement of
  $\alpha$ of size $k-i$, it follows that $$|\Omega_i|= \binom{k}{i}\binom{n-k}{k-i}$$ for
  each $i$.

  Now $|\Omega_{k-2}|/|\Omega_{k-1}|=(n-k-1)(k-1)/4$.  For $G$ to be \tont\ this fraction
  must be equal to $1$ and so $(n-k-1)(k-1)=4$. Thus $n-k-1=1,2$ or $4$ and $k-1=4,2,1$
  respectively. If $k=5$, then $n=7$, contradicting $k<n/2$. Also, if $k=3$ then $n=6$,
  another contradiction. Thus $k=2$ and $n=7$. Moreover, since $k=2$ it follows that
  $\Omega_{0}$ and $\Omega_{1}$ are the only orbits of $G_{\alpha}$ on
  $\Omega\setminus\{\alpha\}$ and so $G$ is \tont\ in this case.
\end{proof}

 Now we turn to the case where the point stabiliser is of imprimitive type.

\begin{lemma}\label{lem:imprim}
  Let $G=A_{k\ell}$ or $S_{k\ell}$, with $k,\ell\ge 2$ and $k\ell\geq 5$, act on the set of 
  partitions of a
  $k\ell$-set into $\ell$ parts of size $k$. Then $G$ is \tont\ if and 
  only if $(k,\ell)$
  is $(3,2)$. In this situation, $G$ is in fact $2$-transitive.
\end{lemma}

\begin{proof}
  Let $\Omega=\{1,\ldots,k\ell\}$ and $\mathcal{P}_1=
  \{\Delta_1,\ldots,\Delta_\ell\}$ be the
  partition of $\Omega$ with $\Delta_i=\{(i-1)k+1,\ldots,ik\}$. Then
$$H:=G_{\mathcal{P}_1}=\left(\left(\Sym(\Delta_1)\times
    \Sym(\Delta_2)\times\cdots\times\Sym(\Delta_\ell)\right)\sdp
  S_\ell\right)\cap G\cong (S_k\Wr S_{\ell})\cap G.$$

Suppose first that $k\ge 3$ and let $\mathcal{P}_2=\{\Lambda_i\}$ where 
for $i=1,\ldots,
\ell-1$, $\Lambda_i=\{(i-1)k+2,\ldots,ik+1\}$ and
$\Lambda_{\ell}=\{1,(\ell-1)k+2,\ldots,k\ell\}$.  Then $H\cap G_{\mathcal{P}_2}\cong
(S_{k-1}\Wr C_\ell)\cap G$. Hence $|\mathcal{P}_2^H| = 
k^{\ell}(\ell-1)!$.  Next let
$\mathcal{P}_3=\{\{3,4,\ldots,k+2\},\{1,2,k+3,\ldots,2k\},\Delta_3,\ldots,
\Delta_{\ell}\}$.  Thus $H\cap G_{\mathcal{P}_3}$ is isomorphic to
$$(J \times(S_k\wr S_{\ell-2}))\cap G,$$
where $J$ is the stabiliser in $S_k\Wr S_2$ of the partition
$$\{1,2\},\{k+1,k+2\},\{3,\ldots,k\},\{k+3,\ldots, 2k\}$$ of $\{1,2,\ldots,2k\}$. If
$k-2=2$, then $J\cong S_2\wr D_8$, otherwise, $J\cong (S_2\times S_{k-2})\wr S_2$.  Hence
$|\mathcal{P}_3^H|= 9\ell(\ell-1)/2$ if $k=4$ or $k^2(k-1)^2\ell(\ell-1)/8$ otherwise.  One
can check that for $\ell\ge 4$, this is less than $|\mathcal{P}_2^H|$ and so $G$ is not
\tont\ in these cases. For $\ell=2,3$, equality only holds for $\ell=2$ and $k=3$.  In
this case $G$ is 2-transitive.

Suppose now that $k=2$ and note that $\ell\geq 3$. Let
$\mathcal{P}_4=\{\{1,4\},\{2,3\},\Delta_3,\ldots,\Delta_{\ell}\}$. Then $H\cap
G_{\mathcal{P}_4}=(D_8 \times S_2\Wr S_{\ell-2}) \cap G$ and so
$|\mathcal{P}_4^H|=\ell(\ell-1)/2$.  Now let
$\mathcal{P}_5=\{\{2,3\},\{4,5\},\{6,1\},\Delta_4,\ldots,\Delta_{\ell}\}$.  Then $H\cap
G_{\mathcal{P}_5}=(S_3 \times (S_2\Wr S_{\ell-3}))\cap G$.  Hence
$|\mathcal{P}_5^H|=4\ell(\ell-1)(\ell-2)/3$ and so $G$ is not \tont.
\end{proof}

\vspace{3mm}
Given $g\in S_n$, we define the \emph{support of $g$} as $\supp(g)=\{i\in\{1,\ldots,n\}\mid 
i^g\neq i\}$. The \textit{minimal degree} of a permutation group $H$ is 
the minimal size of the support of a nontrivial
element of $H$. 

\begin{lemma}\label{lem:support}
  Let $a,b\in S_n$. Then $|\supp([a,b])|\le 2|\supp(a)|$.
\end{lemma}

\begin{proof}
  Let $i\in\{1,\ldots,n\}$. If neither $i$ nor $i^{b^{-1}}$ belong to 
  $\supp(a)$, then
  $i^{[a,b]}=i^{b^{-1}ab}=i^{b^{-1}b}=i$, that is, $[a,b]$ fixes $i$. The result follows.
\end{proof}

\begin{lemma}\label{lem:supportcomp}
 Let $G=A_n$ or $S_n$ and let $H$ be a primitive subgroup of $G$ not 
 containing $A_n$.
\begin{enumerate}
\item[(i)] If $n\ge 26$, then the minimal degree of $H$ is at least $11$.
\item[(ii)] If $n<26$, then the transitive action of $G$ on the right
  cosets of $H$ is either $2$-transitive or not \tont.
\end{enumerate}  
\end{lemma}

\begin{proof} Part (i) follows from classical results on primitive groups
with small minimal degree (see \cite[\S 15]{Wielandt}). Part (ii) was verified
using the primitive groups library in \textsc{Gap} \cite{gap}. \end{proof}

\vspace{3mm}
The remaining case in the proof of Theorem \ref{antont} is that in 
which a point stabiliser in $G$ is primitive in the natural action of degree $n$.

\begin{lemma}\label{lem:Anprim}
  Let $G=A_n$ or $S_n$ and let $H$ be a primitive subgroup of $G$ such that $H$ does not 
  contain $A_n$ and $|H|$ is even. If the transitive action of $G$ on the set of right
  cosets of $H$ is \tont, then it is $2$-transitive.
\end{lemma}

\begin{proof}
    By  Lemma \ref{lem:supportcomp}, we can assume that $n\geq
    26$ and the minimal degree of $H$ is at least 11. Let
    $x=(1\,2\,3\,4\,5)\in A_n$. Since $H$ is primitive, and the minimal degree is at 
least 11, we have $x\notin H$. Let 
    $g\in H\cap H^x$. Then
    $g,g^{x^{-1}}\in H$. However, by Lemma \ref{lem:support}, 
    $[g,x]=1$, that is, $x$
    centralises $g$. Hence $H\cap H^x=C_H(x)$.

  Next let $y=(1\,2)(3\,4)\in A_n$ and let $g\in H\cap H^y$. Arguing as in the previous paragraph, $H\cap
  H^y=C_H(y)$.  Now $|C_H(x)|=|H_{12345}|
  |C_H(x):H_{12345}|$ and
  $|C_H(x):H_{12345}|$ is $1$ or $5$. Furthermore, the size of 
  $C_H(y)$ is $|C_H(y):H_{1234}||5^{H_{1234}}||H_{12345}|$.  Since
  $|H|$ is even, it contains an element $z$ of order two with at least six 2-cycles (for the minimal degree is at least 11). Replacing $H$ by a conjugate if necessary, we may assume that $(1\,2)$ and $(3\,4)$ are 2-cycles of $z$, so $z\in C_H(y)$. Hence $|C_H(y):H_{1234}|\in\{2,4,8\}$. Thus $|C_H(y)|\neq |C_H(x)|$ and
  so $|H:H\cap H^x|\neq |H:H\cap H^y|$.
\end{proof}

\begin{lemma}\label{lem:AGL}
  Let $p\ge 5$ be an odd prime and let $H=\AGL_1(p)\cap A_p$. Then the action of $A_p$ on
  the set of right cosets of $H$ in $A_p$ is \tont\ if and only if $p=5$ and in that case the action is $2$-transitive.
\end{lemma}

\begin{proof}
  If $p=5$ then the action is the 2-transitive action of $A_5$ of degree 6.  If $p=7$
  a quick calculation shows that there are suborbits of lengths 7 and 21. Hence we may
  assume that $p\ge 11$.  Let $x=(1\,2\,3)\in A_p$. Since $H$ has no nontrivial elements of
  support at most 6 it follows from Lemma \ref{lem:support} that $H\cap H^x=C_H(x)$. Moreover,
  $|C_H(x)|=|C_H(x):H_{1,2,3}||H_{1,2,3}|$. Since no nontrivial element of $H$
  fixes more than one point it follows that $|C_H(x)|=1$ or 3. Thus $H$ has an orbit
  of length $p(p-1)/2$ or $p(p-1)/6$ on $\Omega$.  Let $g\in H$ have order $(p-1)/2$ such that $h$ has  two cycles of length $(p-1)/2$. Since $(p-1)/2\geq 5$, there exists $y\in A_p$ which normalises, 
but does not centralise, $\la g\ra$. As $H$ is self-normalising in $A_p$ and $\la g\ra$ is maximal in 
$H$ it follows that $H\cap H^y=\la g\ra$. Thus $H$ has an orbit of length $p$ on the set of right 
cosets of $H$ and so $A_p$ is not \tont.
\end{proof}

We note that an alternative approach to the case where $H$ acts primitively on $\{1,\ldots,n\}$ would be to use the result of \cite{BGS} that the action has a regular suborbit when $n>12$ and hence is not \tont.

\medskip
\emph{Proof of Theorem \ref{antont}}:
  Let $G=A_n$ or $S_n$ be \tont\ but not 2-transitive on the set of cosets of a 
  maximal subgroup $H$. Then $H$
  is primitive, imprimitive or intransitive of degree $n$. If $H$ is primitive, then Lemma
  \ref{lem:Anprim} implies that $|H|$ is odd. 
  By Proposition \ref{onanscott} this forces $n=p$, $G=A_p$ and 
  $H=\AGL_1(p)\cap A_p$ for
  some odd prime $p$. By Lemma \ref{lem:AGL}, this action is not \tont, unless $p=5$ and in that case it is 2-transitive.  Lemma  \ref{lem:imprim} shows that $H$ is not imprimitive. Hence $H$ is 
  intransitive,  and Lemma \ref{lem:intrans} implies that $G$ is $A_7$ or $S_7$  acting on 21 points.

It remains to handle the extra possibilities for $G$ when $n=6$. 
Apart from $A_6$ and $S_6$, the groups with socle $A_6$ are $M_{10}$, 
$\PGL_2(9)$ and ${\rm Aut}(A_6)$. These are easily checked using the Atlas \cite{atlas}. \qed

%
%

\section{Classical Groups: Preliminaries}
\label{sec:prelim}

In this section and the next, we prove

\begin{theorem}\label{thm:classgroups}
  Let $G$ be a finite almost simple primitive permutation group on $\Omega$ with 
  socle a classical 
  simple group $L$ of characteristic $p$.  Let $H$ be the stabiliser in $G$ of 
  a point in $\Omega$.
\begin{enumerate}
\item[(A)] If $G$ has no subdegrees divisible by $p$, then either $|H|$
  is not divisible by $p$ or one of the following holds:
\begin{enumerate}
\item[(i)] $L=\PSL_2(q)$, $q$ is even, $H=N_G(D_{2(q+1)})$ and either 
$|G:L|$ is odd or $q=4$.
\item[(ii)]  $G=\Sp_{2d}(2)$ and $H=O^{\pm}_{2d}(2)$ with $d\geq 3$.
\item[(iii)] $L$ and $H$ are given in Table \ref{tab:classgp}.
\end{enumerate}

\begin{table}[ht]
\caption{}\label{tab:classgp}
 \begin{tabular}{cccll}
\hline
  $L$  &$|\Omega|$ & $H$ & \emph{Conditions} & \emph{Subdegrees} \\
\hline
 $\Sp_4(2)'$ &$6$ & $N_G(A_5)$ & &$1, 5$ \\
 $\Sp_4(2)'$ &$10$  & $N_G(C_3^2)$ & & $1, 9$\\
 $\PSU_3(5)$ & $50$ & $N_G(A_7)$ & & $1, 7, 42$ \\
 $\PSL_2(9)$ & $6$ & $N_G(A_5)$ & & $1, 5$\\
 $\PSL_4(2)$ & $8$ & $N_G(A_7)$ & & $1, 7$ \\
 $\PSp_4(3)$ & $27$ & $N_G(2^4.A_5)$ & & $1, 10, 16$\\
 $\PSL_3(2)$ & $8$ & $C_7\rtimes C_6$ & $G=L.2$ & $1,7$ \\
\hline
\end{tabular}

\end{table}

\vspace{2mm}
\item[(B)] If $G$ is \tont, and not $2$-transitive, then \emph{(A)(i)} holds with $q=2^f\geq 8$ and either $G=L$ or $|G:L|=f$ is prime.
\end{enumerate}
\end{theorem}

In the proof of this theorem we shall need some preliminary information 
on subgroups and conjugacy classes in classical groups.

\begin{proposition}
\label{prn:sylpgen}
Let $G=\SL_d(q), \SU_d(q), \Sp_d(q)$ or $\Omega^{\epsilon}_d(q)$, 
with $d\ge 3$  in the last three cases and $q=p^f$ for some
prime $p$. Let $V$ be the natural module for $G$. 
Then there exists a cyclic subgroup
$T$ of $G$ with order given in Table \ref{tab:T} such that 
$\la S,T\ra=G$ for every Sylow
$p$-subgroup $S$ of $G$. Moreover, $T$ acts irreducibly on an $\ell$-dimensional subspace $W$ of $V$ with $\ell$ given in Table \ref{tab:T} and trivially on a complement of $W$. When $G=\SU_d(q), \Sp_d(q)$ or $\Omega^{\epsilon}_d(q)$ we can take $W$ to be nondegenerate.
\end{proposition}

\begin{table}[ht]
\caption{}
\label{tab:T}
\begin{center}
\begin{tabular}{ccc}
\hline
$G$  & $|T|$ & $\ell$ \\
\hline
$\SL_d(q)$ &  $(q^d-1)/(q-1)$ & $d$\\
$\SU_d(q)$, $d$ odd & $(q^d+1)/(q+1)$ & $d$ \\
$\SU_d(q)$, $d\ge 4$ even & $(q^{d-1}+1)/(q+1)$ & $d-1$\\
$\Sp_d(q)$ &  $q^{d/2}+1$ & $d$ \\
$\Omega_d(q)$, $d\ge 3$ odd & $(q^{(d-1)/2}+1)/(2,q+1)$ & $d-1$ \\
$\Omega^-_d(q)$, $d\ge 4$ & $(q^{d/2}+1)/(2,q+1)$ & $d$\\
$\Omega^+_d(q)$, $d\ge 4$ & $(q^{d/2-1}+1)/(2,q+1)$ &$d-2$ \\
\hline
\end{tabular}
\end{center}
\end{table}

\begin{proof}
  Suppose first that $G=\SL_d(q),\Sp_d(q), \Omega^-_d(q)$ or $\SU_d(q)$, with the additional assumption that $d$ is odd when $G=\SU_d(q)$. Let
  $T$ be the intersection of $G$ and a Singer cycle of $\GL(V)$. Then by \cite[Table
    1]{bereczky00}, $|T|$ is as given in Table \ref{tab:T} and $T$ is irreducible. Let
  $S\in\Syl_p(G)$.  By Tits' Lemma \ref{lem:Tits}, 
  all overgroups of $S$ other than $G$ are
  contained in parabolic subgroups and hence fix a subspace. So it follows that $\la
  T,S\ra=G$.

  Suppose next that $G=\SU_d(q)$ with $d\ge 4$ even. Let $U$ be a nondegenerate hyperplane in $V$ 
  and let $x$ be a nonsingular vector such that $U=\la x\ra^{\perp}$. Since
  $\SU_{d-1}(q)\leqslant G_U$ and $d-1$ is odd, it follows that $G$ has a cyclic subgroup
  $T$ of order $(q^{d-1}+1)/(q+1)$ which acts irreducibly on $U$ and fixes $x$. Let $S\in
  \Syl_p(G)$. If $G \ne \la S,T\ra$ then $\la S,T\ra$ fixes a totally isotropic
  subspace $W$ (by Tits Lemma \ref{lem:Tits}).  Since $T$ acts irreducibly on $U$, we have
  $W\cap U=\{0\}$ and so $W=\la u+\lambda x\ra$ for some nonzero vector $u\in U$ and
  $\lambda\in\GF(q)\backslash\{0\}$. Since $T$ fixes $x$, it follows that $T$ fixes $u$,
  contradicting $T$ acting irreducibly on $U$. Hence $G=\la S,T\ra$.

  Next suppose that $G=\Omega_d(q)$ with $d$ and $q$ odd.  Let $U$ be a
  nondegenerate hyperplane upon which the restriction of the quadratic form is elliptic (i.e. of minus type)
  and let $x$ be a nonsingular vector such that $U=\la x\ra^{\perp}$. Since
  $\Omega^-_{d-1}(q)\leqslant G_U$ it follows that $G$ has a cyclic subgroup $T$ of order
  $(q^{(d-1)/2}+1)/(2,q+1)$ which acts irreducibly on $U$ and fixes $x$. Let $S\in
  \Syl_p(G)$. Then by the same argument in the previous paragraph, $\la T,S\ra=G$.

  Finally, suppose that $G=\Omega^+_d(q)$, with $d\ge 4$. Let $U$ be a nondegenerate
  subspace of codimension 2 upon which the restriction of the quadratic form is elliptic.
  Then $U^{\perp}$ does not contain any singular nonzero vectors. Since
  $\Omega^-_{d-2}(q)\leqslant G_U$ we have that $G$ contains a cyclic subgroup $T$ of order
  $(q^{d/2-1}+1)/(2,q+1)$ which acts irreducibly on $U$ and trivially on $U^{\perp}.$ Let
  $S\in \Syl_p(G)$. If $G \ne \la S,T\ra$ then $\la S,T\ra$ fixes a 
  totally singular
  subspace $W$ (again by Tits' Lemma \ref{lem:Tits}). Since $T$ acts irreducibly on $U$, it
  follows that $W\cap U=\{0\}$. Also $W\cap U^{\perp}=\{0\}$ as $U^{\perp}$ does not
  contain any nonzero singular vectors. Hence either $W=\la v+x\ra$ or $\la v+x,w+y\ra$
  where $v$ and $w$ are linearly independent vectors of $U$ and $\la
  x,y\ra=U^{\perp}$. Since $T$ fixes $x$, it follows that $T$ fixes $v$, contradicting $T$
  being irreducible on $U$. Hence $G=\la S,T\ra$.
\end{proof}

\begin{lemma}
\label{lem:nondegenp}
Let $G$ be a classical group with normal subgroup $X$ as in Proposition $\ref{prn:sylpgen}$ such that $X\neq \SL_d(q)$, and let $U$ be a proper nondegenerate subspace of the natural $G$-module $V$. Then $p$ divides $|G:G_U|$.
\end{lemma}

\begin{proof}
 Since $G\cap \GL(V)$ is transitive on $U^G$, we may assume that $G\leqslant \GL(V)$.
As $U$ is nondegenerate, $G$ fixes the decomposition of $V$ given by $U\perp U^\perp$
with $U^\perp$ also nondegenerate. By \cite[Lemma 4.1.1]{KL}, the groups induced by $G_U$ on $U$ and $U^\perp$ contain Sylow $p$-subgroups of the isometry groups of these spaces. The $p$-parts of $|G_U^U|$, $|G_U^{U^\perp}|$ and $|G|$ can be read off from \cite[Table 2.1C]{KL} and it is easily computed that $|G_U^U|_p|G_U^{U^\perp}|_p<|G|_p$.
\end{proof}

\begin{lemma}
\label{lem:intnondegen}
Let $G$ and $V$ be as in Lemma $\ref{lem:nondegenp}$ of unitary, symplectic or orthogonal type, and 
let $U$ be a nondegenerate proper subspace of $V$ of 
dimension $m\ge 2$. Suppose also that
\begin{enumerate}
\item if $G$ is symplectic then $m\ge 4$, 
\item if $G$ is orthogonal then $m\ge 3$ and $d-m\ge 2$.
\end{enumerate}
Then there exists $W\in U^G$ such that $0\neq W\cap U<U$ and $W\cap U$ is
nondegenerate. Moreover, $p$ divides $|G_U:G_{U,W}|$.
\end{lemma}

\begin{proof}
  Suppose first that $G$ is not symplectic and, if $m,q$ are both even, suppose also that $G$ is
  not orthogonal. Then there exists a nondegenerate proper subspace $Y$ of $U$ with
  codimension 1. By \cite[Lemma 4.1.1 and Proposition 2.10.6]{KL}, $G_Y$ is irreducible on
  $Y^{\perp}$. (Note that if $G$ is orthogonal then $\dim(Y^{\perp})\ge 3$.) Thus $G_Y$ is
  not contained in $G_U$ and so there exists $g\in G_Y$ such that $Y\leqslant U^g\neq
  U$. Hence $Y=U\cap U^g$. Let $W=U^g$. By \cite[Lemma 4.1.1]{KL}, $G_U$ induces a
  nontrivial classical group on $U$ and hence by Lemma \ref{lem:nondegenp},
  $|G_U:G_{U,W}|$ is divisible by $p$. (Note that if $G$ is orthogonal then $m\ge 3$.)

Suppose now that $G$ is symplectic, or $G$ is orthogonal with $m$ and $q$ both even. By hypothesis this implies $m\geq 4$.  Let $Y_1=\la e_1,f_1\ra\leqslant U$ be a hyperbolic pair and
let $Y_2=\la e_d,f_d\ra$ be a hyperbolic pair in $U^\perp$. Then by Witt's Lemma and
\cite[Lemma 4.1.1]{KL}, there exists $g\in G$ interchanging $Y_1$ and $Y_2$ while fixing
setwise $Y_1^{\perp}\cap U$. Then letting $W=U^g$ we have $U\cap W=Y_1^{\perp}\cap U$, a
nondegenerate subspace of $V$.  Hence by Lemma \ref{lem:nondegenp}, $|G_U:G_{U,W}|$ is
divisible by $p$.
\end{proof}

\vspace{3mm}
Next we present a useful lemma concerning the subdegrees of groups of 
Lie type in parabolic actions.

\begin{lemma}\label{unique}
Let $G$ be an almost simple group with socle $L$ of Lie type in 
characteristic $p$. Let $P$ be a maximal parabolic subgroup of $G$. 
Exclude the following cases:
\[
\begin{array}{l}
L = \PSL_d(q) \text{ with } G\leqslant \PGammaL_d(q) \\
L = \POmega^+_{2m}(q), m \hbox{ odd}, P = P_{m-1} \hbox{ or }P_m \\
L = E_6(q), P=P_i\quad (i=1,3,5,6).
\end{array}
\]
Then in its action on the cosets of $P$, the group
$G$ has a unique nontrivial subdegree that is a power of $p$.
\end{lemma}

\begin{proof} This follows from \cite[3.9]{LSSclass}: except in the excluded 
cases, the parabolic $P^-$ opposite to $P$ is $G$-conjugate to $P$ and 
the required
suborbit is the $P$-orbit containing $P^-$. 
\end{proof}

%
%

Next we present some information on conjugacy class sizes in classical groups.

\begin{definition}
 \label{def:nu}
Let $x\in\GL_d(q)$ and $V$ be a $d$-dimensional vector space over $\GF(q)$. Let $K$ be the
algebraic closure of $\GF(q)$ and $\overline{V}=V\otimes K$, a $d$-dimensional vector
space over $K$. Then $x$ acts naturally on $\overline{V}$ and we define $\nu(x)$ to be the
codimension of the largest eigenspace of $x$ on $\overline{V}$. For $x\in\PGL_d(q)$, we
define $\nu(x)$ to be $\nu(\hat{x})$, where $\hat{x}$ is a preimage of $x$ in $\GL_d(q)$. 
\end{definition}

We denote the lower triangular matrix $\begin{pmatrix} 1&0\\1&1\end{pmatrix}$ by $J_2$ and we use $J_2^s$ to denote the block diagonal matrix with $s$ copies of $J_2$.

When $q$ is even and $G$ is a symplectic or orthogonal group the conjugacy classes of involutions are described in \cite{aschseitz}. When $G=\Sp_d(q)$, for each odd positive integer $s\leq d/2$, there is one class of
involutions with $\nu(x)=s$, denoted by $b_s$, while for each even positive integer $s\leq d/2$ there are two classes of involutions with $\nu(x)=s$, denoted $a_s$ and $c_s$. By
\cite[(8.10)]{aschseitz}, the group $\SO^{\epsilon}_d(q)$ meets each of the $\Sp_d(q)$-conjugacy classes of involutions except $\SO^-_d(q)$ contains no involutions of type $a_{d/2}$ for $d\equiv 0\pmod 4$. Only involutions of type $a_s$ or $c_s$ lie in $\Omega^{\epsilon}_d(q)$ and the involutions of type $b_s$ lie in 
$\SO^{\epsilon}_d(q)\backslash\Omega^{\epsilon}_d(q)$. Moreover, two involutions of $\SO^{\epsilon}_d(q)$ that are conjugate in $\Sp_d(q)$ are conjugate under an element of $\Omega^{\epsilon}_d(q)$, except the $a_{d/2}$-class in $\SO^+_d(q)$ which splits into two $\Omega^+_d(q)$-classes denoted $a_{d/2}$ and $a_{d/2}'$ that are fused by $\SO^+_d(q)$.

Combining \cite[Lemma 3.20 and Proposition 3.22]{burness} we obtain the following bounds
on the lengths of conjugacy classes. We  denote $\PSL_n(q)$ by $\PSL^+_n(q)$ and $\PSU_n(q)$ by $\PSL^-_n(q)$.

\begin{proposition}
\label{prn:burnessbounds}
Let $G=\PSL^\epsilon_d(q),\PSp_d(q)$ or $\POmega^{\epsilon}_d(q)$ with $q=p^f$ for some
prime $p$. Let $x\in G$ have order $p$ and let $s=\nu(x)$.  Then $|x^G|> f_i(d,s,q)$ with
$i=1+\delta_{2,p}$ (where $\delta_{2,p}$ is the Kronecker delta function) and $f_i(d,s,q)$ as given in Tables $\ref{tab:conjclassboundsodd}$ and $\ref{tab:conjclassboundseven}$.
\end{proposition}

\begin{table}[ht]
\begin{center}
\caption{Bounds on unipotent conjugacy classes for $p$ odd}
\label{tab:conjclassboundsodd}
\begin{tabular}{ll}
\hline
$G$  & $f_1(d,s,q)$ \\
\hline 
$\PSL^\epsilon_d(q)$ & $\frac{q}{2(q-\epsilon)(q+1)}\max\{q^{2s(d-s)},q^{ds}\}$\\
$\PSp_d(q)$ & $\frac{q}{4(q+1)}\max\{q^{s(d-s)},q^{ds/2}\}$\\
$\POmega^{\pm}_d(q)$ & $\frac{q}{8(q+1)}\max\{q^{s(d-s-1)},q^{d(s-1)/2}\}$\\
$\POmega_d(q)$ & $\frac{1}{4}\max\{q^{s(d-s-1)},q^{d(s-1)/2}\}$\\
\hline
\end{tabular}
\end{center}
\end{table}

\begin{table}[ht]
\begin{center}
\caption{Bounds on unipotent conjugacy classes for $p$ even}
\label{tab:conjclassboundseven}
\begin{tabular}{llll}
\hline
$G$  & Conditions & $x$ & $f_2(d,s,q)$  \\
\hline 
$\PSL^{\epsilon}_d(q)$ & & $[J_2^s,I_{d-2s}]$ & $\frac{q}{2(q-\epsilon)(q+1)}q^{2s(d-s)}$\\
$\PSp_d(q)$            & & $a_s$             &$\frac{1}{2}q^{s(d-s)}$\\
                       & & $b_s,c_s$         &$\frac{1}{2}q^{s(d-s+1)}$ \\
$\POmega^{\pm}_d(q)$&$(s,\epsilon)\neq(d/2,+)$& $a_s$& $\frac{1}{4}q^{s(d-s-1)}$\\
                    &      & $c_s$                    & $\frac{1}{4}q^{s(d-s)}$\\
$\POmega^+_d(q)$ & $s=d/2$ & $a_{d/2},a'_{d/2}$ & $\frac{1}{4}q^{d(d-2)/4}$\\
                 &         & $c_{d/2}$          & $\frac{1}{4}q^{d^2/4}$\\
\hline
\end{tabular}
\end{center}
\end{table}

An outer automorphism of a finite simple group of Lie type can be written as the product
of an inner automorphism, a diagonal automorphism, a field automorphism and a graph automorphism. We follow the
conventions of \cite[Definition 2.5.13]{GLS} as to the definition of a field, graph or
graph-field automorphism. In Table \ref{tab:pslouter}, these automorphisms are referred to
being of type $f$, $g$ and $gf$ respectively.

\begin{lemma}
\label{lem:pslouter}
Let $L=\PSL^{\epsilon}_d(q),\PSp_d(q)$ or $\POmega^{\epsilon}_d(q)$ where $q=p^f$ for some
prime $p$ and let $x\in\Aut(L)\backslash\PGL(V)$ of prime order $r$. Then $|x^L|>h(d,r,q)$
where $h(d,r,q)$ is as given by Table $\ref{tab:pslouter}$.
\end{lemma}
\begin{proof}
This is from \cite[Lemma 3.48]{burness}.
\end{proof}

\begin{table}[ht]
\begin{center}
\caption{Bounds on conjugacy classes of outer automorphisms}
\label{tab:pslouter}
\begin{tabular}{llll}
\hline
$L$ & Type & Conditions & $h(d,r,q)$ \\
\hline
$\PSL^{\epsilon}_d(q)$ & $f$ & $q=q_0^r$, $r>2$ if $\epsilon=-$  &$\frac{1}{2}\left(\frac{q}{q+1}\right)^{(1-\epsilon)/2}q^{(d^2-1)(1-\frac{1}{r})-1}$\\
&$g$ & $r=2$, $d$ odd & $\frac{1}{2}\left(\frac{q}{q+1}\right)^{(1-\epsilon)/2}q^{\frac{1}{2}(d^2+d-4)}$\\
&$g$ & $r=2$, $d>2$ even & $\frac{1}{2}\left(\frac{q}{q+1}\right)^{(1-\epsilon)/2}q^{\frac{1}{2}(d^2-d-4)}$\\
&$gf$ & $(r,q,\epsilon)=(2,q_0^2,+)$, $d>2$& $\frac{1}{2}q^{\frac{1}{2}(d^2-3)}$\\
\hline
$\PSp_d(q)$  & $f$ & $q=q_0^r$ &  $\frac{1}{4}q^{d(d+1)(1-\frac{1}{r})/2}$\\
             & $gf$ & $(d,r,p)=(4,2,2)$, $f$ odd & $q^5$\\
\hline
$\POmega^{\epsilon}_d(q)$, $d$ even & $f$ & $q=q_0^r$&$\frac{1}{4}q^{d(d-1)(1-\frac{1}{r})/2}$\\
                       &$gf$ &$(r,q,\epsilon)=(2,q_0^2,+)$&$\frac{1}{4}q^{d(d-1)/4}$\\
            &$gf$ & $(d,r,q,\epsilon)=(8,3,q_0^3,+)$&$\frac{1}{4}q^{56/3}$\\
       &$g$& $(d,r,\epsilon)=(8,3,+)$& $\frac{1}{8}q^{14}$\\
\hline
$\POmega_d(q)$, $dq$ odd & $f$ &$q=q_0^r$ & $\frac{1}{4}q^{d(d-1)(1-\frac{1}{r})/2}$\\
\hline
\end{tabular}
\end{center}
\end{table}

\begin{lemma}
\label{lem:upperbounds}
Let $L$ be one of the groups in the first column of Table $\ref{tab:trans}$. Then the number of transvections in $L$ is given by the second column of Table $\ref{tab:trans}$.
\end{lemma}

\begin{table}[ht]
\begin{center}
\caption{Transvections}
\label{tab:trans}
\begin{tabular}{ll}
\hline
$L$ & Number of transvections \\
\hline
$\PSL_d(q)$ & $(q^d-1)(q^{d-1}-1)/(q-1)$ \\
$\PSU_d(q)$ & $(q^d-(-1)^d)(q^{d-1}-(-1)^{d-1})/(q+1)$\\
$\PSp_d(q)$ & $q^d-1$\\
\hline
\end{tabular}
\end{center}
\end{table}

\section{Classical groups: proof of Theorem \ref{thm:classgroups}}

In this section we prove Theorem \ref{thm:classgroups}. 
Throughout, $G$ is an
almost simple group with socle $L=\PSL_d(q)$, $\PSU_d(q)$, $\PSp_d(q)$ or
$\POmega^{\epsilon}_d(q)$ acting primitively on a set $\Omega$ with $H=G_{\alpha}$ for
some $\alpha\in\Omega$. If $L=\PSU_d(q)$ we assume that $d\geq 3$ and $(d,q)\neq
(3,2)$. For $L=\PSp_d(q)$ we have $d\geq 4$ and $(d,q)\neq(4,2)$. Finally, if $L$ is an
orthogonal group we assume that $d\geq 7$.

We denote the natural module of $G$ by $V$ and we let $\{v_1,\ldots,v_d\}$ be a basis for $V$
over $\GF(q)$ (over $\GF(q^2)$ when $L=\PSU_d(q)$). For classical groups with socle other
than $\PSL_d(q)$ it is often convenient to use bases specific to the sesquilinear form $B$
and/or quadratic form $Q$, preserved by the group as follows:
\begin{itemize}
\item When $L=\PSp_d(q)$ we call a basis $\{e_1,\ldots,e_{d/2},f_1,\ldots,f_{d/2}\}$ such
  that $B(e_i,e_i)=B(f_i,f_i)=0$ for all $i$ and $B(e_i,f_j)=\delta_{i,j}$ for all $i,j$,
  a \emph{symplectic basis}.
\item When $L=\PSU_d(q)$ with $d$ even we call a basis
  $\{e_1,\ldots,e_{d/2},f_1,\ldots,f_{d/2}\}$ such that $B(e_i,e_i)=B(f_i,f_i)=0$ for all
  $i$ and $B(e_i,f_j)=\delta_{i,j}$ for all $i,j$, a \emph{unitary basis}.
\item When $L=\PSU_d(q)$ with $d$  odd, we call a basis
  $\{e_1,\ldots,e_{(d-1)/2},f_1,\ldots,f_{(d-1)/2},x\}$ such that
  $B(e_i,e_i)=B(f_i,f_i)=B(e_i,x)=B(f_i,x)=0$ for all $i$, $B(e_i,f_j)=\delta_{i,j}$ and
  $B(x,x)=1$ a \emph{unitary basis}.
\item When $L=\POmega^+_d(q)$, we call a basis $\{e_1,\ldots,e_{d/2},f_1,\ldots,f_{d/2}\}$
  where $Q(e_i)=Q(f_i)=0$ and $B(e_i,f_j)=\delta_{i,j}$ for all $i,j$ a \emph{hyperbolic
    basis}.
\item When $L=\POmega^-_d(q)$, we call a basis
  $\{e_1,\ldots,e_{d/2-1},f_1,\ldots,f_{d/2-1},x,y\}$ where $Q(e_i)=Q(f_i)=0$,
  $B(e_i,f_j)=\delta_{i,j}$, $ B(e_i,x)=B(e_i,y)=B(f_i,x)=B(f_i,y)=0$ for all $i,j$,
  $Q(x)=1$, $B(x,y)=1$ and $Q(y)=\zeta$ where ${\bf x}^2+{\bf x}+\zeta$ is irreducible
  over $\GF(q)$, an \emph{elliptic basis}.
\item When $L=\POmega_d(q)$ with $d$ odd, we call a basis
  $\{e_1,\ldots,e_{(d-1)/2},f_1,\ldots,f_{(d-1)/2},x\}$ where $Q(e_i)=Q(f_i)=0$,
  $B(e_i,f_j)=\delta_{i,j}$, $ B(e_i,x)=B(e_i,y)=B(f_i,x)=0$ for all $i,j$ and $Q(x)\neq
  0$, a \emph{parabolic basis}.
\end{itemize}

The maximal subgroups of the classical groups are described by 
Aschbacher's Theorem \cite{asch}. They fall into eight families 
$\mathcal{C}_i$ ($1\le i \le 8$) of ``geometric" subgroups, together 
with a further class $\mathcal{C}_9$ consisting of almost simple groups
in absolutely irreducible representations satisfying certain extra 
conditions (see \cite[p3]{KL}). We shall deal with each case $H \in \mathcal{C}_i$ in 
a separate subsection below. There are three further cases where 
extra possibilities for $H$ have to be considered -- those in which 
$G$ contains a graph automorphism of $L = \PSL_d(q)$ or $\Sp_4(q)$ ($q$ even),
or a triality automorphism of $L = \POmega_8^+(q)$. These are 
considered at the end of the section.

Throughout, we shall use the detailed descriptions of the subgroups in 
the families $\mathcal{C}_i$ which can be found in \cite[Chapter 4]{KL}.

\subsection{Aschbacher class $\mathcal{C}_1$:} 
\label{sec:C1}

Suppose $H \in \mathcal{C}_1$. Here $H$ is the stabiliser of some subspace $U$ of dimension $m$ with $1\le m\leq d/2$.
If $L=\PSL_d(q)$ or $U$ is totally singular, then $H$ has a nontrivial normal $p$-subgroup and so by
Lemma \ref{lem:pnorml}(iv), $G$ has a $p$-subdegree (i.e. a subdegree 
divisible by $p$). If $G$ is an orthogonal group and $U$
is a nonsingular 1-space, then the subdegrees of $G$ are given in \cite{bhs} or \cite[pp.331,332]{Saxl02} from which we see that there is always a $p$-subdegree.  This leaves
us to deal with the case where $L\neq\PSL_d(q)$ and $U$ is a nondegenerate subspace of
dimension $m$ (with $m\geq 2$ if $L\neq\PSU_d(q)$). Then $H$ also stabilises $U^{\perp}$
which has dimension $d-m$. Since $G$ acts primitively on $\Omega$, it follows that $U$ and $U^\perp$ are
not similar. Hence if $G$ is not orthogonal then $d-m>m$, while if $G$ is orthogonal
$d-m\ge m$, with equality implying that $m$ is even with the restriction of the quadratic
form to $U$ being hyperbolic, while the restriction of the quadratic form to $U^{\perp}$
is elliptic.  Note that if $G$ is symplectic, then both $d$ and $m$ are even and so
$d-m\ge 4$ in this case, while if $G$ is orthogonal then $d-m\ge 4$ and $m\ge 2$.  Thus by
Lemma \ref{lem:intnondegen}, we can find $W\in (U^{\perp})^G$ such that $0\neq W\cap
U^{\perp}<U^{\perp}$ and $W\cap U^{\perp}$ is nondegenerate. Moreover, $p$ divides
$|G_{U^\perp}:G_{U^\perp,W}|$. Since $W^{\perp}\in U^G$, $H_W=H_{W^{\perp}}$ and
$G_U=G_{U^\perp}$ it follows that $G$ has a $p$-subdegree. This proves (A) of Theorem \ref{thm:classgroups} in the $\mathcal{C}_1$ case.

For $\frac{3}{2}$-transitivity, if $L\neq\PSL_d(q)$ and $U$ is not totally singular, then
$p$ divides $|\Omega|$ and so the existence of a $p$-subdegree implies that $G$ is not
\tont. If $L\neq \PSL_d(q)$ and $U$ is totally singular, then by Lemma 
\ref{unique}, except in the case $L\neq\POmega^+_d(q)$ with $d\equiv 2\pmod 4$ and $\dim U = d/2$, $G$ has a unique subdegree which is a power of $p$. Hence if $G$ is
not 2-transitive then $G$ is not \tont.

For $L=\POmega^+_d(q)$ with $m=d/2\geq 5$ odd, let
$\{e_1,\ldots,e_{d/2},f_1,\ldots,f_{d/2}\}$ be a hyperbolic basis for $V$. Without loss of
generality we may suppose that $U=\la e_1,\ldots,e_{d/2}\ra$ and note that a maximal totally singular subspace $W$ of $V$ is in $U^L$ if
and only if $W\cap U$ has even codimension in $U$.  Moreover, as $G$ is primitive,
$U^G=U^L$. Thus by \cite[Lemma 9.4.2]{BCN}, the subdegrees are
$$q^{(m-i)(m-i-2)/2}\left[\begin{smallmatrix}m\\i\end{smallmatrix}\right]_q,
$$
for each $i$ with $0\le i<m$ and $m-i$ even.\footnote{By $\left[\begin{smallmatrix}m\\i\end{smallmatrix}\right]_q$ 
we mean the Gaussian coefficient used to denote the number of $i$-subspaces in an $m$-dimensional vector 
space over $\GF(q)$. We have 
$$\left[\begin{smallmatrix}m\\i\end{smallmatrix}\right]_q=\frac{(q^m-1)(q^{m-1}-1)\ldots(q^{m-i+1}-1)}{(q^i-1)\ldots(q-1)}$$} Hence $G$ is not \tont.

If $L=\PSL_d(q)$ the suborbits are $\Omega_i=\{W\in\Omega \mid \dim(U\cap W)=i\}$ for each
$i,$ $0\leq i\leq m$. By \cite[Lemma 9.3.2(ii)]{BCN},
$$|\Omega_0|=q^{m^2}\left[\begin{smallmatrix}d-m\\m\end{smallmatrix}\right]_q.$$ 
When $m=1$, $G$ is 2-transitive. For $m\geq 2$,
\begin{align*}
  |\Omega_{m-1}| 
&=q\left[\begin{smallmatrix}d-m\\1\end{smallmatrix}\right]_q\left[\begin{smallmatrix}m \\
m-1\end{smallmatrix}\right]_q \\
  &=\frac{(q^m-1)q(q^{d-m}-1)}{(q-1)(q-1)}
\end{align*}
is another subdegree. Since  $|\Omega_0|$ and $|\Omega_{m-1}|$ are
not equal, it follows that $G$ has two different subdegrees and so $G$ is
not \tont.

\subsection{Aschbacher class $\mathcal{C}_2$:}
\label{sec:C2}
Here $H$ is the stabiliser of a decomposition 
$V=U_1 \oplus \cdots \oplus U_t$, where all $\dim U_i = m$, $d=mt$ and $t>1$.
Roughly speaking $H$ is the wreath product of a classical group on 
$U_1$ with $S_t$. Detailed descriptions of the subgroups in 
$\mathcal{C}_2$ can be found in \cite[\S 4.2]{KL}.
There are several cases to consider.

\subsubsection{$L=\PSL_d(q)$}

If $t>2$, \cite[Theorem 1.4]{james} implies that there exists $g\in G$ such that $H\cap
H^g\cap L=1$.  If $|H\cap L|$ is divisible by $p$ (which will happen if $m>1$) then it
follows that $G$ has a subdegree divisible by $p$. Moreover, as $p$ divides $|\Omega|$, it
follows that in these cases $G$ is not \tont. Thus we may assume, if $t>2$ that $p$ does not divide $|H\cap L|$ and $m=1$.

Suppose next that $m=1$ and that $p$ divides $|H|$. Then either $p$ divides $|H\cap L|$ (and hence $t=2$),  or $H$ contains a field automorphism of order $p$. If $p$ divides $|H\cap L|$ and $t=2$, then we must have $p=2$. In this
case, let $x\in H\cap L\cong D_{2(q-1)}$ be an involution. Then $|x^G|=q^2-1$ while
$|x^G\cap H|=q-1$. Thus by Lemma \ref{lem:mvconjclass}(ii), $G$ has an even
subdegree. Suppose now that $p$ does not divide $|H\cap L|$ and $H$ contains a field
automorphism $x$ of order $p$. Note that $p$ must be odd and $d<p$. Then Lemma
\ref{lem:pslouter} implies that $|x^G|>\frac{1}{2}q^{(2d^2-5)/3}$.  Since $|H\cap L|$ is coprime to $p$, the Sylow $p$-subgroup of $H$ is cyclic and so $\la x\ra^G\cap H=\la x\ra^H$. When $G=\Aut(L)$ we have $|H|=(q-1)^{d-1}d!f2$ and $|C_H(x)|=(q^{1/p}-1)^{d-1}d!f2$. where $q=p^f$. Thus for all possibilities for $G$ we have $|x^G\cap H|\leq (p-1)|x^H|\leq (p-1)(q-1)^{d-1}/(q^{1/p}-1)^{d-1}$. Hence for $d\geq 3$ we have $|x^G\cap H|\leq \frac{1}{2}q^{d-1}$ and Lemma \ref{lem:mvconjclass}(ii) implies that $G$ has a subdegree divisible by $p$. For $d=2$, we have $|x^G\cap H|<q$. Moreover,  $C_L(x)=\PSL_2(q^{1/p})$ and so in fact $|x^G|>q^{3-3/p}\geq q^2$. Thus Lemma \ref{lem:mvconjclass}(ii) again yields a $p$-subdegree. Since $|\Omega|$ is divisible by $p$, it follows that if $p$ divides $|H|$ then $G$ is not \tont.


Now suppose $m=1$ and $p$ does not divide $|H|$. Then $d<p$ and $p$ is odd. Let $V = \la v_1 \ra \oplus \cdots \oplus\la v_d\ra$ be the decomposition stabilized by $H$,  let 
$\gamma$ be the decomposition $\la v_1+v_2\ra\oplus \la v_2\ra\oplus\ldots
\oplus \la v_d\ra$ and $\delta$
be the decomposition $\la v_1+v_2\ra\oplus \la v_1-v_2\ra\oplus \la v_3\ra\oplus\ldots
\oplus \la v_d\ra$.  Then $|\gamma^H|=(q-1)d(d-1)$ and $|\delta^H|=(q-1)d(d-1)/2$, and so $G$ is not \tont.

Finally suppose that $t=2$ and $m>1$. 
Let $V=\la v_1,\ldots,v_m\ra\oplus\la v_{m+1},\ldots,v_{2m}\ra$
be the decomposition fixed by $H$ and let $\gamma$ be the decomposition $\la
v_1,\ldots,v_{m-1},v_{m+1}\ra\oplus\la v_m,v_{m+2},\ldots,v_{2m}\ra$.  
Then $H_\gamma$
fixes the decomposition $\la v_1,\ldots,v_{m-1}\ra \oplus \la v_m\ra\oplus\la
v_{m+1}\ra\oplus \la v_{m+2},\ldots,v_{2m}\ra$. Hence $p$ divides $|H:H_\gamma|$ (as $H_{U_1}$
has a composition factor isomorphic to $\SL_m(q)$ and $(H_{U_1})_\gamma$ stabilises a
decomposition of $U_1$). Since $p$ divides $|\Omega|$, $G$ is not \tont.

\vspace{3mm}
From now on, we suppose that $L$ is not $\PSL_d(q)$.

\subsubsection{$t=2$ and $U_1,U_2$ are maximal totally singular subspaces:}

Without loss of generality we may suppose that $U_1=\la e_1,\ldots,e_m\ra$ and $U_2=\la
f_1,\ldots,f_m\ra$ with $m=d/2$ such that $B(e_i,f_j)=\delta_{ij}$.  Suppose first that $m\ge 3$ and let
$W_1=\la f_1,f_2,e_3\ldots,e_m\ra$. Then $U_1\cap W_1$ has codimension 2 in $U_1$ and lies
in the same $G$-orbit as $U_1$. (The only time $G$ may not be transitive on the set of
maximal totally isotropic subspaces is when $G$ is an orthogonal group of plus type. In this case, two
maximal totally isotropic subspaces lie in the same orbit if and only if they intersect in
an even codimension subspace.) Let $W_2=\la e_1,e_2,f_3\ldots,f_m\ra$. Then
$\gamma=\{W_1,W_2\}\in\Omega$. It follows that
$H_{U_1,W_1}$ fixes $U_1\cap W_1=\la e_3,\ldots,e_m\ra$ and 
$U_2\cap W_1 = \la f_1,f_2\ra$. Hence $H_\gamma$ fixes the set 
$\{\la e_1,e_2\ra,\la e_3,\ldots,e_m\ra,\la f_1,f_2\ra,\la
f_3,\ldots,f_m\ra\}$ and so $|H:H_\gamma|$ is divisible by $p$. When $m=2$, $G$ is not
orthogonal and so $G$ is transitive on the set of $m$-dimensional totally isotropic
subspaces. Thus let $W_1=\la e_1,f_2\ra\in U_1^G$. Also let $W_2=\la e_2,f_1\ra$ and
$\gamma=\{W_1,W_2\}\in \Omega$. Then $H_\gamma$ fixes $\{\la e_1\ra,\la e_2\ra,\la f_1\ra,\la
f_2\ra\}$ and so $p$ divides $|H:H_\gamma|$.  Thus we have found a $p$-subdegree in all cases
and hence $G$ is also not \tont\ as $|\Omega|$ is divisible by $p$.

\subsubsection{$G$ is orthogonal, $t=2$ and $U_1$ and $U_2$ are nondegenerate, similar but nonisometric subspaces:}
 Here both $q$ and $d/2$ are odd.  By Lemma \ref{lem:intnondegen}, there exists $W_1\in U_1^G$ such
 that $\{0\}\ne U_1\cap W_1\ne W_1$ and $U_1\cap W_1$ is nondegenerate.  Let $W_2$ be a
 nondegenerate subspace complementary to $W_1$, and isometric to $U_2$, and let $\gamma$ be the
 decomposition $V=W_1\oplus W_2$ in $\Omega$. Then $H_{U_1,\gamma}$ fixes the nondegenerate
 subspace $U_1\cap W_1$ of $U_1$ and so $p$ divides $|H:H_\gamma|$ (by Lemma
 \ref{lem:nondegenp}). Since $H$ is not parabolic, $p$ divides $|\Omega|$ and so $G$ is
 not \tont.

\vspace{4mm}
The remaining cases to consider are where $V=U_1\perp U_2\perp\ldots\perp U_t$ where the
$U_i$ are isometric nondegenerate subspaces of dimension $m$. We break up the analysis as
follows.

\subsubsection{$m\geq 3$, or $m=2$ and $G$ is unitary:}  
Here Lemma \ref{lem:intnondegen} implies that there exists a nondegenerate subspace $W_1$
of $U_1\perp U_2$ in $U_1^G$ such that $W_1\cap U_1$ is a proper nondegenerate subspace of
$U_1$. In this case, let $W_2=W_1^{\perp}\cap (U_1\perp U_2)$ and $\gamma$ be the
decomposition $V=W_1\perp W_2\perp U_3\perp\ldots\perp U_t$. Then $H_{U_1,U_2,\gamma}$
fixes $U_1\cap W_1$. Since $H_{U_1,U_2}$ induces a classical group on $U_1$ and $U_1\cap
W_1$ is nondegenerate, it follows by Lemma \ref{lem:nondegenp} that $p$ divides
$|H_{U_1,U_2}:H_{U_1,U_2,\gamma}|$ and hence also divides $|H:H_\gamma|$.  Since $p$
divides $|\Omega|$, $G$ is not \tont.

\subsubsection{$m=2$ and $G$ is symplectic:} 
Note that  \cite[Proposition 6.2.6]{KL} implies $q\neq 2$. Suppose that $\{e_1,\ldots,
e_{d/2},f_1,\ldots,f_{d/2}\}$ is a symplectic basis and each $U_i=\la e_i,f_i\ra$. Let
$\gamma\in\Omega$ be the decomposition $V=\la e_1,f_1+f_2\ra\perp\la
e_1+e_2,f_1\ra\perp\la e_3,f_3\ra \perp\ldots\perp\la e_d,f_d\ra$. Then $H_{\gamma}$ fixes
$\{\la e_1\ra,\la f_1\ra\}$ while $\Sp_2(q)\leqslant H_{\la e_1,f_1\ra}^{\la
  e_1,f_1\ra}$. Thus $G$ has a $p$-subdegree and as $|\Omega|$ is divisible by $p$, $G$ is
not \tont.

\subsubsection{$m=2$ and $G$ is orthogonal:} 
Here $H$ preserves the decomposition $V=U_1\perp U_2\perp\ldots\perp U_{d/2}$ and the
$U_i$ are isometric nondegenerate $2$-spaces, either all hyperbolic or all elliptic. If each $U_i$ is hyperbolic then $V$ is hyperbolic and so
has a hyperbolic basis $\{e_1,\ldots,e_{d/2},f_1,\ldots,f_{d/2}\}$ and we may suppose that
each $U_i=\la e_i,f_i\ra$.  On the other hand, if each $U_i$ is elliptic, let $U_i=\la a_i,b_i\ra$ such that $Q(a_i)=1, Q(b_i)=\zeta$ for some $\zeta\in\GF(q)\backslash\{0\}$, and $B(a_i,b_i)=0$ if $q$ is odd and $B(a_i,b_i)=1$ if $q$ is even. Note that if $p$ divides $|H|$ either $p\leq d/2$ or $H$ contains a field automorphism of order $p$. Moreover, $d/2\geq 4$.

Suppose first that $q$ is even. If each $U_i$ is hyperbolic, let $W_1=\la e_1+e_2,f_1\ra$ and $W_2=\la e_2,f_1+f_2\ra$. Then
$\gamma=W_1\perp W_2\perp U_3\perp\ldots\perp U_{d/2}\in\Omega$ and $H_\gamma$ fixes
$U_1$ and $U_2$. Thus $2\binom{d/2}{2}$ divides $|\gamma^H|$ and so $G$ has an even
subdegree. If each $U_i$ is elliptic,  then 
$$\gamma=\la a_1,b_1+a_2+a_3\ra\perp\la a_2+b_1+b_3, b_2\ra\perp\la
a_3,b_3+b_1+b_2\ra\perp U_4\perp\ldots\perp U_{d/2}\in\Omega$$
and $H_\gamma$ fixes $U_1$, $U_2$ and $U_3$. Thus $|\gamma^H|$ is divisible by
$6\binom{d/2}{3}$ which is even.

Suppose now that $q$ is odd and $p\leq d/2$. By \cite[Propositions 6.2.9 and 6.2.10]{KL}, we have $q\geq 5$.  If each $U_i$ is hyperbolic, let $W_1=\la e_1,f_1+f_3\ra$, $W_2=\la e_1+e_2-e_3,f_2\ra$ and $W_3=\la e_1-e_3,-f_2-f_3\ra$.
If each $U_i$ is elliptic, since $q\geq 5$, we can choose $\lambda_1,\lambda_2\in\GF(q)\backslash\{0\}$ such that $Q(\lambda_1a_1+\lambda_2b_2)=1$ and then there exists $\mu_1,\mu_2 \in \GF(q)\backslash\{0\}$ such that $Q(\mu_1 a_1+\mu_2 b_2)=\zeta$ and $B(\lambda_1a_1+\lambda_2b_2,\mu_1 a_1+\mu_2 b_2)=0$. Then let $W_1=\langle \lambda_1a_1+\lambda_2b_2,b_1\rangle$, $W_2=\la \mu_1 a_1+\mu_2b_2,a_3\rangle$ and $W_3=\langle a_2,b_3\rangle$. Now for both the hyperbolic and elliptic case let $\beta\in\Omega$ be the decomposition $V=W_1\perp
W_2\perp W_3\perp U_4\perp\ldots \perp U_{d/2}$. Then $H_\beta$ fixes $U_1$, $U_2$ and $U_3$, so $6\binom{d/2}{3}$ divides $|\beta^H|$. Thus if $p=3$ then $G$
has a subdegree divisible by $p$. For $p\geq 5$, note that for each pair $U_i\perp
U_{i+1}$, with $4\leq i\leq p-1$ and $i$ even, we can find $Y_i,Y_{i+1}$ of the same
isometry type as $U_1$ such that $Y_i\perp Y_{i+1}=U_i\perp U_{i+1}$ and
$\{Y_i,Y_{i+1}\}\cap\{U_i,U_{i+1}\}=\varnothing$. Let $\gamma$ be the decomposition
$V=W_1\perp W_2\perp W_3\perp Y_4\perp Y_5\perp \ldots \perp Y_{p-1}\perp Y_p\perp
U_{p+1}\perp\ldots\perp U_{d/2}$. Then $H_\gamma$ fixes $\{U_1,U_2,U_3\}$ and
$\{U_1,U_2,\ldots, U_p\}$, so $|\gamma^H|$ is divisible by $\binom{p}{3}\binom{d/2}{p}$,
which is divisible by $p$.

Suppose now that $p>d/2$ and $H$ contains a field automorphism $x$ of order $p$. Note that
since $d\geq 8$ this implies that $p\geq 5$ and $q=p^f\geq p^p$. By Lemma \ref{lem:pslouter},
$|x^G|>\frac{1}{4}q^{d(d-1)2/5}$. Now $|H|\leq (2(q+1))^{d/2}(d/2)!(q-1)f3<q^{3d/4+3}$ due
to the conditions on $q$.  Thus $|x^G\cap H|^2<|x^G|$ and so by Lemma
\ref{lem:mvconjclass}, $G$ has a subdegree divisible by $p$.

Now since $p$ divides $|\Omega|$ it follows that if $p$ divides $|H|$ then $G$ is not
\tont. If $p$ does not divide $|H|$ note that $p\geq 5$.  As above, choose $Y_1,Y_2$ to
be nondegenerate 2-subspaces of $\la U_1,U_2\ra$ of the same isometry type as $U_1$ and orthogonal to each other. Let $\beta'\in\Omega$ be the decomposition given by
$V=Y_1\perp Y_2\perp U_3\perp\ldots\perp U_{d/2}$. Then $\binom{d/2}{2}$ divides
$|(\beta')^H|$ and so let $\ell$ be the integer such that $|(\beta')^H|=\ell\binom{d/2}{2}$.  Now let
$Y_3,Y_4$ be nondegenerate 2-subspaces of $\la U_3,U_4\ra$ of the same isometry type as
$U_1$ and orthogonal to each other. Then we can choose $Y_3,Y_4$ so that if
$\beta''\in\Omega$ is the decomposition $V=U_1\perp U_2\perp Y_3\perp Y_4\perp U_5\perp
\ldots\perp U_{d/2}$ then $|(\beta'')^H|=\ell\binom{d/2}{2}$. Now let $\gamma$ be the
decomposition $V=Y_1\perp Y_2\perp Y_3\perp Y_4\perp U_5\perp\ldots\perp U_{d/2}$. Then $H_\gamma$ fixes $\{U_1,U_2,U_3,U_4\}$ and preserves the partition $\{\{U_1,U_2\},\{U_3,U_4\}\}$. Thus 
$|\gamma^H|=3\ell^2\binom{d/2}{4}$. Except in the case where $d/2=4$ and $\ell\neq 2$, we have $|\gamma^H|\neq |(\beta')^H|$ and so $G$ is not \tont. When $d/2=4$ and $\ell=2$, note that $|\gamma^H|=12$ while for $\beta$ introduced previously (and whose definition is still valid for $p>d/2$), $|\beta^H|$ is divisible by $6\binom{4}{3}=24$. Hence $G$ is not \tont\ in this case either.

\subsubsection{$m=1$ and $G$ is orthogonal:} 
According to \cite[Table 4.2A]{KL},  $q=p$ is odd, and so if $p$ divides $|H|$ it
follows that $p\leq d$.  Recall that we are assuming $d\geq 7$. Let $H$ be the stabiliser
of the decomposition $\alpha=\la v_1\ra\perp\la v_2\ra\perp\ldots\perp\la v_d\ra$. By the
discussion in \cite[p100-101]{KL} we may assume that $B(v_i,v_i)=1$ for all $i$ and hence
$Q(v_i)=2^{-1}$. Note also that not all orthogonal decompositions of $V$ into nonsingular 1-spaces
may lie in $\Omega$ as there are two isometry types of nonsingular 1-spaces. Note also
that for $\lambda\in\GF(p)$, $Q(v_1+\lambda v_2)=Q(v_1)(1+\lambda^2)$ and so
$Q(v_1+\lambda v_2)=0$ if and only if $\lambda^2=-1$. Hence the space
$$\la v_1,v_2\ra \hbox{ is }\left\{ \begin{array}{l}
    \text{elliptic when }q\equiv 3\pmod 4\\
    \text{hyperbolic when }q\equiv 1\pmod 4
               \end{array}\right. $$

Suppose first that $p$ divides $|H|$ with $p>5$. Then there exist at least three 1-dimensional subspaces of $\la v_1,v_2\ra$ upon which the quadratic form has the same parity as it has on $\la v_1\ra$. Thus there exist $\lambda_1,\lambda_2\in\GF(p)\backslash\{0\}$ such
that $Q(\lambda_1v_1+\lambda_2 v_2)=Q(v_1)$ and hence there exists $g\in L$ mapping $v_1$ to
$\lambda_1 v_1+\lambda_2 v_2$.   Now $\la\lambda_1 v_1+\lambda_2 v_2\ra^\perp\cap\la
v_1,v_2,v_3\ra=\la \mu_1v_1+\mu_2 v_2,v_3\ra$ for some $\mu_1,\mu_2\neq0$, and this subspace has a decomposition $\la x_1\ra\perp\la x_2\ra$
with $Q(x_1)=Q(x_2)=Q(v_1)$ such that each $x_i\neq v_3$.  Note that each
$x_i=\sum_{j=1}^{3}\xi_j v_j$ with each $\xi_j\neq0$. Now
\begin{align*}
\gamma=&\la \lambda_1v_1+\lambda_2 v_2\ra\perp\la x_1\ra\perp\la x_2\ra \perp
\la \lambda_1v_4+\lambda_2 v_5\ra\perp\la \mu_1v_4+\mu_2 v_5\ra\perp\ldots \perp \\
&\la
\lambda_1v_{p-1}+\lambda_2 v_p\ra\perp\la \mu_1v_{p-1}+\mu_2 v_p\ra\perp\la
v_{p+1}\ra\perp\ldots\perp\la v_d\ra\in\Omega.
\end{align*}
Then $H_\gamma$ fixes $\la v_3\ra$, $\{\la v_1\ra,\la v_2\ra\}$ and $\{\la
v_1\ra,\ldots,\la v_p\ra\}$ and so $|\gamma^H|$ is divisible by
$3\binom{p}{3}\binom{d}{p}$, which is divisible by $p$.

Suppose next that $p=5$. Since $d\ge 7$, $|H|$ is divisible by $5$. In this case
$Q(v_1)=2^{-1}=3$, which is a nonsquare. Hence if $Q(x)$ is a nonsquare then $\la x\ra\in
\la v_1\ra^G$. The nonsquares in $\GF(5)$ are 2 and 3. Let $x_1=v_1+2v_2+v_3$,
$x_2=v_1+v_2+2v_3$, $x_3=2v_1+v_2+v_3+v_4+2v_5$, $x_4=2v_1+v_2+v_3+3v_4+3v_5$ and
$x_5=2v_1+v_2+v_3+2v_4+v_5$. Then $Q(x_1)=Q(x_2)=Q(x_3)=Q(x_5)=3$ and $Q(x_4)=2$, and so
$$\delta=\la x_1\ra\perp\la x_2\ra\perp\la x_3\ra\perp\la x_4\ra\perp\la
x_5\ra\perp\la v_6\ra\perp\ldots\perp \la v_d \ra \in \alpha^G.$$ Now
$H_{\delta}$ fixes $\{\la v_1\ra,\la v_2\ra,\la v_3\ra\}$ and $\{\la
v_4\ra,\la v_5\ra\}$. Hence $10\binom{d}{5}$ divides
$|\delta^H|$.

Next suppose that $p=3$. Since $d\ge 7$, $|H|$ is divisible by $3$. In this case
$Q(v_1)=2^{-1}=2$, which is a nonsquare. Let $x_1=v_1+v_2+v_3+v_4$, $x_2=v_1-v_2+v_5+v_6$,
$x_3= v_3-v_4+v_5-v_6$, $x_4=v_1+v_2-v_3-v_4$, $x_5=-v_1+v_2+v_5+v_6$ and
$x_6=-v_3+v_4+v_5-v_6$. Then $Q(x_i)=2$ for all $i$, and so
$$\delta=\la x_1\ra\perp\la x_2\ra\perp\la x_3\ra\perp\la x_4\ra\perp\la
x_5\ra\perp\la x_6\ra\perp\la v_7\ra\perp\ldots\perp\la
v_d\ra\in\alpha^G.$$ Now $H_{\delta}$ preserves that partition $\{\{\la v_1\ra,\la
v_2\ra\},\{\la v_3\ra,\la v_4\ra\},\{\la v_5\ra,\la v_6\ra\}\}$. Hence
$15\binom{d}{6}$ divides $|\delta^H|$. 

Thus if $p$ divides $|H|$ we have found a $p$-subdegree and so as $p$ divides $|\Omega|$, it follows that $G$ is not \tont. Suppose now that $p$ does not divide $|H|$. Then $p>d$
and in particular $p>7$.  Using $\lambda_1,\lambda_2,\mu_1$ and $\mu_2$ as before let
$$\beta=\la \lambda_1v_1+\lambda_2 v_2\ra\perp \la \mu_1v_1+\mu_2 v_2\ra\perp\la
v_3\ra\perp\ldots\perp\la v_d\ra\in\Omega.$$
Now $H_{\beta}$ fixes $\{\la v_1\ra,\la v_2\ra\}$ and hence
$|\beta^H|=2^i\binom{d}{2}$ for some $i\ge 0$.  Next let
$$\eta=\la \lambda_1v_1+\lambda_2 v_2\ra\perp \la \mu_1v_1+\mu_2 v_2\ra\perp \la \lambda_1v_3+\lambda_2 v_4\ra\perp \la
\mu_1v_3+\mu_2 v_4\ra\perp\la v_5\ra\perp \ldots\perp\la v_d\ra\in\Omega.$$ Then $H_{\eta}$ fixes
$\{\{\la v_1\ra,\la v_2\ra\},\{\la v_3\ra,\la v_4\ra\}\}$. Hence $|\eta^H|$ is divisible
by $3\binom{d}{4}=\binom{d}{2} (d-2)(d-3)/4$. One of $d-2$ or $d-3$ is an odd number at least 5, and hence $|\eta^H|\neq |\beta^H|$. Thus $G$ is not \tont.

 \subsubsection{$m=1$ and $G$ is unitary:}  
Let $H$ be the stabiliser of the decomposition $V=\la v_1\ra\perp\la
v_2\ra\perp\ldots\perp\la v_d\ra$. By the discussion on \cite[p100--101]{KL} we may assume
that $B(v_i,v_i)=1$ for all $i$ and by \cite[Proposition 4.2.9]{KL}, 
$H\cap L=((q+1)^{d-1}/(q+1,d)).S_d$. Thus if $p$ divides $|H|$ either $p\leq d$ or $H$ contains a
field automorphism of order $p$. Suppose first that $p=2$. Then
$B(v_1+v_2+v_3,v_1+v_2+v_3)=1$ and $\la v_1,v_2,v_3\ra\cap \la v_1+v_2+v_3\ra^{\perp}=\la
v_1+v_2,v_2+v_3\ra$. Both $v_1+v_2$ and $v_2+v_3$ are singular. Since $\la
v_1+v_2,v_2+v_3\ra$ is nondegenerate, there exist nonsingular $x_1,x_2\in\la
v_1+v_2,v_2+v_3\ra$ with $x_1\in x_2^\perp$. Note that for $i=1,2$, $x_i\neq
v_1,v_2,v_3$. Thus
$$\beta=\la v_1+v_2+v_3\ra\perp\la x_1\ra\perp\la x_2\ra\perp\la
v_4\ra\perp\ldots\perp\la v_d\ra\in\Omega.$$ Moreover, $H_{\beta}$
fixes $\la v_1,v_2,v_3\ra$. Hence $\binom{d}{3}=d(d-1)(d-2)/6$ divides
$|\beta^H|$. If $d$ is even, or $d\equiv 1\pmod 4$, it
follows that $G$ has an even subdegree. If $d\equiv 3\pmod 4$ and $d>3$,
let
$$\gamma=\la v_1+v_2+v_3\ra\perp\la x_1\ra\perp\la x_2\ra\perp\la v_4+v_5+v_6\ra\perp\la
y_1\ra\perp\la y_2\ra\perp\la v_7\ra\perp\ldots\perp\la v_d\ra\in\Omega,$$ where $y_1,y_2$
are nonsingular elements of $\la v_4+v_5,v_5+v_6\ra$. Then $|\gamma^H|$ is divisible by
$10\binom{d}{6}$, which is even. Hence if $d>3$, $G$ has a subdegree divisible by
$p=2$. Now let $d=3$ and let $x\in H\cap L$ be an involution. Then by Lemma
\ref{lem:upperbounds}, $|x^G|=(q-1)(q^3+1)$ while $|x^G\cap H|=3(q+1)$. Hence for $q\geq
8$, Lemma \ref{lem:mvconjclass} implies that $G$ has an even subdegree. If $q=4$, a \textsc{Magma} calculation shows that $L$ has a subdegree equal to $150$. Hence for all even
$q\ge 4$, $G$ has an even subdegree (note that $(d,q)\neq (3,2)$).

Next suppose that $p$ is odd and $p\leq d$. In $\la v_1,v_2,v_3\ra$, we have $\la
v_1+v_2\ra^{\perp}=\la v_1-v_2,v_3\ra$. Then there exist
$\lambda,\mu\in\GF(q^2)\backslash\{0\}$ such that $v_1-v_2+\lambda v_3$ and $v_1-v_2+\mu
v_3$ are nonsingular and $B(v_1-v_2+\mu v_3,v_1-v_2+\lambda v_3)=0$. Then
\begin{equation}
 \label{eq:gamma}
\gamma=\la v_1+v_2\ra\perp\la v_1-v_2+\lambda v_3\ra\perp\la v_1-v_2+\mu
v_3\ra\perp\la v_4\ra\perp\ldots\perp\la v_d\ra\in \Omega.
\end{equation}
Now $H_{\gamma}$ fixes $\{\la v_1\ra,\la v_2\ra\}$ and $\la v_3\ra$, and
if $g\in H_{\gamma}$ maps $v_1$ to $\xi v_1$ then $g$ maps $v_2$ to
$\xi v_2$.  Hence $3(q+1)\binom{d}{3}$ divides
$|\gamma^H|$. Thus if $p=3$ we have found a $p$-subdegree.  If $p\ge 5$, let
\begin{align*}
  \delta=&\la v_1+v_2\ra\perp\la v_1-v_2+\lambda v_3\ra\perp\la v_1-v_2+
  \mu v_3\ra\perp\la v_4+v_5\ra\perp\\
  &\la v_4-v_5\ra\perp\ldots\perp\la
  v_{p-1}+v_p\ra\perp\la v_{p-1}-v_p \ra  \perp 
  \la v_{p+1}\ra\perp\ldots\perp\la v_d\ra\in\Omega.
\end{align*}
Then $\binom{p}{3}\binom{d}{p}$ is divisible by $p$ and divides
$|\delta^H|$.

Finally, suppose that $p>d$ and $H$ contains a field automorphism $x$ of order $p$. Note that $p\geq 5$ and $q\geq p^p>d!$. By
Lemma \ref{lem:pslouter}, $|x^G|>\frac{q}{2(q+1)}q^{(4d^2-9)/5}>q^{(4d^2-10)/5}$, as $4<q^{1/5}$. Now $|H|\leq (q+1)^{d-1}d!2f<(q+1)^{d-1}q^2<(2q)^{d-1}q^2<q^{(4d+2)/3}$.  Hence $|x^G|>|x^G\cap H|^2$ for $d\geq 5$. For $d=4$, we
see that $|H|\leq (q+1)^348f<q^3384f<q^4$ as $q\geq 5^5$. Since $|x^G|>q^{10}$ we also have $|x^G|>|x^G\cap H|$ in this case. Finally, for $d=3$ we have $|H|\leq (q+1)^212f< \frac{11}{10}q^212f<q^{13/5}$ as both $p$ and $f$ are at least 5. Since $|x^G|>q^{26/5}$ in this case we also have that $|x^G|>|x^G\cap H|^2$. Hence for all values of $d$, Lemma \ref{lem:mvconjclass} yields a $p$-subdegree. 

Now $p$ divides $|\Omega|$ and so if $p$ divides $|H|$, $G$ is not \tont. If $(|H|,p)=1$
then $p$ is odd and $p>d$. Let $\beta=\la v_1+v_2\ra \perp \la v_1-v_2\ra\perp\la
v_3\ra\perp\ldots\perp\la v_d\ra\in \Omega$. If $g\in H_{\beta}$ then $g$ can interchange
$v_1$ and $v_2$, but if $g:v_1\mapsto \lambda v_1$ then $g:v_2\mapsto \lambda v_2$.  Thus
$|\beta^H|=(q+1)\binom{d}{2}$. For $\gamma$ defined in (\ref{eq:gamma}) we have already seen that $|\gamma^H|$ is divisible by $3(q+1)\binom{d}{3}$ (the definition of $\gamma$ and $|\gamma^H|$ do not depend on the condition $p\leq d$). Thus if $d>3$, then $G$ is not \tont. If $d=3$, we see that
$|H_\gamma\cap L|=2$ and so $3(q+1)^2/(q+1,3)$ divides $|\gamma^H|$. Since $(d,q)\neq(3,2)$, it follows that $|\gamma^H\neq |\beta^H|$ and so $G$ is not \tont.

\subsection{Aschbacher class $\mathcal{C}_3$:}

Here $H$ is the stabiliser of an extension field structure of $V$ as a $b$-dimensional
vector space over $\GF(q^a)$ where $d=ab$ and $a$ is prime. The 
subgroups are described in detail in \cite[\S 4.3]{KL}.

\begin{lemma}
 If $b>1$ then $G$ has a
$p$-subdegree and $G$ is not \tont.
\end{lemma}

\begin{proof}
Suppose first that if $G$ is orthogonal then $H$ is not unitary. By 
hypothesis, $b\geq 2$ and if $G$ is
orthogonal, \cite[Table 4.3.A]{KL} implies that $b\geq 3$. Since $a\neq 1$ it follows that
$H$ is insoluble. Let $T$ be a $\GF(q)$-subfield subgroup of $H$. Choose a $\GF(q^a)$ basis
for $V$ and let $W$ be the $\GF(q)$-span of this basis. Then we can choose $T$ to preserve the
decomposition $V=W\otimes \GF(q^a)$. Since $\GF(q^a)$ is $a$-dimensional over $\GF(q)$, we
can find an element of $G\setminus H$ which preserves this tensor decomposition and
centralises $T$. Hence $N_H(T)<N_G(T)$.  Let $S$ be a Sylow $p$-subgroup of $H$. Then $\la
S,T\ra$ contains the weakly closed normal subgroup $H^{(\infty)}$ of $H$ and so Lemma
\ref{lem:weaklyclosed} implies that $G$ has a $p$-subdegree.

Next we suppose $L=\POmega^{\pm}_{2d}(q)$ and $H$ is a unitary group of dimension $d$. By
Proposition \ref{prn:sylpgen}, $H_0=H^{(\infty)}$ has a cyclic subgroup $T$ such that for
$S\in\Syl_p(H)$, we have $H_0\leqslant \la S,T\ra$. Now $T$ fixes a nondegenerate
$\GF(q^2)$-hyperplane and acts trivially on its perp. Thus $T$ fixes a $\GF(q)$-subspace
of codimension $2$ and acts trivially on its perp. Hence $N_{H}(T)<N_{G}(T)$ and
so Lemma \ref{lem:weaklyclosed} implies that $G$ has a subdegree divisible by $p$.

Since $|\Omega|$ is always divisible by $p$ and $G$ has a 
$p$-subdegree, it follows that $G$ is not \tont.
\end{proof}

It remains to deal with the case where $b=1$. This only occurs where
$L=\PSL^{\epsilon}_d(q)$ and $d$ is prime.  First we deal with the case where $(d,q)=(2,2^f)$. Recall that a transitive group is strongly \tont\ if all non-principal constituents of the permutation character are distinct and have the same degree.

\begin{lemma}\label{lem:L2}
Let $G$ be a group with socle $\PSL_2(q)$, where $q = 2^f\geq 4$, and let $H$ be a maximal subgroup of $G$
with $H \cap L$ dihedral of order $2(q+1)$. Consider the action of $G$ of degree $q(q-1)/2$
on the set of cosets of $H.$
\begin{enumerate}
\item[(1)] $G$ has no even subdegrees if and only if the index $|G:L|$ is odd.
\item[(2)] $G$ is strongly \tont\ if and only if either $G=L$ or $f$ is prime.
\item[(3)] $G$ is \tont\ if and only if either $G=L$ or  $f$ is prime.
\item[(4)] $G$ is $2$-transitive if and only if $q=4$.
\end{enumerate}
The nontrivial character degrees and the nontrivial subdegrees of the examples in $(2)$ and $(3)$ are
$q+1$ when $G=L$ and $(q+1)f$ when $G=L.f$.
\end{lemma}

\begin{proof}
We work with characters of $L$ -- the character theory for the groups $\PSL_2(q)$ is very well known, and so is the action of the outer automorphisms on the
irreducible characters (see, eg, \cite[p500]{Dixon05}). We supply details for completeness.

There are precisely $k=(q-2)/2$ irreducible characters of $L$ of degree $q+1$. 
We denote these by $\chi_j$.
The permutation character of our action of $L$ is
$\pi = 1 + \chi_1 + \cdots + \chi_k:$
to check that the multiplicity of each $\chi_j$ is 1, it is enough by Frobenius reciprocity to check
that the sum of $\chi_j (h)$ over $h \in H \cap L$ equals $2(q+1)$; this is clear, since the value of $\chi_j$
on the $q+1$ involutions is $1$, and on all nontrivial elements of order dividing $q+1$ is $0$. The assertion now follows
by comparing the degrees of $\pi$ and the character sum on the right. 
It now follows that $L$ is always strongly \tont, and hence \tont,
with rank $k+1$ and the nontrivial subdegrees are all $q+1$.

We now consider the outer automorphisms of $L$. The outer automorphism group is cyclic of 
order $f$, and we can choose the representative $\sigma$ of the generating coset of the group of 
inner automorphisms in the automorphism group in the usual way as the generator 
of the group of automorphisms induced by the Galois group of the field $\GF(q)$.
The elements of $L$ that we have not mentioned yet 
have orders dividing $q-1$ and are conjugate to powers of the diagonal matrix $a$
with diagonal entries $\rho, \rho^{-1}$, where $\rho$ is a primitive element of
$\GF(q)$. Then $\chi_j (a^i) = \epsilon^{ij} + \epsilon^{-ij}$, where 
$\epsilon$ is a primitive complex $(q-1)^{\mathrm{th}}$ root of unity.
The action of $\sigma$ takes $a$ to $a^2$, and its action on the set $\{ \chi_1, \dots ,\chi_k \}$ 
is equivalent to the action of the generator
of the Galois group sending $\rho$ to $\rho ^2$.

All the assertions now follow easily, noting in addition that the number of orbits
of any outer automorphism of $L$ on the set of nontrivial $(H \cap L)$--orbits
in this action equals the number of orbits on the set $\{ \chi_1, \dots ,\chi_k \}.$ 

(1) If $|G:L|$ is odd, all subdegrees of $G$ are odd, since this
is true of $L$ and the outer automorphisms present in $G$
have odd order. Hence $G$ has no even subdegrees.

If $|G:L|$ is even, an involutory outer automorphism in $G$ will act nontrivially
on the set $\{ \chi_1, \dots ,\chi_k \}$ 
and hence will fuse two of the $L$-suborbits, so $G$ will have an even subdegree.

(2) and (3). If $f$ is prime, $\sigma$ will act semiregularly
on the set $\{ \chi_1, \dots ,\chi_k \}$,
whence $G$ is strongly \tont\ and hence also \tont. 

(4) When $q=4$ we have $\PSL_2(4)\cong \PSL_2(5)$ and the action is 2-transitive of degree 6.

Conversely, if $G$ is \tont, then $\sigma$ must be semiregular
on the set of $(H \cap L)$--orbits, whence it must be semiregular also 
on the set $\{ \chi_1, \dots ,\chi_k \}$, 
and hence as a Galois automorphism of $\GF(q)$.
It follows that $f$ is a prime.
\end{proof}

We remark that it is not hard to find an explicit $G$-invariant correspondence
between the set of orbitals of the action of $L$ on the sets of $H$-cosets and $\GF(q)$, making the action of $\sigma$ quite explicit. This was done 
in much more generality by Inglis \cite{inglis} (note that $\PSL_2(q) = \Sp_2(q)$ and the maximal dihedral subgroups of
$\PSL_2(q)$ are the orthogonal groups $O^{\pm}_2(q)$).

Before dealing with the remaining cases we need to set up some notation. By \cite[Proposition 4.3.6]{KL}, $H\cap L=C_\ell\rtimes C_d$
where $\ell=(q^d-\epsilon)/((q-\epsilon)(d,q-\epsilon))$. Suppose first that
$\epsilon=+$. Then we can identify $V$ with $\GF(q^d)$. Let $\mu$ be a primitive
element of $\GF(q^d)$. Define the maps $\mu:\xi\mapsto \xi\mu$ and $\phi:\xi\mapsto \xi^q$
on $V$. Both are $\GF(q)$-linear and $\la \mu,\phi\ra\cong C_{q^d-1}\rtimes C_d$. Then
$H\cap L$ is the image in $\PSL_d(q)$ of the intersection of $\la \mu,\phi\ra$ with
$\SL_d(q)$. Note that $\xi^\phi=\xi$ if and only if $\xi\in\GF(q)$ and so $1$ is an
eigenvalue of $\phi$ with multiplicity $1$. For $\epsilon=-$ we have $d\geq 3$, and we
identify $V$ with $\GF(q^{2d})$ and let $\mu\in\GF(q^{2d})$ have order $q^d+1$. We then
define the maps $\mu:\xi\mapsto \xi\mu$ and $\phi:\xi\mapsto \xi^{q^2}$ on $V$ which are
both $\GF(q^2)$-linear and preserve the Hermitian form
$B(\xi_1,\xi_2)=Tr_{q^{2d}\rightarrow q^d}(\xi_1\xi_2^{q^d})$. Then $\la
\mu,\phi\ra\cong C_{q^d+1}\rtimes C_d$ and $H\cap L$ is the image in $\PSU_d(q)$ of the
intersection of $\la \mu,\phi\ra$ with $\SU_d(q)$. Moreover, $1$ is an eigenvalue of $\phi$ with
multiplicity $1$. Indeed, for both values of $\epsilon$, $x^d-1$ is the characteristic
polynomial of $\phi$ and so for $d\neq p$, each $d^{\mathrm{th}}$ root of unity occurs as
an eigenvalue with multiplicity 1. When $d=p$, the element $\phi$ has Jordan form $J_p$.

We first look for $p$-subdegrees. Recall Definition \ref{def:nu} of $\nu(x)$ for an element $x\in\PGL_d(q)$.

\begin{lemma}
\label{lem:psubbeq1}
Let $G$ be an almost simple group with socle $L=\PSL^{\epsilon}_d(q)$ and $H$ be as above with $(d,q)\neq (2,2^f)$. If $p$ divides $|H|$ and $G$ has no
subdegrees divisible by $p$ then $G$ is the $2$-transitive group $L_3(2).2$ of degree $8$ as in line $7$ of Table \ref{tab:classgp}.
\end{lemma}

\begin{proof}
Suppose first that $d\geq 3$ and that $p$ divides $|H|$. Then either $p=d$ or $H$ contains a
field automorphism, graph automorphism or graph-field automorphism of order $p$. If $p=d$
then $p$ is odd and we let $x=\phi\in H\cap L$ have order $p$. Since $\phi$ has Jordan form $J_p$, we have $\nu(x)=d-1$. Thus by Proposition \ref{prn:burnessbounds},
$|x^G|>\frac{q}{2(q-\epsilon)(q+1)}q^{d(d-1)}>q^{d(d-1)-3}$. Since $|H\cap L|<q^{d+1}$ it
follows from Lemma \ref{lem:mvconjclass}(ii) that if $p=d\geq 5$ then $G$ has a subdegree
divisible by $p$.  For $d=p=3$, \cite[Lemma 3.18 and Lemma 3.20]{burness} implies that
$|x^L|=q(q^2-1)(q^3-\epsilon)$ while we have that $|x^L\cap H|\leq 2(q^2+\epsilon q+1)$.
Thus except for $(q,\epsilon)=(3,+)$, Lemma \ref{lem:mvconjclass}(ii) yields a subdegree
divisible by $p=3$. For $q=3$, a \textsc{Magma} \cite{magma} calculation shows that the subdegrees
for $L=\PSL_3(3)$ are $\{1,13^5,39^2\}$.  Thus we may assume that $d\neq p$ and $H$ contains a field, graph or graph-field automorphism $x$ of order $p$. Now $N_{\Aut(L)}(H)=C_{(q^d-1)/(q-1)}\rtimes (C_{df}\times C_2)$ or $C_{(q^d+1)/(q+1)}\rtimes C_{2df}$ for $\epsilon=+$ or $-$ respectively. It follows that $|x^G\cap H|\leq q^d$. Now by Lemma
\ref{lem:pslouter}, $|x^G|>\frac{1}{2}q^{(d^2-3)/2}$. Thus for $d\geq 5$, $|x^G|>|x^G\cap
H|^2$ and so by Lemma \ref{lem:mvconjclass}(ii), $G$ has a $p$-subdegree. This leaves the cases $d=3$ and $d=2$.

Let $d=3$. Then $p=2$ or $p\geq 5$. Let $\overline{H}=N_{\PGL_3(q)}(H)$. For $p\geq 5$, we have that $x$ is a field automorphism and so Lemma \ref{lem:pslouter} implies that
$|x^G|>\frac{1}{2}\left(\frac{q}{q+1}\right)^{(1-\epsilon)/2}q^{(4d^2-9)/5}>q^{22/5}>q^4$. Now $C_{\overline{H}}(x)= C_{q^{2/p}+\epsilon q^{1/p}+1}\rtimes C_3$ and there is only one conjugacy class of subgroups of $H$ of order $p$. Thus $|x^G\cap H| <(p-1)(q^2+\epsilon q+1)/(q^{2/p}+\epsilon q^{1/p}+1)<(p-1)2q^{2-2/p}<q^2$. Hence $|x^G|>|x^G\cap H|^2$ and Lemma \ref{lem:mvconjclass}(ii) yields a $p$-subdegree. So we now assume that $p=2$. If $x$ is a field automorphism of order 2 then $\epsilon=+$ and so Lemma \ref{lem:pslouter} implies $|x^G|>\frac{1}{2}q^{3}$ while we have $|x^G\cap H|<(q^2+q+1)/(q+q^{1/2}+1)<q$. Thus Lemma \ref{lem:mvconjclass}(ii) again gives a $p$-subdegree.
If $x$ is a graph-field automorphism then $\epsilon=+$ and by \cite[Proposition 4.8.5]{KL}, $C_L(x)\leq \PGU_3(q^{1/2}).$ Hence $|x^G|\geq|L|/|\PGU_3(q^{1/2})|=q^{3/2}(q^{3/2}-1)(q+1)/(3,q-1)$. Moreover, $C_{\overline{H}}(x)=C_{q^{3/2}+1}\rtimes C_3$. Hence $|x^G\cap H|^2\leq (q^{3/2}-1)^2<|x^G|$ and so by Lemma \ref{lem:mvconjclass}(ii), $G$ has a subdegree divisible by $p$. Finally, let $x$ be a graph automorphism of order 2. By \cite[(19.9)]{aschseitz}, $C_L(x)=\PSO_3(q)$ and so $|x^L|=q^2(q^3-\epsilon)/(3,q-\epsilon)$. Moreover, $|x^G\cap H|<q^2+\epsilon q+1$ and so for $q\geq 8$ or $(q,\epsilon)=(4,-)$, Lemma \ref{lem:mvconjclass}(ii) implies that $G$ has a subdegree divisible by $p$ . This leaves us to consider $G=\PSL_3(2).2$ of degree $8$, and primitive groups of degree 960 with socle $L=\PSL_3(4)$ (note that $(q,\epsilon)\neq (2,-)$). The first has subdegrees 1 and 7 and so does not have an even subdegree. A \textsc{Magma}~\cite{magma} calculation shows that all groups in the second case with an even order point stabiliser have an even subdegree. 
  
Finally we deal with the case where $d=2$ and so $L=\PSL_2(q)$. Recall that $q=p^f$ is odd. If $p$ divides $|H|$ then $H$ contains a field automorphism $x$ of order $p$ and in particular, $p$ divides $f$. Then $C_L(x)=\PSL_2(q^{1/p})$ and so
$|x^G|>q^{3-3/p}\geq q^2$. Moreover, letting $\overline{H}=N_{\Aut(L)}(H)$ we have $|\overline{H}|=2(q+1)f$ and $|C_{\overline{H}}(x)|=2(q^{1/p}+1)f$. Since the Sylow $p$-subgroup of $\overline{H}$ is cyclic, it contains a unique conjugacy class of subgroups of order $p$ and so $|x^G\cap H|\leq (q^{1-1/p}-q^{1-2/p}+\cdots +1)(p-1)<q$. Thus Lemma \ref{lem:mvconjclass}(ii)
implies that $G$ has a subdegree divisible by $p$.
\end{proof}

To deal with $\frac{3}{2}$-transitivity, we first need the following lemma.

\begin{lemma} \label{lem:ddivx2} Let $p$ and $d$ be primes with $d$ odd and let
 $q=p^f$ for some positive integer $f$. Then $d$ does not divide
  $y=(q^d-\epsilon)/((q-\epsilon)(q-\epsilon,d))$ for $\epsilon=\pm1$
\end{lemma}

\begin{proof}
 By Fermat's Little Theorem, $d$ divides $q^{d-1}-1$. Suppose that $d$ divides $y$. Then
 $d$ divides $q^d-\epsilon$ and so divides $q^d-\epsilon
 q^{d-1}=q^{d-1}(q-\epsilon)$. Since $d$ divides $y$ it is coprime to $q$ and so $d$
 divides $q-\epsilon$. If $q\equiv ad+\epsilon \pmod {d^2}$ then using the Binomial
 expansion we see that for each $r$, $q^r\equiv (\epsilon)^r+rad(\epsilon)^{r-1}\pmod
 {d^2}$. Hence $(q^d-\epsilon)/(q-\epsilon)\equiv d+\epsilon ad+\epsilon 2ad+ \cdots
 +\epsilon(d-1)ad\pmod{d^2}\equiv d+\epsilon ad(d(d-1)/2))\pmod{d^2}\equiv d\pmod {d^2}$.
 Thus $d$ does not divide $(q^d-\epsilon)/((q-\epsilon)(q-\epsilon,d))$.
\end{proof}

\begin{lemma}
Let $G$ be an almost simple group with socle $L=\PSL^{\epsilon}_d(q)$ and $H$ be as above with $(d,q)\neq (2,2^f)$. Then $G$ is not \tont\ unless $L=\PSL_3(2)$ acting $2$-transitively of degree $8$.
\end{lemma}

\begin{proof}
Since $p$ divides $|\Omega|$ in all cases, if $G$ has a $p$-subdegree then $G$ is not
\tont. Thus by Lemma \ref{lem:psubbeq1}, we are left to consider the case where $p$ does 
not divide $|H|$. In particular, if $d\geq 3$ then $p\neq d$. We note that if 
$(d,\epsilon,q)=(3,+,2)$ then $L\cong\PSL_2(7)$ acting 2-transitively of degree 8, so 
assume we are not in this case.

Suppose first that $d\geq 3$ with $p\neq d$. Now for $d\geq 3$, we have $H\cap
L=C_\ell\rtimes C_d$ where $\ell=(q^d-\epsilon)/((q-\epsilon)(q-\epsilon,d))$. Note that
for $M\leqslant C_\ell$ with $M$ nontrivial, Lemma \ref{lem:ddivx2} implies that $|M|$ is coprime to $d$ and
so as $H=N_G(M)$ it follows that if $M\leqslant H\cap H^g$ then $g\in H$. Thus for all
$g\in G\backslash H$, either $H\cap H^g\cap L=C_d$ or $H\cap H^g\cap L=1$. Now for $d\neq
p$, $\phi$ is a semisimple element with 1 as an eigenvalue with multiplicity 1. Thus
$C_L(\phi)$ is not contained in $H$, and so there exists $g\in L$ such that $H\cap H^g\cap
L=\la \phi\ra\cong C_d$, that is, there exists a subdegree of length $\ell$, which by Lemma \ref{lem:ddivx2} is not divisible by $d$. Thus it is sufficient to find a subdegree divisible by $d$. Since
$\nu(\phi)=d-1$, \cite[Proposition 3.36]{burness} and 
\cite[Theorem 4.2.2(j)]{GLS} imply that
$|\phi^G|>\frac{1}{2}\left(\frac{q}{q+1}\right)^{d-1}q^{d(d-1)}>\frac{1}{2}q^{(d-1)^2}>q^{(d-1)^2-1}$. On
the other hand, $|\phi^G\cap H|<|H\cap L|<q^d.d=q^{d+\log_q(d)}\leq q^{d+\log_2(d)}$. Hence
for $d\geq 5$, $|\phi^G|>|\phi^G\cap H|^2$ and so by Lemma \ref{lem:mvconjclass}(ii), $G$ has a
subdegree divisible by $d$. For $d=3$ and $\epsilon=+$ we have $q\geq 4$ and $|\phi^G|>\frac{1}{2}q^6$
while $|\phi^G\cap H|\leq(q^2+q+1).2$. It follows that $|\phi^G|>|\phi^G\cap H|^2$
and so by Lemma \ref{lem:mvconjclass}(ii), $G$ has a subdegree divisible by $d=3$. 
For $\epsilon=-$, note that $q\geq 4$ and so
$|\phi^G|>\frac{1}{2}\left(\frac{q}{q+1}\right)^2q^6>\frac{1}{2}q^5$. Now $|\phi^G\cap H|\leq
(q^2-q+1)2$ and so for $q\geq 7$, $|\phi^G|>|\phi^G\cap H|^2$. For $q=5$, we actually have
$|\phi^G\cap H|\leq (q^2-q+1)2/3$ which is sufficient to show $|\phi^G|>
|\phi^G\cap H|^2$.  Thus for
$q\geq 5$, Lemma \ref{lem:mvconjclass} yields a subdegree divisible by $d$.  For $L=\PSU_3(4)$,
a \textsc{Magma} \cite{magma} calculation shows that the subdegrees of $L$ are $1, 13^9$ and
$39^{38}$. Thus $G$ is not \tont\ in this case also.

Suppose now that $d=2$ and $q$ is odd. Then $H\cap L\cong D_{q+1}$. Suppose first that
$(q+1)/2$ is odd. Then $|H|$ and $|\Omega|$ are even. Let $x\in H\cap L$ have order
2. Then $|x^G|=q(q+1)/2$ while $|x^G\cap H|=(q+1)/2$. It follows from Lemma
\ref{lem:mvconjclass}(ii) that $G$ has an even subdegree and so $G$ is not \tont.  Suppose
now that $(q+1)/2$ is even. Then $H$ is the centraliser of an involution and for each
$g\in G\backslash H$ we have $L\cap H\cap H^g=1,C_2$ or $C_2^2$. Now $H\cap L$ contains
$(q+3)/2$ involutions and so each involution is contained in $(q+3)/2$ conjugates of
$H$. The number of conjugates of $H$ containing an involution of $H\cap L$ is at most
$(q+3)^2/4$. Then since there are $q(q-1)/2$ conjugates of $H$ and $q(q-1)/2>(q+3)^2/4$
for $q\geq 11$, it follows that in these cases there exists $g\in G\backslash H$ such that
$H\cap H^g\cap L=1$, and hence a subdegree divisible by $q+1$.  A \textsc{Magma} \cite{magma}
calculation verifies the existence of such a subdegree for $q=7$. Since $q+1$ does not
divide $|\Omega|-1$ it follows that $G$ is not \tont.
\end{proof}

\subsection{Aschbacher class $\mathcal{C}_4$:}

Here $H$ is the stabiliser in $G$ of a tensor product structure $V=U\otimes W$ where
$\dim(U)=a$, $\dim(W)=b$ and $d=ab$ with $a,b\ge 2$ and $U$ and $W$ are not isometric. Detailed descriptions of the subgroups are given in \cite[\S 4.4]{KL}.
We note that by \cite[Lemma 3.7]{LieShalev}, if $g=(g_1,g_2)\in\GL(U)\otimes \GL(W)$ then
$\nu(g)\geq \max\{a\nu(g_2),b\nu(g_1)\}$, with $\nu$ as in Definition \ref{def:nu}. Moreover, if $g_1$ has order $p$ and $g_2=1$
then $\nu(g)=b\nu(g_1)$ and similarly, if $g_2$ has order $p$ and $g_1=1$ then
$\nu(g)=a\nu(g_2)$.

 By \cite[Table 4.4A]{KL} the possibilities for $G$ and $H$ are as follows:

\begin{enumerate}
\item $G$ with socle $\PSL_d(q)$ and $H$ of type $\GL_a(q)\otimes\GL_b(q)$ with $a<b$,
\item $G$ a unitary group and $H$ of type $U_a(q)\otimes U_b(q)$ with
  $a<b$,
\item $G$ a symplectic group and $H$ of type $\Sp_a(q)\otimes
  O^{\epsilon}_b(q)$ with $q$ odd and $b\ge 3$ and $a\geq 2$,
\item $G$ of type $O^+$ and $H$ of type $\Sp_a(q)\otimes \Sp_b(q)$ with
$2\leq a<b$,
\item $G$ an orthogonal group and $H$ of type $O^{\epsilon_1}_a(q)\otimes
  O^{\epsilon_2}_b(q)$ with $q$ odd, $a,b\ge 3$ and
  $(a,\epsilon_1)\neq(b,\epsilon_2)$.
\end{enumerate}

Suppose first that $L=\PSL^{\epsilon}_d(q)$ and let $g=(1,x)\in H$ 
of order $p$ such that
$\nu(x)=1$. Thus $\nu(g)=a<b$ and so if $h\in g^G\cap H$ then $h=(1,x_1)$ for some
$x_1\in\PSL^{\epsilon}_b(q)$ of order $p$ with $\nu(x_1)=1$. Since $a<d/2$, Proposition
\ref{prn:burnessbounds} implies that
$|g^G|>\frac{1}{2}\frac{q}{(q-\epsilon)(q+1)}q^{2a(d-a)}>q^{2a^2(b-1)-4}$ while by Lemma
\ref{lem:upperbounds} $|g^G\cap
H|=(q^b-(\epsilon^b))(q^{b-1}-(\epsilon)^{b-1})/(q-\epsilon)<q^{2b-1}$.  Since $b\geq 3$,
it follows from Lemma \ref{lem:mvconjclass}(ii) that $G$ has a subdegree divisible by $p$.

Next suppose that $G$ is a symplectic group and $H$ is of type $\Sp_a(q)\otimes
O^{\epsilon}_b(q)$ with $q$ odd and $b\ge 3$. Let $g=(x,1)\in H$ of order $p$ with
$\nu(x)=1$. Then $\nu(g)=b\geq 3$ and by Proposition \ref{prn:burnessbounds},
$|g^G|>\frac{q}{4(q+1)}q^{b(ab-b)}>q^{b^2(a-1)-2}$. Since 
$O^{\epsilon}_b(q)$ does not
contain any unipotent elements with $\nu(y)=1$, if $a>b/2$ then $|g^G\cap H|$ is the
number of transvections in $\Sp_a(q)$. Thus by Lemma \ref{lem:upperbounds}, $|g^G\cap
H|<q^a$. Since $a\geq 2$, it follows from Lemma \ref{lem:mvconjclass} that $G$ has a
$p$-subdegree. If $a\leq b/2$, then $H$ has a weakly closed normal subgroup
$H_0\cong\PO^{\epsilon}_b(q)$. By Proposition \ref{prn:sylpgen}, $H_0$ contains a cyclic
subgroup $T$ such that for any Sylow $p$-subgroup $S$ of $H$ we have $H_0\leqslant\la
T,S\ra$. If $\epsilon=+$ then $T$ centralises a 2-subspace of $W$ and hence a
$2a$-subspace of $V$. Thus $N_H(T)<N_G(T)$. If $\epsilon=-$, then $T$ acts irreducibly on
$W$ and preserves a direct sum decomposition of $V$ into $a$ totally singular subspaces of $V$ each
with dimension $b$.  Now $T$ is irreducible on each subspace in the decomposition and
$|T|$ is at most $q^{b/2}+1$. Thus $T$ is centralised by an element of order $q^b-1$ which
also fixes this decomposition and is not in $H$.  Thus by Lemma \ref{lem:weaklyclosed},
$G$ has a subdegree divisible by $p$.

Next suppose that $G$ is of type $O^+$ and $H$ is of type $\Sp_a(q)\otimes \Sp_b(q)$ with
$a<b$. Let $g=(1,x)\in H$ of order $p$ where $\nu(x)=1$. Then $\nu(g)=a$ and $h\in g^G\cap
H$ if and only if $h=(1,x_1)$ for some $x_1\in x^{\Sp_b(q)}$ with $\nu(x_1)=1$. By Proposition
\ref{prn:burnessbounds}, $|g^G|>\frac{1}{16}q^{a(d-a-1)}>q^{a^2(b-1)-a-4}$ while by Lemma
\ref{lem:upperbounds}, $|g^G\cap H|<q^b$.  Since $b\geq 4$, Lemma
\ref{lem:mvconjclass}(ii) implies that $G$ has a subdegree divisible by $p$ except when $(a,b)=(2,4)$. In this case, \cite[(15.1)]{asch} implies that $H$ is conjugate under a triality automorphism to the stabiliser of a nondegenerate 3-space when $q$ is odd,  while when $q$ is even $H$ is conjugate to a subgroup of the stabiliser of a nonsingular 1-space. Hence we have already seen that there is a subdegree divisible by $p$.


Finally suppose that $G$ is an orthogonal group and $H$ is of type
$O^{\epsilon_1}_a(q)\otimes O^{\epsilon_2}_b(q)$ with $q$ odd, $a,b\ge 3$ and
$(a,\epsilon_1)\neq(b,\epsilon_2)$. We may assume that $a\leq b$. Let $g=(1,x)\in H$ of order $p$ where $\nu(x)=2$. Then
$\nu(g)=2a$ and if $a<b$ then $h\in g^G\cap H$ if and only if $h=(1,x_1)$ for some $x_1\in
x^{O^{\epsilon_2}_b(q)}$ with $\nu(x_1)=2$. By Proposition \ref{prn:burnessbounds},
$|g^G|>\frac{q}{8(q+1)}q^{2a(ab-2a-1)}>q^{2a(ab-2a-1)-3}$, while by \cite[Proposition
  3.22]{burness}, $|g^G\cap H|<2q^{2b-4}<q^{2b-3}$. Since $a,b\geq 3$, it follows from
Lemma \ref{lem:mvconjclass}(ii) that $G$ has a subdegree divisible by $p$. It
remains to consider the case $a=b$ and $\epsilon_1\neq\epsilon_2$. Note this implies that
$a$ is even. If $(x_1,x_2)\in g^G\cap H$ then $\nu(x_i)\in\{0,2\}$. Thus $|g^G\cap H|< q^{4b-6}$. Again, as $a=b\geq 4$, Lemma
\ref{lem:mvconjclass}(ii) yields a $p$-subdegree.

In all cases we have found a $p$-subdegree and so as $p$ divides $|\Omega|$, $G$ is not \tont.

\subsection{Aschbacher class $\mathcal{C}_5$:}

Here $H$ is the stabiliser of a subfield structure for a subfield $\GF(q_0)$ where $q=q_0^r$ with $r$ a prime. Descriptions of these groups can be found in 
\cite[\S 4.5]{KL}.

Suppose first that $H$ is of the same Lie type as $G$ and we initially exclude the cases where
$L=\PSL_2(q)$ and $\PSU_3(q)$ with $H$ of the form $\PSL_2(3)$ and $\PSU_3(2)$
respectively. Note also that if $d=2$ then $q_0\neq 2$. Then $H$ has a weakly closed
normal subgroup $H_0$ which is an insoluble classical group. By Proposition
\ref{prn:sylpgen}, there exists a cyclic subgroup $T$ of order given by Table \ref{tab:T},
with $q$ replaced by $q_0$ such that given a Sylow $p$-subgroup $S$ of $H$, $H_0\leqslant
\la T,S\ra$. Moreover, $T$ is centralised by the subgroup $\overline{T}$ of $G$ given by Proposition
\ref{prn:sylpgen}. Hence $N_H(T)<N_G(T)$ and so by Lemma \ref{lem:weaklyclosed}, $G$ has a
subdegree divisible by $p$.

Now let $L=\PSL_2(q)$ and $q_0=3$. Let $x\in H$ be a transvection. Then
$|x^G|\geq \frac{1}{2}(q^2-1)$ while $|x^G\cap H|\leq 8$. It follows from Lemma
\ref{lem:mvconjclass}(ii) that when $q\geq 27$, the group $G$ has a subdegree divisible by $p$. When $q=9$, a Magma \cite{magma} calculation verifies that there is a subdegree of size 6. For $L=\PSU_3(q)$,
$q$ even, and $H\cap L=\PSU_3(2)\cong 3^2:Q_8$, let $x\in H\cap L$ be an
involution. Note also that $r$ is odd \cite[\S 4.5]{KL}, so $q\geq 8$. Then $|x^G\cap H|=9$ while by Proposition \ref{prn:burnessbounds},
$|x^G|>q^5/(2(q+1)^2)>202$. Hence by Lemma \ref{lem:mvconjclass}(ii), $G$ has a subdegree
divisible by $p=2$.

Next suppose that $H$ is of a different Lie type to $G$. Then by \cite[Table 4.5A]{KL} the possibilities are:
\begin{itemize}
\item $L=\PSU_d(q)$ and $H$ of type $O^{\epsilon}_d(q)$ with $\epsilon\in\{+,-,\circ\}$ and $q$ odd.
\item $L=\PSU_d(q)$ and $H$ of type $\Sp_d(q)$.
\item $L=\POmega^+_d(q)$ and $H$ of type $O^-_d(q^{1/2})$.
\end{itemize}
Then if $q$ is odd and $H$ is not of type $O_3(3)$ or $O^+_4(3)$, $H$ contains a weakly
closed normal subgroup $H_0$ which is a nonlinear classical group and by \cite{odddegree},
$|G:H|$ is even. Hence by Lemma \ref{lem:23/4}, $G$ has a subdegree divisible by $p$.

If $L=\PSU_3(3)$ and $H\cap L\cong \PSO_3(3)$ we see from \cite[p14]{atlas} that $H$ is
not maximal in $G$. Similarly, for $L=\PSU_4(3)$ and $H\cap L\cong \PSO^+_4(3).2$,
\cite[p52]{atlas} shows that $H$ is not maximal in $G$.

Next suppose that $L=\PSU_d(q)$ for $q$ even and 
$H=\PSp_d(q)$ (see \cite[Proposition 4.5.6]{KL}).  
By Proposition \ref{prn:sylpgen}, $H$ has a cyclic subgroup $T$ of order
$q^{d/2}+1$ such that $H=\la S,T\ra$ for any Sylow 2-subgroup $S$ of $H$.  Now $T$ is self-centralising in $H$ but we claim that $T$ is contained in a torus $\overline{T}$ of order $q^d-1$ in $L$. If $d/2$ is odd one can see this by viewing $T < \Sp_2(q^{d/2}) < \GU_2(q^{d/2}) < \GU_d(q)$: there is a torus $\overline{T}$ of order $q^d-1$ in $\GU_2(q^{d/2})$ containing $T$, and the $\Sp_2(q^{d/2})$ is in $\Sp_d(q) < \GU_d(q)$; 
and if $d/2$ is even then if one takes an element in $\GU_d(q)$ of 
order a primitive prime divisor\footnote{A \emph{primitive prime divisor} of $q^d-1$ is a prime divisor of $q^d-1$ that does not divide $q^i-1$ for any $i<d$.} of $q^d-1$, its centralizer in $\Sp_d(q)$ is $T$ and its centralizer in $\GU_d(q)$ is $\overline{T}$.   Thus in both cases, $N_H(T)<N_G(T)$ and so by Lemma \ref{lem:weaklyclosed}, $G$ has a subdegree divisible by $p$.  

Finally suppose $L=\POmega^+_d(q)$ for $q$ even and 
$H=\POmega^-_d(q^{1/2})$ (see \cite[Proposition 4.5.10]{KL}).  
By Proposition \ref{prn:sylpgen}, $H$ has a cyclic subgroup $T$
of order $q^{d/4}+1$ such that $H=\la S,T\ra$ for any Sylow 2-subgroup $S$ of $H$. Now a Singer cycle of order $(q^{1/2})^{d/2}+1$ is $\Omega^-_2((q^{1/2})^{d/2})$. Moreover, $\Omega^+_d(q)$ has a $\mathcal{C}_3$-subgroup $\Omega^+_2(q^{d/2})$, which is a torus of order $q^{d/2}-1$, and this contains $\Omega^-_2((q^{1/2})^{d/2})$, which is the original Singer cycle. Since $d\geq 8$ it follows that $T$ is self-centralising in $\Omega^-_d(q)$ but not in $G$. Thus $N_H(T)<N_G(T)$ and so by Lemma \ref{lem:weaklyclosed}, $G$ has a subdegree divisible by $p$.

In all cases we have found a $p$-subdegree and $|\Omega|$ is divisible by $p$. Hence $G$
is not \tont.

\subsection{Aschbacher class $\mathcal{C}_6$:}

Here $H$ is the normaliser of a symplectic-type $r$-group. Descriptions 
can be found in \cite[\S 4.6]{KL}.

First we deal with the case where $d=2$. Here $H\cap L=A_4$ or $S_4$. Moreover, $q=p\geq
5$ and so $p$ does not divide $|H|$. We note first that if $p=5$ or $7$ then $G$ is
2-transitive and so we assume that $q\geq 11$. Let $x\in H$ have order 3 and note that $3$
divides either $q-1$ or $q+1$. Then $|x^G|=q(q\pm1)$ while $|x^G\cap H|=8$. Hence by Lemma
\ref{lem:mvconjclass}(ii), $G$ has a subdegree divisible by 3. Let $S$ be a Sylow
3-subgroup of $H$. Then $N_H(S)\leqslant S_3$ while $S_3<D_{q-\epsilon}\leqslant N_G(S)$. Thus
there exists $g\in G\backslash H$ such that $S\leqslant H\cap H^g$ and so $G$ also has a subdegree not
divisible by 3, showing that $G$ is not \tont.

For $d>2$ and $L=\PSL^{\epsilon}_d(q)$, there are two types of $\mathcal{C}_6$ subgroups:
those of type $r^{1+2m}.\Sp_{2m}(r)$ for $r$ odd, $q\equiv \epsilon\pmod r$ and $d=r^m$, and
those of type $(4\circ 2^{1+2m}).\Sp_{2m}(2)$ for $q=p\equiv \epsilon\pmod 4$ and
$d=2^m$. In the first case, $q=p^f$ where $f$ is the smallest integer such that $p^f\equiv
\epsilon\pmod r$. By \cite[Propositions 4.6.5]{KL}, if $H$ is of the first type then
$$H\cap L=\left\{\begin{array}{ll}
                  3^2.Q_8 & d=3 \text{ and } q\equiv \epsilon4 \text{ or }\epsilon7\pmod 9\\
                  r^{2m}.\Sp_{2m}(r) & \text{otherwise} 
              \end{array}\right.$$
If $H$ is of type $(4\circ 2^{1+2m}).\Sp_{2m}(2)$, then by \cite[Proposition
4.6.6]{KL},
$$H\cap L=\left\{ \begin{array}{ll}
                2^4.A_6 & d=4 \text{ and } p\equiv \epsilon5\pmod 8 \\
                2^{2m}.\Sp_{2m}(2)& \text{otherwise}
                \end{array}\right.$$

Suppose first that $H$ is of type $(4\circ 2^{1+2m}).\Sp_{2m}(2)$. Since $q=p$, $p$ divides
$|H|$ if and only if $p$ divides $|H\cap L|$.  Suppose then that $p$ divides $|H\cap L|$
and $x\in H\cap L$ has order $p$. If $m\geq 4$, then \cite[Lemma 6.3]{burness} implies
that $\nu(x)\geq 4$. Then by Proposition \ref{prn:burnessbounds},
$|x^G|>\frac{q}{2(q-\epsilon)(q+1)}q^{8(d-4)}>q^{2^{m+3}-35}$. Now $|H\cap
L|^2<2^{4m^2+6m}$ and so $|x^G|>|x^G\cap H|^2$. Thus by Lemma \ref{lem:mvconjclass}(ii),
$G$ has a $p$-subdegree. For $m=3$, \cite[Lemma 6.3]{burness} implies that $\nu(x)\geq 2$
and so $|x^G|>\frac{q}{2(q-\epsilon)(q+1)}q^{4(d-2)}>q^{2^5-11}=q^{21}$. Also $H\cap
L=2^6.\Sp_6(2)$ and so for $q\geq 7$, we have $|x^G|>|x^G\cap H|^2$. Thus Lemma
\ref{lem:mvconjclass}(ii) once again yields a $p$-subdegree. It remains to consider $(\epsilon,d,q)=(+,8,5)$
or $(-,8,3)$. When $q=3$, \textsc{Gap} \cite{gap} calculations show that $H\cap L$ contains precisely 143360 elements $x$ of order 3 with $\nu(x)=5$. By Proposition \ref{prn:burnessbounds}, for such an element $x$ we have $|x^L|> 3^{37}>|x^L\cap H|^2$ and so Lemma \ref{lem:mvconjclass} implies the existence of a subdegree divisible by $p=3$. When $q=5$, \textsc{Gap} \cite{gap} calculations show that there is a regular suborbit and hence one divisible by $p=5$. (In both the $p=3$ and 5 cases we can construct $H\cap L$ using the algorithm outlined in \cite[Section 9]{holtRD}. To find the existence of a regular suborbit we simply choose random elements $g\in L$ until we find one with $H\cap H^g\cap L=1$.)
 Finally, if $m=2$ note that $H\cap L\leqslant 2^4.\Sp_4(2)$,
and so if $|H\cap L|$ is divisible by $p$ then either $\epsilon=+$ and $p=5$, or
$\epsilon=-$ and $p=3$.  For $L=\PSL_4(5)$, we have $H\cap L=2^4.A_6$ and using \textsc{Magma} \cite{magma} we see that the subdegrees for $L$ are
$$1, 16^2, 30, 80, 96^2, 120, 160,  240^6, 320, 360, 480^6, 960^{13}, 1152^{10}, 1440^{10}, 1920^{21}, 2880^{28}, 5760^{190}$$
Thus $L$, and hence $G$, has a subdegree divisible
by $p=5$. For $L=\PSU_4(3)$, a \textsc{Magma} \cite{magma} calculation shows that the subdegrees
for $L$ are $30$, $96$, $120$ and $320$. Thus we have found a subdegree divisible by 3.

To show that $G$ is not \tont, note that $H$ has a normal 2-subgroup and so Lemma
\ref{lem:pnorml}(iv) implies that $G$ has an even subdegree. By \cite{odddegree},
$|\Omega|$ is even and hence $G$ is not \tont.

Next suppose that $H$ is of type $r^{1+2m}.\Sp_{2m}(r)$. Now $q=p^f$ where $f$ is the
smallest integer such that $p^f\equiv \epsilon\pmod r$. Hence $f$ divides $r-1$.  It
follows that if $p$ divides $|H|$ then $p$ divides $|H\cap L|$. Let $x\in H\cap L$ have
order $p$. If $(r,m)=(3,1)$ then the divisors of $|H\cap L|$ imply that $p=2$. Hence
$(\epsilon,q)=(+,4)$ or $(-,2)$. The latter is not possible as we have excluded
$\PSU_3(2)$. For $L=\PSL_3(4)$, a \textsc{Magma} \cite{magma} calculation shows that the subdegrees
are $\{ 1, 72^3, 18^3, 9 \}$ which include subdegrees divisible by $3$. Thus we may
assume that $d\geq 5$ and so by \cite[Lemma 6.3]{burness}, $\nu(x)\geq 2$. Hence by
Proposition \ref{prn:burnessbounds}, $|x^G|>q^{4.r^m-12}$. Now $H\cap L\leqslant
r^{2m}.\Sp_{2m}(r)$ and so $|x^G\cap H|^2<r^{4m^2+6m}$. Now $r\leq q+1$ so $|x^G\cap H|^2<q^{8m^2+12m}$. Thus for $r\geq 11$, Lemma \ref{lem:mvconjclass}(ii) yields a $p$-subdegree. For $r=7$, note that $q\geq 8$ and so $|x^G|>|x^G\cap H|^2$ for all values of $m$. Hence  Lemma \ref{lem:mvconjclass}(ii) then yields a $p$-subdegree. For $r=5$ we have $q\geq 4$ and we see that $|x^G|>|x^G\cap H|^2$ except when $(q,m)=(4,1)$. In this case $H\cap L=5^2.\Sp_2(5)$, which contains a unique conjugacy class of involutions. Moreover, such involutions have centraliser $\Sp_2(5)$ in $H\cap L$. Hence $|x^G\cap H|^2=5^4<|x^G|$ and Lemma  \ref{lem:mvconjclass}(ii) again yields a $p$-subdegree. Finally, suppose that $r=3$ and note that $d\geq 5$ implies that $m\geq 2$. If $q\geq 4$ then $|x^G|>4^{4.3^m-12}>3^{4m^2+6m}>|x^G\cap H|^2$ and so Lemma \ref{lem:mvconjclass}(ii) yields a $p$-subdegree. If $q=2$ then Lemma \ref{lem:mvconjclass}(ii) yields a $p$-subdegree except in the case where $m=2$. In this case $H\cap L=3^4.\Sp_4(3)$ and $L=\PSU_9(2)$. Then $\Sp_4(3)$ contains a central element of order 2 and 90
elements of order 2 whose $-1$ eigenspace has dimension 2. Hence $|x^G\cap
H|<3^4+90.3^2$. Thus $|x^G\cap H|^2<2^{24}<|x^G|$ and so by Lemma
\ref{lem:mvconjclass}(ii), $G$ has a subdegree divisible by $p=2$.





We now consider $\frac{3}{2}$-transitivity. For $q$ even we have just shown that we can always find an even subdegree. Since
$|\Omega|$ is even this implies that $G$ is not \tont. Hence we may assume that $q$ is
odd. Suppose first that $d=r=3$ and $q\equiv \epsilon4,\epsilon7\pmod 9$.  Here $H\cap
L=3^2.Q_8$. Let $x\in H\cap L$ be an involution. Then $|x^G\cap H|=9$. There is a unique
conjugacy class of involutions in $\PSL^{\epsilon}_3(q)$ and $|x^L|=q^2(q^2+\epsilon
q+1)$. Since $L\neq \PSU_3(2)$, it follows that $|x^G|>|x^G\cap H|^2$ and so Lemma
\ref{lem:mvconjclass}(ii) implies that there is an even subdegree. By \cite{odddegree},
$|\Omega|$ is even and so $G$ is not \tont.

If $d>3$, or $d=3$ and $q\equiv \epsilon\pmod 9$, note that by Lemma \ref{lem:pnorml}(iv),
$G$ has a subdegree divisible by $r$.  The order of $L$ is divisible by
$(q-\epsilon)^{d-2}$ and hence divisible by $r^{d-2}$. Moreover, since $q\equiv
\epsilon\pmod r$, $(q^r-\epsilon)/(q-\epsilon)$ is also divisible by $r$. Hence for each
term $q^{ir}-1$ with $i$ even we get an extra $r$ dividing $|L|$ and for each term
$q^{ir}-\epsilon$ with $i$ odd we also get an extra $r$ dividing $|L|$. Since $d=r^m$ it
follows that $|L|$ is divisible by $r^{r^m-2}r^{r^{m-1}}=r^{r^m+r^{m-1}-2}$. On the other
hand the largest power of $r$ dividing $|H\cap L|$ is $r^{m^2+2m}$. Hence $|\Omega|$ is
divisible by $r$ and so $G$ is not \tont.

Next suppose that $L=\PSp_d(q)$ where $q=p$ and $d=2^m\geq 4$ and $H$ is of type
$2_{-}^{1+2m}.O^-_{2m}(2)$. Suppose that $p$ divides $|H|$.  Then $p$ divides $|H\cap L|$
and so let $x\in H\cap L$ have order $p$. We first suppose that $m\geq 4$ and so by
\cite[Lemma 6.3]{burness} we have that $\nu(x)\geq 4$. Then by Proposition
\ref{prn:burnessbounds}, it follows that
$|x^G|>\frac{q}{4(q+1)}q^{4(d-4)}>q^{2^{m+2}-18}$. By \cite[Proposition 4.6.9]{KL}, $H\cap
L\cong 2^{2m}.\Omega^-_{2m}(2).c$ where $c=1$ if $p\equiv \pm 3\pmod 8$ and $c=2$ if
$p\equiv \pm 1\pmod 8$. Hence
$$|H\cap L|=2^{c-1}.2^{2m}.2^{m(m-1)}(2^m+1)\prod_{i=1}^{m-1}(2^{2i}-1) < 2^{2m^2+m+2}$$
Thus $|H\cap L|^2<q^{4(2m^2+m+2)/3}<|x^L|$ for $m\geq 5$. For $m=4$, explicitly
calculating $|H\cap L|$ also yields $|H\cap L|^2<|x^G|$. Hence there is a
$p$-subdegree by Lemma \ref{lem:mvconjclass}(ii).

When $m=3$ we have that $|H\cap L|$ divides $2^{13}.3^4.5$ and so $p=3$ or 5. Thus $H\cap
L=2^6.\Omega^-_6(2)$.  When $p=3$, a \textsc{Magma} \cite{magma} calculation shows that $H\cap L$
contains 5120 elements $x$ of order $3$ such that $\nu(x)=5$. Thus by Proposition
\ref{prn:burnessbounds}, $|x^G|>\frac{1}{16}3^{21}$ and so by Lemma
\ref{lem:mvconjclass}(ii), $G$ has a subdegree divisible by $3$. For $p=5$, a \textsc{Magma}
\cite{magma} calculation shows that $H\cap L$ contains 82944 elements of order 5 and for
each such element $x$, $\nu(x)= 6$. Thus by Proposition \ref{prn:burnessbounds},
$|x^G|>\frac{1}{24}5^{25}$ and so by Lemma \ref{lem:mvconjclass}(ii), $G$ has a subdegree
divisible by $5$.

For $m=2$, $|H\cap L|$ divides $2^7.3.5$ and so $p=3$ or 5. Hence $H\cap
L=2^4.\Omega^-_4(2)$.  When $p=5$ a \textsc{Magma} \cite{magma} calculation shows that the
subdegrees for $L$ are
$$1, 10, 40, 80, 120, 160^2, 192^2,  240, 320, 480, 960^3$$
of which many are divisible by $p=5$.  When $p=3$ we see that the
subdegrees of $L$ are 1, 16 and 10, none of which are divisible by 3; 
this is an example in line 6 of Table \ref{tab:classgp} of Theorem \ref{thm:classgroups}(A).

Finally suppose that $L=\POmega^+_d(q)$ where $q=p$ and $H$ is of type
$2_{+}^{1+2m}.O^+_{2m}(2)$.  Suppose that $p$ divides $|H|$.  Then $p$ divides $|H\cap L|$
and so let $x\in H\cap L$ have order $p$. By \cite[p512]{asch}, if $m=3$ such subgroups
are conjugate under a triality automorphism to the stabiliser of a 1-space decomposition
and so this case has already been dealt with in Section \ref{sec:C2}. Thus we may assume that $m\geq 4$. Then by
\cite[Lemma 6.3]{burness} we have that $\nu(x)\geq 4$. Then by Proposition
\ref{prn:burnessbounds} it follows that
$|x^G|>\frac{q}{8(q+1)}q^{4(d-4-1)}>q^{2^{m+2}-23}$. By \cite[Proposition 4.6.9]{KL},
$H\cap L\cong 2^{2m}.\Omega^+_{2m}(2).c$ where $c=1$ if $p\equiv \pm 3\pmod 8$ and $c=2$ if
$p\equiv \pm 1\pmod 8$. Thus $|H\cap L|^2\leq2^{4m^2+4m+2}<q^{4(2m^2+m+1)/3}<|x^G|$ for
$m\geq 5$. For $m=4$ and $q\geq 5$, explicitly calculating $|H\cap L|$ also yields $|H\cap
L|^2<|x^G|$. Thus we have found $p$-subdegrees except in the case where $m=4$ and
$q=p=3$. Since $\Omega^+_8(2)$ contains $365120$ elements of order 3, $H\cap L$ contains
at most $2^8$ times this number. Thus $|x^G\cap H|^2<|x^G|$ and we have once again found a
$p$-subdegree.

To show that $G$ is not \tont\ when $L=\PSp_d(q)$ or $\POmega^+_d(q)$, note that Lemma
\ref{lem:pnorml}(iv) implies that $G$ has an even subdegree and by \cite{odddegree},
$|\Omega|$ is even.

\subsection{Aschbacher class $\mathcal{C}_7$:}

Here $H$ is the stabiliser of a tensor product decomposition $V=U_1\otimes U_2\otimes
\cdots\otimes U_t$ where each $U_i$ has dimension $m$ and $d=m^t$ with $t\ge 2$. Descriptions of the groups can be found in \cite[\S 4.7]{KL}.

Suppose first that $L=\PSL_d(q)$ or $\POmega^\epsilon_d(q)$ with the additional assumption that if $L$ is an  orthogonal group then each $U_i$ is also orthogonal. Then $|H\cap L|\leq |\GL_m(q)|^tt!< q^{tm^2}t!<q^{3m^t}.$  By \cite[Tables 3.5A and 3.5E]{KL}, $m\geq 3$ and when $G$ is orthogonal $q$ is odd. Let $x=([J_3,I_{m-3}],I_m,\ldots,I_m)\in H\cap L$, where 
$$J_3=\begin{pmatrix} 1&0&0\\1&1&0\\0&1&1\end{pmatrix}.$$ Then $\nu(x)=2m^{t-1}$ and so by  Proposition \ref{prn:burnessbounds}, $|x^G|>\frac{q}{q^4}(q^{m^t(2m^{t-1}-1)/2})=q^{m^t(2m^{t-1}-1))/2-3}$. Now $6m^t\leq m^t(2m^{t-1}-1)/2-3$ if and only if $6<(2m^{t-1}-1)/2$. Thus for $t\geq3$ or $t=2$ and $m\geq 7$  we have that $|x^G|>|H\cap L|^2$ and so by Lemma \ref{lem:mvconjclass} we have a subdegree divisible by $p$. For $t=2$ we in fact have that $|H\cap L|<2q^{2m^2}< q^{2m^2+1}$ and $|x^G|>q^{m^2(2m-1)/2-3}$ and so $|x^G|>|H\cap L|^2$ for $m\geq 5$. For $m=3$ we have $|x^G|>q^{39/2}$ and for $m=4$ we have $|x^G|>q^{53}$. Using $|O_3(q)|<q^4$ and $|O^{\epsilon}_4(q)|<q^7$ we are then able to show that in the remaining orthogonal cases we also have $|x^G|>|H\cap L|^2$. For $L=\PSL_{16}(q)$ note that by Proposition \ref{prn:burnessbounds} we actually have $|x^G|>q^{126}$ and this yields $|x^G|>|H\cap L|^2$. Finally, for $L=\PSL_9(q)$, consider the element $x=([J_2,I_1],I_3)\in H\cap L$. Then $\nu(x)=3$ and by Proposition \ref{prn:burnessbounds} we have $|x^G|>q^{34}$. Now if $y\in x^G\cap H$ then either $y=(g_1,g_2)$ where $g_i\in\GL_3(q)$ and is a transvection or the identity, or $y=(g,g^{-1})\sigma$ where $g\in\GL_3(q)$ and $\sigma$ interchanges the two factors of the tensor decomposition.  Hence by Lemma \ref{lem:upperbounds} we have $|x^G\cap H|<(q^5)^2+q^9<2q^{10}\leq q^{11}$. Thus $|x^G|>|x^L\cap H|^2$ and so by Lemma \ref{lem:mvconjclass} there exists a subdegree divisible by $p$.

Next suppose that $L=\PSp_d(q)$ or $\POmega^+_d(q)$ with each $U_i$ a symplectic space. Then $|H\cap L|\leq |\Sp_m(q)|^tt!<q^{(m^2+m)t/2}t!<q^{2m^t}$. Let $x=([J_2,I_{m-2}],I_m,\ldots,I_m)\in H\cap L$ have order $p$. Now $\nu(x)=m^{t-1}$ and so by  Proposition \ref{prn:burnessbounds}, $|x^G|>\frac{1}{q^2}q^{m^{t-1}(m^t-m^{t-1}-1)}=q^{m^{t-1}(m^t-m^{t-1}-1)-2}$. Now $4m^t\leq m^{t-1}(m^t-m^{t-1}-1)-2$ if $t\geq3$ and $m\geq 3$, or $t=2$ and $m\geq 6$. Hence for these values of $m$ and $t$ we have $|x^G|>|H\cap L|^2$ and so by Lemma \ref{lem:mvconjclass} there is a subdegree divisible by $p$. By \cite[Table 3.5C]{KL}, if $t=2$ then $L=\POmega^+_d(q)$. Since $d\geq 8$ in this case it follows that $m\neq 2$. When $(m,t)=(4,2)$  we in fact have that $|H\cap L|<q^{21}$ while $|x^G|>q^{42}$ and so we can again use Lemma \ref{lem:mvconjclass} to find a subdegree divisible by $p$. We are left to consider the case where $m=2$ and $t=3$. By \cite[Tables 3.5C and 3.5E]{KL} we have that $L=\PSp_8(q)$ when $q$ is odd and $\POmega^+_8(q)$ when $q$ is even.  However, when $q$ is even \cite{asch} implies that $H$ is not maximal in $G$. Hence $q$ is odd. Then Proposition \ref{prn:burnessbounds} implies that $|x^G|>q^{30}$ while $|H\cap L|<6q^9<q^{11}$. Hence $|x^L|>|H\cap L|^2$ and Lemma \ref{lem:mvconjclass} implies the existence of a subdegree divisible by $p$. 

This leaves us to consider the case where $L=\PSU_d(q)$. By \cite[Table 3.5B]{KL} we have $m\geq 3$.  Let $x=([J_3,J_1^{m-3}],I_m,\ldots,I_m)\in H\cap L$. Then $\nu(x)=2m^{t-1}$ and so by  Proposition \ref{prn:burnessbounds}, $|x^G|>q^{4m^{t-1}(m^t-2m^{t-1})-4}=q^{4m^{2t-2}(m-2)-4}$. Moreover, $|H\cap L|\leq |\GU_m(q)|^tt!<q^{(m^2+m)t+t\log_2t}<q^{4m^t}$. Now $8m^t\leq 4m^{2t-2}(m-2)-4$ if and only if $8<4m^{t-2}(m-2)$. Thus for $t\geq 3$ or $t=2$ and $m\geq 5$ we have that  $|x^L|>|H\cap L|^2$. For $t=2$ we in fact have that $|H\cap L|<2q^{2(m^2+m)}<q^{2(m^2+m)+1}$ and $|x^G|>q^{2m^3-4}$ and hence $|x^G|>|H\cap L|^2$ for $m\geq 3$. Hence Lemma \ref{lem:mvconjclass} implies that there is a subdegree divisible by $p$.  

In all cases we have found a subdegree divisible by $p$ and so as $p$ divides $|\Omega|$ (Lemma \ref{lem:Tits}) it follows that $G$ is not \tont.

\subsection{Aschbacher class $\mathcal{C}_8$:}

Here $H$ is a classical group on $V$. Descriptions can be found in 
\cite[\S 4.8]{KL}.  
If $L\neq \PSL_d(q)$ then the only cases which occur are
for $L=\PSp_d(q)$ ($q$ even) with $H$ an orthogonal group. When $q=2$ and $d\geq 6$ the two actions (of $\Sp_{d}(2)$ on cosets of $O^\pm_d(2)$) are
2-transitive, as in Theorem \ref{thm:classgroups}(A)(ii). For $(d,q)=(4,2)$, then the action of $\Sp_4(2)'$ on the cosets of $O^\pm_4(2)\cap\Sp_4(2)'$ is also 2-transitive and we have lines 1 and 2 of Table \ref{tab:classgp} of Theorem \ref{thm:classgroups}(A).
For $q>2$, from the proof of \cite[Prop. 1]{LPSclosures} there is
a unique suborbit of length $(q^d-\epsilon)(q^{d-1}+\epsilon)$ and $(q-2)/2$ of length
$q^{d-1}(q^d-\epsilon)$. Hence there is a $p$-subdegree and $G$ is not \tont.

For $L=\PSL_d(q)$ first we deal with the cases where $H$ is soluble. These are
$\SU_3(2)\leqslant \SL_3(4)$, $\SO_3(3)\leqslant
\SL_3(3)$, $\SO^+_4(3).2\leqslant \SL_4(3)$.  It is not difficult to find the subdegrees
directly by computer for $L$ in these cases: 
\begin{center}
\begin{tabular}{ccl}
\hline
$G$&$H$&Subdegrees for $L$\\
\hline
$\SL_3(4)$&$\SU_3(2)$&$1$, $9$, $18^3$, $72^3$\\
$\SL_3(3)$&$\SO_3(3)$&$1$, $3$, $4^2$, $6$, $12^8$,  $24^5$\\
$\SL_4(3)$&$\SO^+_4(3).2$& $1$, $9$, $18^2$, $24^2$, $32$, $36^3$, $48^2$, $72^5$,  $144^9$, $192$, $288^{17}$,  $576^6$\\
\hline
\end{tabular}
\end{center}

We see that in each example above, $G$ has a subdegree
divisible by $p$, and $G$ is not \tont. Moreover, any overgroup will also
have an even subdegree and a subdegree divisible by $p$, so will not be \tont.

{}From now on we may assume that $H$ is insoluble. Let $S$ be a Sylow $p$-subgroup of
$G$. By Proposition \ref{prn:sylpgen}, we can find a cyclic subgroup $T$ of $H$ which is
either irreducible on $V$, or acts irreducibly on a hyperplane or codimension 2 subspace
and trivially on the perp, such that $\la S,T\ra$ contains a weakly closed normal subgroup of
$H$.  Moreover, we easily see that $N_{H}(T)<N_{G}(T)$ and so Lemma
\ref{lem:weaklyclosed} implies that there is a subdegree divisible by $p$. Since
$|\Omega|$ is divisible by $p$, it follows that $G$ is not \tont.

\subsection{$\mathcal{C}_9$ groups}

We use the following theorem of Burness, Guralnick and Saxl \cite{base2}.

\begin{theorem}
\label{thm:C9}
  Let $G$ be an almost simple classical group with socle $L$ and let
  $H\in\mathcal{C}_9$. Moreover, we suppose that $L$ is not isomorphic to an alternating group and the action of $G$ on the set of right cosets of $H$ is not permutation isomorphic to a classical group acting on the set of right cosets of a $\mathcal{C}_i$-subgroup with $1\leq i\leq 8$. Then either the action of $G$ on the set of cosets of $H$ has a regular suborbit or $(L,H\cap L)$ is given by Table \ref{tab:notbase2}.
\end{theorem}

We note that since $G$ is not a Frobenius group, the existence of a regular suborbit implies that $G$ is not \tont. Moreover, if $p$ divides $|H|$ then the regular suborbit has length divisible by $p$. 

The $\mathcal{C}_9$-subgroups excluded by hypothesis from Theorem \ref{thm:C9} are listed in \cite[Table 2]{base2}. When the action of $G$ is isomorphic to a classical group with a $\mathcal{C}_i$-action for $i\leq 8$, the classical group has the same characteristic as $G$, so $\frac{3}{2}$-transitivity and the existence of a subdegree divisible by $p$ has already been determined in the previous sections. When $L$ is isomorphic to an alternating group, since Section \ref{an} has already considered when such groups can be \tont, it remains to check for $p$-subdegrees. The actions under consideration here are when $L=\PSL_2(9)$ of degree 6 on the set of cosets of $N_G(A_5)$, and $L=\PSL_4(2)$ of degree 8 on the cosets of $N_G(A_7)$. Both groups are 2-transitive and the unique nontrivial subdegree is not divisible by $p$. These provide the examples in lines 4 and 5 of Table \ref{tab:classgp} in Theorem \ref{thm:classgroups}(A).

\begin{table}[ht]
\caption{$\mathcal{C}_9$ actions without a regular suborbit}
\label{tab:notbase2}
\begin{tabular}{lll}
\hline
$L$ & $H\cap L$& Subdegrees of $L$\\
\hline
$\Omega_7(q)$ & $G_2(q)$, $q\equiv \epsilon\pmod 4$, $\epsilon=\pm1$ &  $q^6-1$, $q^3(q^3+\epsilon)$, $(q-4-\epsilon)/4$ times $q^3(q^3+1)$,\\ &&$(q-2+\epsilon)/4$ times $q^3(q^3-1)$ \\
$\Omega_7(q)$ & $G_2(q)$, $q$ even& $q^6-1$, $q/2$ times $q^3(q^3-1)$, $(q-2)/2$ times $q^3(q^3+1)$ \\
$\PSp_4(q)$, $q$ even & $Sz(q)$ & $(q-1)(q^2+1)$, 2 times $\frac{1}{2}q(q-1)(q^2+1)$,\\
 && $(q-2)/2$ times $q^2(q^2+1)$, \\ 
 && $(q\pm \sqrt{2q})/4$ times $q^2(q-1)(q\pm\sqrt{2q} +1)$\\
$\PSL_3(4)$ & $A_6$ & $1, 10, 45$\\
$\PSL_2(19)$ & $A_5$ & $1, 6, 20, 30$\\
$\PSL_2(11)$ & $A_5$ & $1, 10$\\
$\PSU_6(2)$ & $\PSU_4(3).2$& $1, 567, 840$\\
            &$M_{22}$ & $1, 77, 231, 1155, 1232, 2640, 6160, 9240$\\
$\PSU_4(3)$ &$\PSL_3(4)$ & $1, 56, 105$ \\
            &$A_7$ &  $1, 105, 140, 210^2, 630$\\
$\PSU_3(5)$ & $A_6.2$ & $1, 12, 90, 72$\\
            & $A_7$ & $1, 7, 42$\\
            & $\PSL_3(2)$ & $1, 14^2, 21, 28, 56^3, 84^4, 168$\\
$\PSU_3(3)$ & $\PSL_3(2)$& $1, 7^2, 21$\\
$\Sp_8(2)$ & $S_{10}$&$1,210, 1575, 5600, 5670$\\
$\Sp_6(2)$ & $\PSU_3(3).2$&$1,  56, 63$\\
$\Omega^+_{14}(2)$ & $A_{16}$ &subdegrees include $130767436800, 290594304000,$\\
      && $435891456000, 653837184000, 871782912000,$\\
     && $1307674368000, 2615348736000, 3487131648000,$\\
      && $ 5230697472000, 10461394944000$\\
$\Omega^-_{12}(2)$ & $A_{13}$ & subdegrees include $3603600, 14414400, 16216200,$\\
&& $ 21621600, 28828800, 32432400, 43243200,$\\
 && $48648600, 64864800, 86486400, 97297200, 129729600,$ \\
 && $194594400, 259459200, 389188800, 518918400,$\\
 && $ 778377600, 1556755200$\\
$\Omega^-_{10}(2)$ & $A_{12}$&$1, 462, 25202, 5775, 30800, 62370$\\
$\POmega^+_8(3)$ & $\Omega^+_8(2)$ &$1, 960^3, 3150, 22400$\\
$\Omega^+_8(2)$ & $A_9$& $1, 84, 315, 560$\\
$\Omega_7(3)$ & $S_9$&$1, 126, 315, 560, 1680, 2520^2, 4320, 5670, 7560$  \\
             &$\Sp_6(2)$&$1,288, 630, 2240$\\
\hline
\end{tabular}
\end{table}

It remains to consider the actions listed in Table \ref{tab:notbase2}.
The subdegrees for the infinite families are given in \cite[Proposition 2]{LPSclosures} 
and \cite[Theorem A]{LawtherSaxl}, and we see that there are subdegrees divisible by $p$ and that $G$ is not \tont. For the remaining cases, the subdegrees were calculated using
\textsc{Magma} and are given in Table \ref{tab:notbase2}. The table shows that there
are $p$-subdegrees in all cases where $p$ divides $|H|$ except for $L=\PSU_3(5)$ acting on $A_7$, which is in line 3 of Table \ref{tab:classgp} of Theorem \ref{thm:classgroups}(A). Moreover, none of the groups are \tont.

%
%

\subsection{Graph automorphisms}
\label{sec:graph}

When $L = \PSL_d(q)$, $\Sp_4(q)$ ($q$ even) or $\POmega^+_8(q)$ and
$G$ contains a graph automorphism (a triality in the latter case), 
there are extra maximal subgroups which we need to
consider, and we do so in this section.

\subsubsection{Linear groups} 

When $L=\PSL_d(q)$ and $G$ contains a graph automorphism there are two extra classes of
subgroups that we need to consider.

First let $\Omega$ be the set of all pairs $\{U,W\}$ of complementary subspaces of $V$
with $\dim(U)=m<d/2$. Let $x=\{\la v_1,\ldots,v_m\ra,\la v_{m+1},\ldots,v_d\ra\}\in\Omega$
and $H=G_x$. Then $H\cap\PGammaL_d(q)$ is the stabiliser of the decomposition $\la
v_1,\ldots,v_m\ra\oplus \la v_{m+1},\ldots,v_d\ra$. Now $y=\{\la
v_{m+1},\ldots,v_{2m}\ra,\la v_1,\ldots,v_m,v_{2m+1},\ldots,v_d\ra\}\in\Omega$ and
$H_y\cap\PGammaL_d(q)$ is the stabiliser of the decomposition $\la
v_1,\ldots,v_m\ra\oplus\la v_{m+1},\ldots,v_{2m}\ra\oplus \la
v_{2m+1},\ldots,v_d\ra$. Hence $|H:H_y|$ is divisible by $p$ and as $p$ divides
$|\Omega|$, it follows that $G$ is not \tont.

Finally, let $m<d/2$ and let $\Omega$ be the set of pairs of subspaces $(U,W)$ of
complementary dimensions with $\dim(U)=m$ and $U < W$.  By Lemma \ref{unique}, there
is a unique subdegree which is a power of $p$. Since $G$ is not 2-transitive, it is not \tont.

\subsubsection{4-dimensional symplectic groups}

When $q$ is even, $\Sp_4(q)$ has a graph automorphism which interchanges totally isotropic
1-spaces and totally isotropic 2-spaces.  Since $\Sp_4(2)'\cong A_6$ we have already checked for $\frac{3}{2}$-transitivity in this case. It is straighforward to check that for all primitive groups with socle $\Sp_4(2)'$ the only ones with no even subdegrees are those of degrees 6 and 10, and these correspond to $H$ being an orthogonal group, that is, we have the examples in lines 1 and 2 of Table \ref{tab:classgp} in Theorem \ref{thm:classgroups}(A).  From now on we assume that $q>2$.
Aschbacher \cite{asch} gives three further classes of maximal subgroups of a group $G$
with socle $L=\PSp_4(q)$ for $q>2$ even, and containing a graph automorphism. These are as
follows:

\begin{enumerate}
\item $\mathcal{C}_1'$: stabilisers of pairs $\{U,W\}$ of subspaces of $V$ such that $U$ is a totally
  isotropic 1-space and $W$ is a totally isotropic 2-space containing $U$.

\item $N_G(X)$ where $X$ is the stabiliser in $O^+_4(q)$ of a 
nondegenerate 2-space; here $N_L(X) = D_{2(q\pm 1)} \Wr S_2$. 

\item $N_G(C_{q^2+1})$: here $H\cap L = C_{q^2+1}\rtimes C_4$. 
\end{enumerate}

If $H\in\mathcal{C}_1'$, then $H$ has a normal 2-subgroup and so by Lemma
\ref{lem:pnorml}(iv), $G$ has an even subdegree. However, $|\Omega|$ is odd so to show
that $G$ is not \tont\ we find two distinct subdegrees as follows. Let
$\{e_1,e_2,f_1,f_2\}$ be a symplectic basis for $V$ and suppose that $H$ is the stabiliser
of $x=\{\la e_1\ra,\la e_1,e_2\ra\}$.  Then
$$L_x^{\la e_1,e_2\ra} =\left\{\left(\begin{array}{cc} \lambda_1 & 0 \\ \mu
  &\lambda_2 \end{array}\right) \Big\vert \lambda_1,\lambda_2\in\GF(q)\backslash\{0\},\mu\in\GF(q)\right\}.$$  Let $y=\{\la e_2\ra,\la e_1,e_2\ra\}$. Then $|G_x:G_{xy}|=2q$ as $G_{xy}$
can no longer interchange $\la e_1\ra$ and $\la e_1,e_2\ra$.  Next let $z=\{\la
e_1+e_2\ra,\la e_1+e_2,f_1+f_2\ra\}$. If $g\in L_{xz}$ maps $e_1$ to $\lambda_1 e_1$ and $e_2$ to $\mu e_1+\lambda_2 e_2$ we must have that $\lambda_2=\lambda_1+\mu$.  Thus $q-1$ divides $|L_x:L_{xz}|$ and hence also $|G_x:G_{xz}|$. Since $q>2$ it follows that
$G$ is not \tont.

Next suppose that $H=N_G(X)$ where $X$ is the stabiliser in $O^+_4(q)$ of a nondegenerate
2-space. Now $O^+_4(q)$ contains an involution $y=[J_2^2]$ where the two Jordan blocks are
hyperbolic 2-spaces and also an involution $z=[J_2^2]$ where the two Jordan blocks are
elliptic 2-spaces. Both $y$ and $z$ are of type $c_2$ in the notation of \cite{aschseitz}. Hence by Proposition
\ref{prn:burnessbounds} we can always choose an involution $x\in H\cap L$ such that
$|x^G|>\frac{1}{2}q^6$. Now $H\cap L=D_{2(q\pm 1)}\Wr S_2$ which contains at most
$(q+1)^2+2(q+1)=(q+1)(q+3)$ involutions. 
Since $q>2$, it follows from Lemma \ref{lem:mvconjclass}(ii)
that $G$ has an even subdegree. Since $|\Omega|$ is even, $G$ is not \tont.

Finally suppose that $H=N_G(X)$ where $X\cong C_{q^2+1}$. Here 
$H\cap L$ is
contained in the extension field subgroup $\Sp_2(q^2).2$ and $H\cap L=C_{q^2+1}\rtimes
C_4$. Let $x\in H\cap L$ be an involution. Then $x$ is an involution of $\Sp_2(q^2)$. Letting $\{e_1,f_1\}$ be a symplectic basis for $V$ over $\GF(q^2)$ we may assume that $x$ interchanges $e_1$ and $f_1$.  If $B:V\times V\rightarrow \GF(q^2)$ is the $\GF(q^2)$-alternating form preserved by $H\cap L$ then we can take $\overline{B}=Tr_{q^2\rightarrow q}\circ B:V\times V\rightarrow\GF(q)$ to be the $\GF(q)$-alternating form. Given $\mu\in\GF(q^2)\backslash\GF(q)$ we have $Tr(\mu^2)\neq 0$. Thus as an element of $\Sp_4(q)$ we have that $\nu(x)=2$ and $x$ maps $\mu e_1$ to $\mu f_1$ with $\overline{B}(\mu e_1,\mu f_1)=Tr(\mu^2)\neq 0$. Hence by \cite[(7.6)]{aschseitz}, $x$ is of type $c_2$. Thus Proposition \ref{prn:burnessbounds} implies that  $|x^G|>\frac{1}{2}q^6$.  Since $|x^G\cap
L|=q^2+1$, Lemma \ref{lem:mvconjclass}(ii) implies that $G$ has an even subdegree. Since
$|\Omega|$ is even, $G$ is not \tont.

\subsubsection{8-dimensional orthogonal groups}

The group $L=\POmega^+_8(q)$ has a graph automorphism of order three which permutes the
set of totally singular 1-spaces and the two classes of maximal totally singular
4-spaces. Any automorphism of $G$ that induces a permutation of order three on the corresponding 
three classes of subgroups is called a \emph{triality automorphism}. If $G$ is an almost
simple group with socle $L$ and contains such a triality automorphism then there are
several extra families of maximal subgroups that we need to consider. These are given in
\cite{KLOplus8q} and are as follows. We let $d=(2,q-1)$.

\begin{enumerate}
\item $G_{\{U,X,W\}}$ where $U$ is a totally singular 1-space, $X$ and $W$ are totally
  singular 4-spaces with $\dim(X\cap W)=3$, and $U< X\cap W$.

\item $N_G(G_2(q))$. 

\item $N_G(N_1)$, where $N_1$ is the intersection of the stabiliser in $L$ of an
  anisotropic 2-space with the normaliser in $L$ of an irreducible $\SU_4(q)$. The
  preimage of $N_1$ in $\Omega^+_8(q)$ is $(\frac{1}{d}C_{q+1}\times
  \frac{1}{d}\GU_3(q)).2^d$.

\item $N_G(N_2)$, where $N_2$ is the intersection of the stabiliser in $L$ of a hyperbolic
  line $\la e_1,f_1\ra$ and the stabiliser in $L$ of a decomposition of $V$ into two
  totally singular 4-spaces containing $e_1$ and $f_1$ respectively. Moreover, $q\neq
  2$. The preimage of $N_2$ in $\Omega^+_8(q)$ is $(\frac{1}{d}C_{q-1}\times
  \frac{1}{d}\GL_3(q)).2^d$.

\item $N_G(N_3)$ where $N_3=(D_{\frac{2}{d}(q^2+1)}\times D_{\frac{2}{d}(q^2+1)}).2^2$.

\item $N_G(N_4)$, where $N_4=[2^9]\rtimes \PSL_3(2)$ and $q=p\ge 3$. (We use $[2^9]$ to denote an unspecified group of order $2^9$.) When $q\equiv \pm 3 \pmod 8$ these have odd index.
\end{enumerate}

If $H$ is as in case (1) then by Lemma \ref{unique}, $G$ has a
unique subdegree which is a power of $p$. Hence $G$ is not
\tont. 

If $H$ is as in case (2) then $H\cap L=G_2(q)<\POmega_7(q)<L$. It was seen in Table
\ref{tab:notbase2}, that $\POmega_7(q)$ in its action on cosets of $G_2(q)$ has a
subdegree divisible by $p$ and so by Lemma \ref{lem:overgroup}(ii), so does $L$ and hence
$G$ also. Hence $G$ is not \tont\ (since $p$ divides $|\Omega|$).

If $H$ is as in case (3), note that $N_1=L\cap H<R<L$, where $R$ is a 4-dimensional
unitary group whose matrices have entries from $\GF(q^2)$. Moreover, the action of $R$ on
the set of cosets of $N_1$ in $R$ is the primitive action of $R$ on nondegenerate 1-spaces
over $\GF(q^2)$. We have already seen in Section \ref{sec:C1} that this action of $R$ has
a subdegree divisible by $p$ and so by Lemma \ref{lem:overgroup}(ii), $L$, and hence $G$,
has a subdegree divisible by $p$. Since $p$ divides $|\Omega|$ it follows that $G$ is not
\tont.

For $H$ in case (4), note that $N_2=L\cap H<R<L$, where $R$ is the stabiliser in $L$ of a
hyperbolic line. Moreover, the action of $R$ on the set of cosets of $N_2$ is equivalent to the
primitive action of $O^+_6(q)$ on decompositions of a 6-dimensional vector space into
complementary totally singular 3-spaces. We have already seen in Section
\ref{sec:C2} that this
action of $R$ has a subdegree divisible by $p$ and so be Lemma \ref{lem:overgroup}(ii),
$L$, and hence $G$, has a subdegree divisible by $p$. Since $p$ divides $|\Omega|$ it
follows that $G$ is not \tont.
  
For $H$ in case (5), by \cite[Proposition 3.3.1]{KLOplus8q}, $H\cap L=N_3$.  If $p$
divides $|H|$, either $p=2$, or $p$ is odd and $H$ contains an outer automorphism of order
$p$. Suppose first that $p=2$ and let $x\in H\cap L$ be an involution. Then by Proposition
\ref{prn:burnessbounds} and the fact that $x$ is not of $b$-type,
$|x^G|>\frac{1}{4}q^{10}$. Since $H\cap L$ contains at most $(q^2+1)^2.4$ involutions it
follows from Lemma \ref{lem:mvconjclass}(ii) that for $q\geq 8$, $G$ has a subdegree
divisible by $p=2$. For $q=4$, a \textsc{Magma} \cite{magma} calculation shows that $H\cap L$
contains only $391$ involutions while there is an involution $x\in H\cap L$ such that
$|x^L|=16707600$. Lemma \ref{lem:mvconjclass}(ii) then yields a subdegree divisible by
$p$. For $q=2$, a similar calculation reveals that $H\cap L$ contains 55 involutions and
contains an involution $x$ such that $x^L=3780$. Again, Lemma \ref{lem:mvconjclass}(ii)
yields a subdegree divisible by $p$. Next suppose that $p$ is odd and $H$ contains an
outer automorphism $x$ of order $p$. Then by Lemma \ref{lem:pslouter},
$|x^G|>\frac{1}{8}q^{14}$, while $|x^G\cap H|<|L\cap H|(p-1)\leq 4q(q^2+1)^2$. Thus Lemma
\ref{lem:mvconjclass}(ii) implies that $G$ has a subdegree divisible by $p$. Since
$|\Omega|$ is divisible by $p$ this implies that if $|H|$ is divisible by $p$ then $G$ is
not \tont.

If $|H|$ is coprime to $p$ (still with $H$ as in (5)) then $p$ is odd, and $p\ne 3$ as $H$ contains a triality. Let $x\in H\cap L$ be an involution.  We see
in the proof of \cite[Proposition 3.3.1]{KLOplus8q} that $N_3$ is a subgroup of
$(\Omega^-_4(q)\times\Omega^-_4(q)).2^2$, the stabiliser in $L$ of a pair
$\{W,W^{\perp}\}$ for some elliptic 4-space $W$. Thus we can choose $x$ such that $x$ acts
nontrivially on both $W$ and $W^\perp$ and so $\nu(x)\geq 2$. Then by \cite[Proposition
  3.37]{burness}, $|x^G|>\frac{1}{4(q+1)}q^{13}$. Now $|H\cap L|=(q^2+1)^2.4$ and so as $p>3$, Lemma \ref{lem:mvconjclass}(ii) implies that $G$ has an even subdegree. Since $|\Omega|$ is even, it follows that $G$ is not \tont.

Finally suppose that $H$ is as in case (6).  Here $L=\POmega^+_8(p)$ with $p\geq 3$ a
prime and $H\cap L=[2^9]\rtimes\PSL_3(2)$.  It follows from \cite[Proposition 3.4.2]{KL} that $H\cap L<R<L$ where $R=\POmega^+_8(2)$.  A \textsc{Magma} \cite{magma} calculation shows that the subdegrees of $R$ acting on the cosets of $H\cap L$ in $R$ are
$$\{1,14^2, 28, 168^2, 224,  448^2, 512\}$$ and by Lemma \ref{lem:overgroup}(ii) these are
also subdegrees of $L$ acting on $\Omega$. Thus $G$ is not \tont. Moreover, if $p$ divides $|H|$ then $p=3$ or 7 and $G$ has a subdegree divisible by $p$.

%
%

\section{Exceptional Groups of Lie Type}

In this section we prove Theorems \ref{tontas} and \ref{pTheorem} for
almost simple groups of exceptional Lie type.

\begin{theorem}\label{psubdegexcep} 
Let $G$ be an almost simple group with socle an exceptional group of Lie type in
characteristic $p$.  Suppose that $G$ acts primitively on a set $\Omega$ with point stabiliser $H$, and assume $p$ divides $|H|$. Then one of the following holds:
\begin{enumerate}
 \item $G$ has a subdegree which is divisible by $p$.
 \item $L=G_2(2)'$, $|\Omega|=28$ and $H=N_G(3^{1+2})$. (Here the subdegrees are 1, 27.)
\item $L=G_2(2)'$, $|\Omega|=36$, $H=\PSL_3(2)$ and $G=L$. (Here the subdegrees are 1, 7, 7, 21.)
\item $L={}^2\!G_2(3)'$,  $|\Omega|=9$, $H=N_G(2^3)$ and $G={}^2\!G_2(3)$. (Here the subdegrees are 1, 8.)
\end{enumerate}
\end{theorem}

\begin{theorem}\label{32excep} 
If $G$ is an almost simple group of exceptional Lie type which is $\frac{3}{2}$-transitive but not $2$-transitive on a set $\Omega$, then $G={}^2\!G_2(3)'$ of degree 28.
\end{theorem}

Note that $^2\!G_2(3)' = \PSL_2(8)$ so this case is recorded in part (iii) of Theorem \ref{tontas}.

In this section we will use Lie notation for classical groups instead of the notation used previously, for example $L_n(q)$ denotes $\PSL_n(q)$ and $U_n(q)$ denotes $\PSU_n(q)$. We also use $\Alt_n$ and $\Sym_n$ to denote the alternating and symmetric group of degree $n$ to avoid confusion with the group $A_n(q)$ of Lie type. Moreover, for a group $G$ of Lie type we use $W(G)$ to denote the Weyl group of $G$ and $G_0$ to denote the socle of $G$. For a finite group $X$, $\soc(X)$ denotes the socle of $X$.

\subsection{Preliminary lemmas}

As the simple groups $G_2(2)'$ and $^2\!G_2(3)'$ are the classical groups $\PSU_3(3)$ and $\PSL_2(8)$ respectively, we have already considered them in the analysis for \tont\ groups. The assertions about their subdegrees divisible by $p$ can be easily verified. From now on we will exclude the simple groups $G_2(2)'$ and $^2\!G_2(3)'$. In the next few results we also exclude $^2\!F_4(2)'$ -- this is dealt with separately in Lemma \ref{2f42}.

%

The next two lemmas give bounds for the sizes of certain unipotent classes in groups of
Lie type. The first follows from the determination of unipotent classes in exceptional
groups of Lie type (see \cite{miz1,miz2}, and \cite{LSei} for complete information), and the second from elementary calculations in
classical groups.

\begin{lemma}\label{bdsexcep}
Let $G=G(q)$, $q=p^a$ be a simple exceptional group of Lie type. Let $u_\alpha$ be a long
root element of $G$, and let $u$ be a non-identity unipotent element of $G$ which is not a
long root element (or a short root element when $(G,p) = (F_4(q),2)$ or $(G_2(q),3)$).
Then bounds for the sizes of the classes $u_\alpha^G$ and $u^G$ are given in Table
$\ref{bdtbl}$.
\end{lemma}

\begin{table}
\caption{Bounds for unipotent class sizes in exceptional groups}\label{bdtbl}
\[
\begin{array}{lll}
\hline
G & \hbox{bounds for }|u_\alpha^G| & \hbox{bounds for }|u^G| \\
\hline
E_8(q) & q^{58}<|u_\alpha^G|<2q^{58} & |u^G|>q^{92}  \\
E_7(q) &  q^{34}<|u_\alpha^G|<2q^{34} & |u^G|>q^{52}  \\
E_6^\epsilon(q) & q^{19}(q^3-1)<|u_\alpha^G|<2q^{22} & |u^G|>q^{31} \\
F_4(q) & q^{16}<|u_\alpha^G|<2q^{16} & |u^G|>q^{21} \\
G_2(q) & |u_\alpha^G|=q^6-1 & |u^G|\ge (q^6-1)(q^2-1) \\
^2\!F_4(q) & |u_\alpha^G|=(q^6+1)(q^3+1)(q^2-1) & |u^G|>q^{13}(q-1) \\
^3\!D_4(q) & |u_\alpha^G| > q^8(q^2-1) & |u^G|>q^{16} \\
^2\!G_2(q) & |u_\alpha^G|= (q^3+1)(q-1) & - \\
^2\!B_2(q) & |u_\alpha^G| = (q^2+1)(q-1) &  - \\
\hline
\end{array}
\]
\end{table}

\begin{lemma}\label{bdsclass}
Let $G=G(q)$, $q=p^a$ be a simple classical group, and let $u_\alpha$ be a long root
element of $G$. Upper bounds for the size of the class $u_\alpha^G$ are given in Table
$\ref{classtbl}$.
Moreover, if $p=2$ and $G = \SO_{2n}^{\epsilon}(q)$ ($n\ge 2$, $\epsilon = \pm$), then the
number of reflections in $G$ is $q^{n-1}(q^n-\epsilon)$.
\end{lemma}

\begin{table}
\caption{Bounds for class sizes of long root elements in classical groups}\label{classtbl}
\[
\begin{array}{cc}
\hline
G & |u_\alpha^G|\le  \\
\hline
A_n^\epsilon(q) & 2q^{2n} \\
B_n(q) &  2q^{4n-4}  \\
C_n(q) & q^{2n} \\
D_n^\epsilon(q) & 2q^{4n-6} \\
\hline
\end{array}
\]
\end{table}

\begin{lemma}\label{outerbd}
Let $G=G(q)$, $q=p^a$ be a simple exceptional group of Lie type.  Suppose $p>2$ and there
is an outer automorphism $\phi$ of $G$ of order $p$. Then one of the following holds:

{\rm (i)} $\phi$ is a field automorphism with centralizer $G(q^{1/p})$;

{\rm (ii)} $G = \,^3\!D_4(q)$, $p=3$ and $\phi$ is a graph automorphism with
centralizer $G_2(q)$ or $q^5.A_1(q)$.
\end{lemma}
\begin{proof} This follows from \cite[Section 4.9]{GLS}. \end{proof}

We shall need the following result concerning the maximal subgroups of exceptional groups
of Lie type.  This is an amalgamation of results from several papers, taken from
\cite[Theorem 8]{lseisurv}, where references can be found. In part (vii), $\bar G$ denotes
a simple algebraic group over $\overline{\GF(q)}$ of the same type as $G$, and $\sigma$ a
Frobenius morphism of $\bar G$ such that $G_0 = \bar G_\sigma'$.

\begin{lemma}\label{maxcub} 
Let $G$ be an almost simple group with socle $G_0 = G(q)\,(q=p^a)$ an exceptional group of
Lie type over $\GF(q)$, and let $H$ be a maximal subgroup of $G$. Then one of the following
holds.
\begin{enumerate}
\item[(i)] $H$ is parabolic.
\item[(ii)] $H$ is a subgroup of maximal rank, given by \cite{LSS}.
\item[(iii)] $\soc(H)$ is as in Table $\ref{socs}$.
\item[(iv)] $H = G(q_0)$, a subgroup of the same type as $G$ 
(possibly twisted), over a subfield $\GF(q_0)$ of $\GF(q)$.
\item[(v)] $H$ is a local subgroup, given by \cite[Theorem 1]{CLSS}.
\item[(vi)] $G_0 = E_8(q)$, $p>5$, and $H \cap G_0 = \PGL_2(q)\times
\Sym_5$ or $(\Alt_5\times \Alt_6).2^2$.
\item[(vii)] $G_0 = E_8(q)$, $E_7(q)$, $E_6^\epsilon(q)$ or $F_4(q)$, 
and $\soc(H) = H(r)$,
a group of Lie type over $\GF(r)$, where $r = p^b$. Moreover ${\rm rank}(H(r))
\le \frac{1}{2}{\rm rank}(G)$; and either $r\le 9$, or $H(r) = A_2^\pm (16)$,
or $H(r) \in \{ A_1(r),\,^2\!G_2(r), \,^2\!B_2(r)\}$. Finally, $H(r)$ is not
of the form $M_\sigma'$, where $M$ is a $\sigma$-stable subgroup of positive dimension
in $\bar G$.
\item[(viii)] $\soc(H)$ is a simple group that is not a group of Lie type of characteristic $p$, 
and the possibilities for $\soc(H)$ are given by \cite{nongen}.
\end{enumerate}
\end{lemma}

\begin{table}
\caption{Socles in Lemma \ref{maxcub}(iii)}\label{socs}
\[
\begin{array}{lll}
\hline
G_0  & \hbox{possibilities for }{\rm soc}(H) \\
\hline
G_2(q) & A_1(q)\,(p\geq7) \\
^3\!D_4(q) & G_2(q),\,A_2^\epsilon(q) \\
F_4(q) & A_1(q)\,(p\geq 13), \;G_2(q)\,(p=7),\; A_1(q)G_2(q)\,(p\geq
3,q\geq 5) \\
E_6^\epsilon(q) & A_2^\epsilon(q)\,(p\geq 5),\; G_2(q)\,(p\neq 7), \;C_4(q)\,(p\geq 3),
\;F_4(q), \\
& A_2^\epsilon(q)G_2(q)\,((q,\epsilon)\neq(2,-)) \\
E_7(q) & A_1(q)\,(2 \hbox{ classes}, p\geq 17,19), \;A_2^\epsilon(q)\,(p\geq 5),
\;A_1(q)A_1(q)\,(p \geq 5), \\
& A_1(q)G_2(q)\,(p\geq 3,q\geq 5),
\;A_1(q)F_4(q)\,(q\geq 4), \;G_2(q)C_3(q) \\
E_8(q) &  A_1(q)\,(3 \hbox{ classes}, p\geq 23,29,31),\; B_2(q)\,(p\geq 5),
\;A_1(q)A_2^\epsilon(q)\,(p \geq 5), \\
& G_2(q)F_4(q), 
\;A_1(q)G_2(q)G_2(q)\,(p \geq 3,q\geq 5), \\
& A_1(q)G_2(q^2)\,(p \geq 3,q\geq 5) \\
\hline
\end{array}
\]
\end{table}

\subsection{Classifying $p$-subdegrees}

In this section we prove Theorem \ref{psubdegexcep}. 
Let $G$ be an almost simple group with socle $G_0 = G(q)\,(q=p^a)$ an exceptional group of
Lie type in characteristic $p$.  Let $G$ act primitively on a set $\Omega$, let $H =
G_\alpha$ where $\alpha\in \Omega$, and suppose that $p$ divides $|H|$.
Now $H$ is a maximal subgroup of $G$. We treat the various possibilities
for $H$ given by Lemma \ref{maxcub}.  

We first deal with $^2\!F_4(2)'$.

\begin{lemma}\label{2f42} Theorem $\ref{psubdegexcep}$ holds when $G_0= {}^2\!F_4(2)'$.
\end{lemma}

\begin{proof} The subdegrees for all the primitive actions of ${}^2\!F_4(2)'$ were 
determined by \textsc{Gap} \cite{gap} calculations and are given in Table
\ref{tab:titsgroup}. In each case there is an even subdegree. By \cite[p74]{atlas} the
only maximal subgroup of ${}^2\!F_4(2)$, not arising from a maximal subgroup of
${}^2\!F_4(2)'$ is $C_{13}\rtimes C_{12}$. The subdegrees of this action are $1, 13, 26, 39^{10},  78^{37}, 52^8$ and $156^{1453}$.
\end{proof}

\begin{table}
\begin{center}
\caption{Subdegrees of primitive actions of ${}^2\!F_4(2)'$}
\label{tab:titsgroup}
\begin{tabular}{ll}
\hline
Degree & Subdegrees\\
\hline
351 &   126, 224 \\
351 &   126, 224 \\
364 &   12, 108, 243 \\
364 &   12, 108, 243 \\
378 &   52, 117, 208 \\
378 &   52, 117, 208 \\
2808 &  $56, 63, 84^2, 252^2, 504, 1512$ \\
3159 & $ 14, 64, 168, 224^2, 448^4, 672 $\\
3888 & $  78^2, 91^3, 182^3, 364^2, 546^2, 1092 $\\
7371 & $  18, 32^2, 64, 72^3, 96^4, 144^2, 288^{14}, 576^{4}$ \\
\hline
\end{tabular}
\end{center}
\end{table}

From now on assume that $G_0 \ne \,^2\!F_4(2)'$.  In view of Lemma
\ref{lem:mvconjclass}(ii), in proving Theorem \ref{psubdegexcep}
we may assume that for any non-identity $p$-element $u \in H$, we have
\begin{equation}\label{ineq} 
|u^G| < |u^G\cap H|^2.
\end{equation}

\begin{lemma}\label{parabol}
Theorem $\ref{psubdegexcep}$ holds if $H$ is a parabolic subgroup of $G$.
\end{lemma}
\begin{proof} This holds by Lemma \ref{lem:pnorml}(iv). \end{proof}

The next two lemmas deal with the proof of Theorem \ref{psubdegexcep} in the case where
$H$ is a subgroup of maximal rank, as in (ii) of Lemma \ref{maxcub}.  The lists of maximal
subgroups of maximal rank can be found in Tables 5.1 and 5.2 of \cite{LSS}: the subgroups
in Table 5.2 are normalizers of maximal tori in $G$, and those in Table 5.1 are not. It is
convenient to handle these cases separately.

\begin{lemma}\label{maxrk1}
Theorem $\ref{psubdegexcep}$ holds if $H$ is a subgroup of maximal rank which is not the
normalizer of a maximal torus.
\end{lemma}
\begin{proof} Here $H$ is as in \cite[Table 5.1]{LSS}. 
We list the possibilities for $H$ in Table \ref{Ttable}; for notational convenience each
subgroup is recorded via a subgroup $H^0$ of small index in $H$.

\begin{table}
\caption{Subgroups of maximal rank}\label{Ttable}
\[
\begin{array}{l|l}
\hline
G_0 & \hbox{ possibilities for }H^0  \\
\hline
E_8(q) & A_1(q)E_7(q),\, D_8(q),\, A_8^\epsilon(q),\, A_2^\epsilon(q)E_6^\epsilon(q), \\
       & D_4(q)^2,\, D_4(q^2),\, ^3\!D_4(q)^2,\, ^3\!D_4(q^2), \\
       & A_4^\epsilon(q)^2,\, ^2\!A_4(q^2),\,A_2^\epsilon(q)^4,\, ^2\!A_2(q^2)^2,
       \\ 
       & ^2\!A_2(q^4),\, A_1(q)^8 \\
E_7(q) & E_6^\epsilon(q)\cdot (q-\epsilon ),\, A_1(q)D_6(q),\, A_7^\epsilon(q),\,
         A_2^\epsilon(q)A_5^\epsilon(q), \\ 
       & A_1(q)^3\,D_4(q),\, A_1(q^3)\,^3\!D_4(q),\,A_1(q)^7,\, 
       A_1(q^7) \\
E_6^\epsilon(q) & A_1(q)A_5^\epsilon(q),\,A_2^\epsilon(q)^3,\, A_2(q^2)A_2^{-\epsilon}(q),\, 
            A_2^\epsilon(q^3), \\
          & D_4(q)\cdot (q-\epsilon)^2,\,^3\!D_4(q)\cdot (q^2+\epsilon q+1),\,
          D_5^\epsilon(q)\cdot(q-\epsilon) \\
F_4(q) & A_1(q)C_3(q),\, B_4(q),\, D_4(q),\, ^3\!D_4(q),\, A_2^\epsilon(q)^2 \\
       &  B_2(q)^2\,(q\hbox{ even}),\, B_2(q^2)\,(q\hbox{ even}) \\
G_2(q) & A_1(q)^2,\, A_2^\epsilon(q) \\
^2\!F_4(q)\,(q>2) & ^2\!A_2(q),\, ^2\!B_2(q)^2,\, B_2(q) \\
^3\!D_4(q) & A_1(q)A_1(q^3),\, A_2^\epsilon(q)\cdot (q^2+\epsilon q+1)  \\
^2\!G_2(q)\,(q>3) & A_1(q)\times 2 \\
\hline
\end{array}
\]
\end{table}

Suppose first that 
$G_0 \ne {}^2\!F_4(q)$ or $^2\!G_2(q)$ and $H^0$ 
does not contain a 
long root element of $G$. The only such cases are:
\begin{equation}\label{exce}
\begin{array}{rl}
G_0 = E_8(q): & ^2\!A_4(q^2),\,D_4(q^2),\, {}^3D_4(q^2),\, 
^2\!A_2(q^2)^2,\, ^2\!A_2(q^4) \\
             E_7(q): & A_1(q^7) \\
             E_6^\epsilon(q): & A_2^\epsilon(q^3) \\
             F_4(q): & B_2(q^2)\,(p=2) \\
\end{array}
\end{equation}
All of these cases are easily dealt with as follows.  For the $E_8(q)$ or $E_7(q)$ cases
observe that for a non-root unipotent element $u \in H^0$ we have $|u^G| > q^{92}$ or
$q^{52}$ respectively by Lemma \ref{bdsexcep}, and so we have $|u^G| > |u^G\cap H|^2$ even if we use the total
number of unipotent elements of $H$ as an upper bound for $|u^G\cap H|$; the conclusion
follows by Lemma \ref{lem:mvconjclass}(ii).  For the $E_6^\epsilon(q)$ case, observe that
$H^0 = A_2^\epsilon(q^3)$ arises from a subsystem subgroup $A_2^3$ of the algebraic group
$E_6$, and hence a transvection $u$ in $H^0$ lies in the unipotent class labelled by the
Levi $3A_1$ in $E_6$, and we have $|u^G| > q^{40}$ (see \cite{miz1}), hence again $|u^G| >
|u^G\cap H|^2$. And for the $F_4(q)$ case the same argument applies, using an element $u
\in H^0$ in the class $A_1\tilde{A}_1$ of $F_4$ (for which $|u^G| > q^{28}$ by
\cite{shinf4p2}).

Suppose now that none of the cases in (\ref{exce}) holds (still excluding $G_0 =\,
^2\!F_4(q)$ or $^2\!G_2(q)$). Then $H^0$ contains a long root element $u_\alpha$ of $G$.
Using \cite[1.13]{LLS1}, we see that if $u$ is any element of $u_\alpha^G \cap H$ then one
of the following holds:
\begin{enumerate}
\item[(i)] $u$ is a long root element of one of the quasisimple factors of $H^0$
\item[(ii)] $p=2$, $H^0$ has a factor $D_n^\epsilon(q)$, and $u$ is a reflection in a
subgroup $D_n^\epsilon(q).2$ of $H$.
\end{enumerate}
Lower bounds for $|u_\alpha^G|$ are given by Lemma \ref{bdsexcep}, and upper
bounds for $|u_\alpha^G \cap H|$ follow from Lemmas \ref{bdsexcep} and \ref{bdsclass}.
We find from these that 
$$
|u_\alpha^G| > |u_\alpha^G\cap H|^2,
$$
except in the following cases:
$$
\begin{array}{l}
G_0 = E_8(q),\;H^0 = A_1(q)E_7(q) \\
G_0 = E_7(q),\; H^0 = E_6^\epsilon(q)\cdot(q-\epsilon) \hbox{ or }A_1(q)D_6(q)
\\
G_0 = E_6^\epsilon(q),\; H^0 = D_5^\epsilon(q)\cdot (q-\epsilon) \\
G_0 = F_4(q),\; H^0 = B_4(q),\,D_4(q) \hbox{ or }^3\!D_4(q) \\
G_0 = G_2(q),\; H^0 = A_2^\epsilon(q)
\end{array}
$$
Hence by Lemma \ref{lem:mvconjclass}(ii) we may assume that one of these cases holds. 

Consider the case where $G_0 = E_8(q)$, $H^0 = A_1(q)E_7(q)$. If we let $T$ be a subsystem
subgroup $^2\!E_6(q)$ of $H$, then $C_G(T)$ contains a subgroup $^2\!A_2(q)$ not lying in
$H$, so $N_G(T) \not \le H$. Also for any Sylow $p$-subgroup $S$ of $H$, $\la T,S \ra$
contains the factor $E_7(q)$ of $H^0$ by Lemma \ref{lem:Tits}.  Hence $G$ has a subdegree
divisible by $p$, by Lemma~\ref{lem:weaklyclosed}.

Next let $G_0 = E_7(q)$, $H^0 = E_6^\epsilon(q)\cdot(q-\epsilon)$ or $A_1(q)D_6(q)$. In
the latter case we take $T$ to be a subgroup $^2\!A_5(q)$ of the $D_6$ factor; then
$C_G(T)$ contains a subgroup $^2\!A_2(q)$ not lying in $H$, and the argument of the
previous paragraph applies. And when $H^0 = E_6^\epsilon(q)\cdot(q-\epsilon)$, define $T$
to be a subgroup $F_4(q)$ of $H^0$, and note that $C_G(T)$ contains a subgroup $A_1(q)$
(see \cite[4.6]{rooty}), whereas $C_H(T)$ does not, provided $(q,\epsilon) \ne
(2,-)$. Thus with this exception, Lemma~\ref{lem:weaklyclosed} again gives the
conclusion. If $(q,\epsilon) = (2,-)$, then $G = E_7(2)$, $H = 3.\,^2\!E_6(2).S_3$. Choose
a subsystem subgroup $T = D_4(2)$ of $H$. From \cite[Table 5.1]{LSS} we see that
$|N_H(T)/T| = 2^23^4$, while $|N_G(T)/T| = 2^43^4$. Hence $N_G(T) \not \le H$. Moreover
$T$ lies in no parabolic subgroup of $H$, and hence $\la S,T \ra$ contains
$3.\,^2\!E_6(2)$ for any Sylow 2-subgroup $S$ of $H$, by Lemma \ref{lem:Tits}. Now Lemma
\ref{lem:weaklyclosed} gives the conclusion.

Now let $G_0 = E_6^\epsilon(q)$, $H^0 = D_5^\epsilon(q)\cdot (q-\epsilon)$.  For $\epsilon
= -$ take $T$ to be a maximal torus of order $(q+1)^6$ in $H$ (or of index $(3,q+1)$ in
this), and apply Lemma \ref{lem:weaklyclosed} with the characteristic $p$. Now suppose $\epsilon = +$, $H^0
= D_5(q)\cdot (q-1)$. If $q$ is odd then $H = C_G(t)$ for some involution $t \in
G$. There is a subgroup $D = D_4(q)$ of $H$ such that $Z(D) = \la t,u \ra$, where $u$ is a
conjugate of $t$. Then $C_H(t,u)$ is a 2-point stabiliser and the subdegree $|H:C_H(t,u)|$
is divisible by $p$, as required. If $q$ is even, note that $G$ contains a graph
automorphism, since otherwise $H$ lies in a parabolic subgroup. Let $T$ be a maximal torus
in $H$ of order $(q^5-1)(q-1)$ (or of index $(3,q-1)$ in this), lying in an $A_5$ Levi
subgroup.  This torus is not normalized by a graph involution of $D_5(q)$, whereas
$N_G(T)/T$ does contain a involution (see \cite{carcomp}). Hence $N_G(T) \not \le H$, and
so $T$ lies in a 2 point stabiliser $H \cap H^g$ for some $g\in G$. If the subdegree
$|H:H\cap H^g|$ is odd, then $H\cap H^g$ must contain the derived subgroup of an
$A_4(q)$-parabolic of $H$. However such a parabolic is not normalized by a graph
involution of $D_5(q)$, so since $H$ contains such a graph involution, it follows that
$|H:H\cap H^g|$ is even anyway.

When $G_0 = F_4(q)$, $H^0 = B_4(q)$, $D_4(q)$ or $^3\!D_4(q)$, we take $T = (q+1)^4$,
$^2\!A_2(q)$ or $A_2(q)$ respectively.  In the first case $T$ is a maximal torus and
$N_G(T)$ induces $W(F_4)$ on $T$, so $N_G(T) \not \le H$; and in the second and third
cases $N_G(T)$ contains $^2\!A_2(q)^2$ or $A_2(q)^2$, so again $N_G(T) \not \le H$. Now
the conclusion follows from Lemma \ref{lem:weaklyclosed}.

When $G_0 = G_2(q)$, $H^0 = A_2^\epsilon(q)$, the subdegrees are given by
\cite[Proposition 1]{LPSclosures} for $\epsilon = -$, and by \cite[6.8]{LPSsub} for
$\epsilon=+$, and there is a subdegree divisible by $p$.

Finally, we need to handle the cases where $G_0 =\, ^2\!F_4(q)$ or $^2\!G_2(q)$.  Consider
the first case $^2\!F_4(q)$. Here $q\ge 8$ (as we have already dealt with the $q=2$ case),
and $H^0 = \,^2\!A_2(q)$, $^2\!B_2(q)^2$ or $B_2(q)$.  Let $u$ be an involution in $H^0$,
and when $H^0 \ne \,^2\!A_2(q)$ take $u$ to be a non-root involution. Then by Lemma
\ref{bdsclass} we have $|u^G| > q^9(q^2-1)$, and also $|u^G| > q^{13}(q-1)$ when $H^0 \ne
\,^2\!A_2(q)$. If $i_2(H)$ denotes the number of involutions in $H$, then by
\cite[1.3]{LLS2} we have $i_2(H) < 2(q^5+q^4)$, $4(q^3+q^2)^2$ or $2(q^6+q^5)$ in the
respective cases for $H^0$. Hence we see that $|u^G| > i_2(H)^2 \ge |u^G\cap H|^2$, giving
the conclusion by Lemma \ref{lem:mvconjclass}(ii).  Finally, if $G_0 =\, ^2\!G_2(q)$, $H^0
= 2\times L_2(q)$, then for an element $u \in H^0$ of order 3 we have $|u^G| =
\frac{1}{2}q(q^3+1)(q-1)$ (see \cite{ward}), while $|u^G\cap H| = q^2-1$, so again Lemma
\ref{lem:mvconjclass}(ii) gives the result. \end{proof}

\begin{lemma}\label{maxrk2}
Theorem $\ref{psubdegexcep}$ holds if $H$ is the normalizer of a maximal torus.
\end{lemma}
\begin{proof} Here $H$ is as in \cite[Table 5.2]{LSS}. If $p$ does not divide
$|H\cap G_0|$, then $H$ contains an outer automorphism $u$ of order $p$, and bounds for
  $|u^G|$ are given by Lemma \ref{outerbd}. Otherwise, $p$ divides $|H\cap G_0|$ (and in
  particular $p$ divides $|W(G)|$), and it is clear from the action on 
  the adjoint module $L(G)$ that $H$
  contains a non-identity $p$-element $u$ which is not a long root element of $G$; bounds
  for $|u^G|$ are given in Lemma \ref{bdsexcep}.  We may assume by Lemma
  \ref{lem:mvconjclass}(ii) that $|H| > |u^G|^{1/2}$, and so from the above bounds, we see
  that $H$ is as in the following table (recall that we are assuming that $G_0 \ne
  \,^2\!F_4(2)'$ in view of Lemma \ref{2f42}):
\[
\begin{array}{lll}
\hline
G_0 & H \cap G_0 & q \\
\hline
E_7(q) & 3^7.W(E_7) & q=2 \\
E_6^\epsilon(q) & ((q+1)^6/(3,q+1)).W(E_6) & q=2 \hbox{ or }3, \,\epsilon=- \\
          & 7^3.(3^{1+2}.\SL_2(3)) & q=2,\,\epsilon=+ \\
F_4(q) & 7^2.(3\times\SL_2(3)) & q=2 \\
^2\!B_2(q) & (q-1).2,\; (q\pm \sqrt{2q}+1).4 & \\
\hline
\end{array}
\]
Consider the case where $G_0 = E_7(q), q=2$ and $H = 3^7.W(E_7)$. Here $H$ has an element
$x$ of order 8. Inspection of \cite{miz2} shows that the smallest class of elements of
order 8 in $G$ is the class labelled $D_4(a_1)$, which has size greater than $q^{94}/6$.
However $|H| = 3^7|W(E_7)| < 2^{34} < 2^{47}/\sqrt{6}$, a contradiction.

The cases in the table with $G_0 = E_6^\epsilon(q)$ are handled similarly, using an
element of order 8 in $H$ (if $q=2,\epsilon=-$), an element of order 4 (if $q=2,
\epsilon=+$), an element of order 9 (if $q=3$), and \cite{miz1} for the classes of
$G$. Likewise, for $G=F_4(q)$, $q=2$, we use an element of order 4 in $H$, together with
\cite{atlas} for the classes.

Finally, in the case where $G_0 =\, ^2\!B_2(q)$, pick an element $u$ of order 2 or 4 in
$H$ and observe using \ref{bdsexcep} that $|u^G\cap H| < |u^G|^{1/2}$.  This completes the
proof. \end{proof}

\begin{lemma}\label{tbl3}
Theorem $\ref{psubdegexcep}$ holds if $H$ is 
as in Lemma $\ref{maxcub}(iii)$.
\end{lemma}
\begin{proof} Here $H_0 = {\rm soc}(H)$ is as in Table \ref{socs}. 
In all cases $H$ contains a unipotent element $u$ which is not a long root element. By
Lemma \ref{bdsexcep} and our assumption that $|u^G| < |u^G\cap H|^2 < |H|^2$, we see that
$H$ is as in the following table:
\[
\begin{array}{ll}
\hline
G_0 & \hbox{possibilities for }H_0  \\
\hline
E_8(q) & G_2(q)F_4(q) \\
E_7(q) & G_2(q)C_3(q),\,A_1(q)F_4(q) \\
E_6^\epsilon(q) & F_4(q),\,C_4(q)\,(q \hbox{ odd}),\,A_2^\epsilon(q)G_2(q) \\
F_4(q) & A_1(q)G_2(q)\,(q \hbox{ odd}),\,G_2(q)\,(p=7) \\
^3\!D_4(q) & G_2(q),\,A_2^\epsilon(q) \\
\hline
\end{array}
\]
Suppose first that $G_0 \ne \,^3\!D_4(q)$.  The cases $(G_0,H_0) =
(E_6^\epsilon(q),C_4(q))$ $(p>2)$ and $(F_4(q),G_2(q))$ $(p=7)$ are covered by Lemma
\ref{lem:23/4}. In all other cases $H$ contains a long root element $u_\alpha$ of $G$, and
using \cite[1.13]{LLS1}, we see that $u_\alpha^G \cap H$ consists of long root elements of
$H$. Now a check using Lemma \ref{bdsexcep} shows that, excluding the case
$(E_6^\epsilon(q),F_4(q))$, we have $|u_\alpha^G| > |u_\alpha^G\cap H|^2$, a
contradiction.  In the exceptional case $G_0 = E_6^\epsilon(q)$, $H_0 = F_4(q)$, and we
take a subsystem subgroup $T = \,^3\!D_4(q)$ of $H$. Then $T$ is centralized by an element
of order $q^2+\epsilon q+1$ in $G$ not in $H$, so $N_G(T) \not \le H$. Now the conclusion
follows from Lemma \ref{lem:weaklyclosed}.

Suppose finally that $G_0 =\,^3\!D_4(q)$. If $H_0 = G_2(q)$, let $T$ be a subsystem
subgroup $\SL_3(q)$ of $H_0$. Then $N_G(T) \not \le H$ as $G$ has an element of order
$q^2+q+1$ centralizing $T$, so the result follows from Lemma \ref{lem:weaklyclosed}. Now
consider $H_0 = A_2^\epsilon(q)$. We may assume that $p=2$ by Lemma \ref{lem:23/4}. Let $u
\in H_0$ be an involution which is not a long root element of $G$, so that $|u^G| >
q^{16}$ by Lemma \ref{bdsexcep}. Then $|u^G\cap H|$ is certainly no more than the total
number of involutions in $H_0$, which is at most $2(q^5+q^4)$ by \cite[1.3]{LLS2}. Hence
again $|u^G| > |u^G\cap H|^2$. \end{proof}

\begin{lemma}\label{subfld}
Theorem $\ref{psubdegexcep}$ holds if $H$ is as in Lemma $\ref{maxcub}(iv)$.
\end{lemma}
\begin{proof} Here $H$ has socle $H_0 = G(q_0)$, a group of the same type as $G$ (possibly
twisted) over a subfield $\GF(q_0)$ of $\GF(q)$. We take $T$ to be a maximal torus of $H$ as
in the following table (for the existence of $T$, see \cite{carsln}). The table also gives
a primitive prime divisor $(q_0)_r$ of $q_0^r-1$ that divides $|T|$:
\[
\begin{array}{llll}
H_0 & |T| & (q_0)_r \hbox{ dividing }|T| \\
\hline
E_8(q_0) & q_0^8-q_0^7+q_0^5-q_0^4+q_0^3-q_0+1 & (q_0)_{30} \\
E_7(q_0) & (q_0^6-q_0^3+1)(q_0+1) & (q_0)_{18} \\
E_6^\epsilon(q_0) & q_0^6+\epsilon q_0^3+1 & (q_0)_9\,(\epsilon=+),\,(q_0)_{18}\,(\epsilon=-) \\
F_4(q_0) & q_0^4-q_0^2+1 & (q_0)_{12} \\
G_2(q_0) & q_0^2+q_0+1 & (q_0)_3 \\
^2\!F_4(q_0)' & q_0^2+\sqrt{2q_0^3}+q_0+\sqrt{2q_0}+1 & (q_0)_{12} \\
^3\!D_4(q_0) & q_0^4-q_0^2+1 & (q_0)_{12} \\
^2\!G_2(q_0) & q_0+\sqrt{3q_0}+1 & (q_0)_6 \\
^2\!B_2(q_0) & q_0+\sqrt{2q_0}+1 & (q_0)_4 \\
\hline
\end{array}
\]
In all cases $T$ lies in a maximal torus of $G$ which centralizes it, the order of which
is given by \cite{carsln}, and hence we see that $C_G(T) \not \le H$ with the following
exceptions: $H_0 = \,^2\!G_2(3) < G_2(3)$ and $H_0 = \,^2\!F_4(2) < F_4(2)$.  Moreover the
divisor $(q_0)_r$ of $|T|$ shows that $T$ is not contained in any parabolic subgroup of
$H_0$. Hence, apart from in the above exceptional cases, the conclusion follows from Lemma
\ref{lem:weaklyclosed}. As for the exceptional cases: when $H_0 = \,^2\!G_2(3) < G_2(3)$,
take $T = 2^3$, a Sylow 2-subgroup of $H_0$, and apply Lemma \ref{lem:weaklyclosed}; and
when $H_0 = \,^2\!F_4(2) < F_4(2)$, take $T = 3^{1+2}$, a Sylow 3-subgroup of
$H_0$. \end{proof}

\begin{lemma}\label{local}
Theorem $\ref{psubdegexcep}$ holds if $H$ is as in Lemma $\ref{maxcub}(v)$ or $(vi)$.
\end{lemma}
\begin{proof} Here $H$ is either a local subgroup given by \cite[Theorem 1]{CLSS},
or one of the subgroups in \ref{maxcub}(vi). If $u$ is an element of order $p$ in $H$, it
is easy using Lemma \ref{bdsexcep} to check that $|H|^2 < |u^G|$ except in the following
cases:
$$
\begin{array}{ll}
G_0 & H\cap G_0 \\
\hline
G_2(3) & 2^3.\SL_3(2) \\
^2\!E_6(2) & U_3(2) \times G_2(2) \\
E_7(3) & L_2(3) \times F_4(3) \\
E_7(q)\,(q \hbox{ odd}) & (2^2\times D_4(q).2^2).S_3 \\
\hline
\end{array}
$$ The last case is dealt with using Lemma \ref{lem:23/4}, and the second and third cases
are handled as in the proof of Lemma \ref{tbl3}.

When $G_0 = G_2(3)$ and $H\cap G_0 = 2^3.\SL_3(2)$, a \textsc{Gap} \cite{gap} calculation
shows that the subdegrees for $G_0$ are $14, 64, 168, 224^2, 448^4$ and $672$.
\end{proof}

\begin{lemma}\label{liep}
Theorem $\ref{psubdegexcep}$ holds if $H$ is 
as in Lemma $\ref{maxcub}(vii)$.
\end{lemma}
\begin{proof} Here $G_0$ is  $E_8(q)$, $E_7(q)$, $E_6^\epsilon(q)$ or $F_4(q)$, 
and $\soc(H) = H(r)$, a group of Lie type over $\GF(r)$, where $r = p^b$. Moreover
${\rm rank}(H(r)) \le \frac{1}{2}{\rm rank}(G)$; and either $r\le 9$, or $H(r) = A_2^\pm
(16)$, or $H(r) \in \{ A_1(r),\,^2\!G_2(r), \,^2\!B_2(r)\}$. Regard $G_0$ as $\bar
G_\sigma'$, where $\sigma$ is a Frobenius morphism of the simple algebraic group $\bar G$
over $\overline{\GF(q)}$ of the same type as $G$.

Assume that $p$ is odd. Now $|G:H|$ is even by \cite{odddegree}, so by Lemma
\ref{lem:23/4} we may assume that $H(r)$ is $A_2(r)$ or $A_4(r)$ (the latter only if $G_0
= E_8(q))$. Moreover $H(r)$ contains no long root element of $G$ by \cite{rooty}, so if $1
\ne u \in H(r)$ is a unipotent element, then our assumption $|u^G\cap H|^2 > |u^G|$,
together with Lemma \ref{bdsexcep}, leaves only the possibility
\[
H(r) = A_2(9) < F_4(3) = G_0.
\]
However $A_2(9)$ has an element of order $3^4+3^2+1$, whereas $F_4(3)$ has no torus
divisible by this number (see \cite{carsln}), so $A_2(9) \not \le F_4(3)$.

Now consider $p=2$. Suppose first that $H(r)$ contains a long root element $u_\alpha$ of
$G$. If $r>2$ we argue as in \cite[p.437]{LLS2} (fourth paragraph) that there is a
subgroup $M$ of positive dimension in $\bar G$ such that $H(r) = M_\sigma'$, contrary to
Lemma \ref{maxcub}(vii). Hence $r=2$ and ${\rm soc}(H) = H(2)$. Moreover, elements of
$u_\alpha^G\cap H(2)$ are root elements of $H(2)$ by \cite[Theorem 1]{tim}, so using Lemma
\ref{bdsexcep} we see that $|u_\alpha^G\cap H(2)|^2 < |u_\alpha^G|$, contrary to
assumption. 

Finally, assume that $p=2$ and $H(r)$ contains no long root element of $G$.  We may assume
that $G_0 \ne F_4(2)$ by \cite{NW}.  Let $u \in H(r)$ be an involution. Then $|u^G\cap
H(r)| \le i_2(H(r))$, the number of involutions in $H(r)$, so by assumption we have
$i_2(H(r)) > |u^G|^{1/2}$.  By \cite[1.3]{LLS2} we have $i_2(H)< 2(r^M+r^{M-1})$, where $M
= \dim \bar G - N$, $N$ being the number of positive roots in the root system of $\bar G$
(and $M$ is half this number when $H(r)$ is of type $^2\!F_4$, $^2\!G_2$ or
$^2\!B_2$). Also lower bounds for $|u^G|$ by Lemma \ref{bdsexcep}. One now checks that the
only possibilities for $H(r)$ satisfying the inequality $i_2(H) > |u^G|^{1/2}$ and also
having order dividing $|G|$ are as follows: 
\[
\begin{array}{ll}
G_0 & H(r) \\
\hline
E_8(2) & C_4(8),\,D_4^\epsilon(8) \\
E_7(2) & C_3(8),\,A_3^\epsilon(8) \\
E_6^\epsilon(2) & A_3^\epsilon(8),\,A_3(4),\,C_2(8),\,A_2(16),\,A_2^\epsilon(8),\,G_2(4) \\
E_6^\epsilon(4) & C_3(8) \\
F_4(4) & A_2^\epsilon(16) \\
F_4(q) & A_1(q^6) \\
\hline
\end{array}
\]
In most of these cases it is easy to use \cite{carsln} to produce an element of large
order in $H(r)$ which does not divide the order of a maximal torus of $G_0$: for example,
$D_4^\epsilon(8)$ has an element of order $2^9+1$, so cannot lie in $E_8(2)$, and so
on. The possibilities which do not succumb to this argument are:
\[
\begin{array}{l}
G_0 = E_6^\epsilon(2):\; H(r) = A_2^\epsilon(8),\,G_2(4) \\
G_0 = F_4(4):\; H(r) = A_2^-(16)
\end{array}
\]
For $G_0 = E_6^\epsilon(2)$ and $H(r) = A_2^\epsilon(8)$ we calculate $i_2({\rm
  Aut}(H(r))$ precisely and check that it is less than $2^{31/2}$, hence less than
$|u^G|^{1/2}$.  And for $H(r) = G_2(4)$, observe that $H(r)$ has a subgroup $\SL_3(4)$;
this centralizes a 3-element of $G_0$, from which we see that it is a subsystem group,
hence contains long root elements of $G_0$, contrary to assumption.  The case with $G_0 =
F_4(4)$ does not arise as $U_3(q^2) \not \le F_4(q)$ by \cite[4.5]{LSS}. This completes
the proof. \end{proof}

\begin{lemma}\label{nongen}
Theorem $\ref{psubdegexcep}$ holds if $H$ is as in Lemma $\ref{maxcub}(viii)$.
\end{lemma}
\begin{proof} Here $H_0 = {\rm soc}(H)$ is a simple group not in ${\rm Lie}(p)$.  
The possibilities for $H_0$ are given by \cite[Theorem 2]{nongen}.  In every case one
checks that all prime divisors of $|{\rm Out}(H_0)|$ also divide $|H_0|$; hence the fact
that $p$ divides $|H|$ implies that $p$ divides $|H_0|$.

Suppose first that $H_0$ contains a long root element $u_\alpha$ of $G$.  Then $p=2$ by
\cite[6.1]{rooty}.  Theorem 1 of \cite{tim} gives a list of possible isomorphism types for
$H_0$, and identifies $u_\alpha$ as a root involution for each type.  Combining this with
\cite{nongen}, we see that $H_0 = A_6$, $J_2$ or $Fi_{22}$ and $|u_\alpha^G\cap H| = 45$,
315 or 3510 respectively. Now the bound $|u_\alpha^G\cap H|^2 > |u^G|$, together with
Lemma \ref{bdsexcep}, reduces us to the following possibilities:
\[
\begin{array}{ll}
H_0 & G_0 \\
\hline
J_2 & G_2(4) \hbox{ or }F_4(2) \\
Fi_{22} & ^2\!E_6(2) \\
\hline
\end{array}
\]
If $(H_0,G_0) = (J_2,G_2(4))$ then $G$ has rank 3 and degree 416 (see \cite[p.97]{atlas}),
so has an even subdegree; and $J_2$ does not occur as the socle of a maximal subgroup of
$F_4(2)$ or its automorphism group by \cite{NW}. For $(H_0,G_0) = (Fi_{22},\,^2\!E_6(2))$,
let $T$ be a subgroup of $H$ of order 11. Then $C_G(T) \not \le H$ and $\la T,S \ra = H_0$
for any Sylow 2-subgroup $S$ of $H_0$, so Lemma \ref{lem:weaklyclosed} gives the
conclusion.

Now suppose that $H_0$ contains no long root element of $G$. Let $u \in H_0$ be an element
of order $p$. Then $|u^G\cap H|^2 > |u^G|$, so using \cite{nongen} and the lower bound for
$|u^G|$ in Lemma \ref{bdsexcep} (and also the known lists of maximal subgroups for $G_0$
of type $^2\!F_4$, $G_2$, $^3\!D_4$, $^2\!G_2$, $^2\!B_2$), we see that $G_0,H_0$ are as
in the following table:
\[
\begin{array}{lll}
G_0 & H_0 & T\\
\hline
^2\!E_6(2) & Fi_{22},\,\Omega_7(3),\,J_3,\,A_{12} & 11,13,\,[3^5],\,[3^5]
 \\
F_4(3) & ^3\!D_4(2) & [2^{12}] \\
F_4(2) & L_4(3),\,J_2,\,A_{10} & 13,\,[3^3],\,[3^4] \\
G_2(4) & L_2(13) & 7 \\
G_2(3) & L_2(13) & 7 \\
\hline
\end{array}
\]
(Recall that we already eliminated $G_0 = \,^2\!F_4(2)'$ in Lemma \ref{2f42} and the case $(G_0,H_0)=(G_2(4),J_2)$ was done in the previous paragraph.) 
For the remaining cases we choose a subgroup $T$ of $H_0$ as in the table and apply 
Lemma \ref{lem:weaklyclosed}. 
\end{proof}

\subsection{Proof of $\frac{3}{2}$-transitivity}

Here we prove Theorem \ref{32excep}.
Let $G$ be an almost simple group of exceptional Lie type with socle $G_0 = G(q)$
($q=p^a$), and suppose that $G$ acts primitively on a set $\Omega$ with point stabiliser
$H=G_\alpha$.

\vspace{4mm}
\noindent(A) Assume first that $p$ divides $|\Omega|$. If $p$ divides $|H|$ then by
Theorem \ref{psubdegexcep}, $G$ has a subdegree divisible by $p$, so cannot be \tont.  Now
consider the case where $H$ is a $p'$-group, that is, has order coprime to $p$. By Lemma \ref{maxcub} this means that one of
the following holds:
\begin{enumerate}
\item[(i)] $H$ is a maximal torus normalizer (as in Lemma \ref{maxcub}(ii));
\item[(ii)] $H$ is a maximal local subgroup (as in Lemma \ref{maxcub}(v));
\item[(iii)] $G_0 = E_8(q)$ and $H\cap G_0 = (\Alt_5 \times \Alt_6).2^2$ (as in Lemma \ref{maxcub}(vi));
\item[(iv)] ${\rm soc}(H)$ is a simple group that is not a group of Lie type of characteristic $p$ (as in Lemma \ref{maxcub}(viii)).
\end{enumerate}

Consider case (i). Here $H$ is as in \cite[Table 5.2]{LSS}. The fact that $p$ does not
divide $|H|$ implies that $G$ is of type $E_8$, $E_7$, $E_6^\epsilon$ or $^3\!D_4$ and $p
\ge 5,11,5$ or 3 respectively. Moreover, one checks that there is a prime $r \le 5$ which
divides both $|H|$ and $|G:H|$.  Let $s \in H$ be a (semisimple) element of order
$r$. Then by \cite[4.2]{LLS2}, we have $|s^G| > q^{112},q^{53},q^{31}$ or $q^{16}$
respectively. A glance at \cite[Table 5.2]{LSS} shows that $|H|^2$ is much less than
$|s^G|$. Hence by Lemma \ref{lem:mvconjclass}(ii), $G$ has a subdegree divisible by
$r$. Since $r$ divides $|G:H| = |\Omega|$, it follows that $G$ is not \tont\ on $\Omega$.

Now consider case (ii). Here $H$ is as in \cite[Theorem 1]{CLSS}, so as $H$ is a
$p'$-group, one of the following possibilities holds:
\[
\begin{array}{lll}
G_0 & H \cap G_0 & p \\
\hline
G_2(p) & 2^3.L_3(2) & p=5 \hbox{ or }p\ge 11 \\
F_4(p) & 3^3.L_3(3) & p \ge 5 \\
E_6^\epsilon(p) & [3^6].L_3(3) & p \ge 5 \\
E_8(p^a) & [2^{15}].L_5(2) & p\ge 11,\,a=1 \\
        & 5^3.L_3(5) & p \ge 7,\,a\le 2 \\
\hline
\end{array}
\]
For $G$ of type $E_8$ or $E_6^\epsilon$ we use the argument of the previous paragraph,
taking $r=2$. For $G = F_4(p)$ we also use this argument with $r=2$, noting that $|s^G\cap
H|$ is at most the number of involutions in $H$, which is 351, while $|s^G| > p^{16}$ by
\cite[4.2]{LLS2}. Finally for $G = G_2(p)$, 3 divides $|H|$ and $|G:H|$, and $H$ has 224
elements of order 3, so Lemma \ref{lem:mvconjclass}(ii) gives a subdegree divisible by 3
unless $p=5$. When $p=5$, $G$ has base 2 \cite[Table 12]{BLSbase} and so is not \tont.

Case (iii) is easily dealt with using the above argument with $r=2$.

Finally consider case (iv). Here the possibilities for $H_0$ are given by \cite{nongen}
(and also the known lists of maximal subgroups for $G_0$ of type $^2\!F_4$, $G_2$,
$^3\!D_4$, $^2\!G_2$, $^2\!B_2$). In all cases both $|H|$ and $|G:H|$ are divisible by
2. Taking an involution $s \in H$, it is easy to check that $|H|^2 < |s^G|$ with the
following exceptions: $H_0 = \,^3\!D_4(2) < F_4(5)$ and $H_0 = U_3(3) < G_2(5)$. However
in these exceptional cases one checks that $i_2(H)^2 < |s^G|$ (where $i_2(H)$ is the
number of involutions in $H$). Hence \ref{lem:mvconjclass}(ii) shows that there is an even
subdegree in all cases, and so $G$ is not \tont.

\vspace{4mm}
\noindent (B) Now assume that $p$ does not divide $|\Omega|$. Then $H$ is a parabolic
subgroup. 
By Lemma \ref{unique}, except in the cases where $G_0 = E_6(q)$ and 
$H = P_i$ ($i=1,3,5,6$), 
$G$ has a unique nontrivial suborbit of size a power of $p$, 
and hence is not \tont\ provided it is not 2-transitive
(which does occur when $G_0 = \,^2\!B_2(q)$ or $^2\!G_2(q)$).  Finally, consider the case $G_0 = E_6(q)$. We can take $H=P_1$ or $P_3$ (the others
are images of these under the graph automorphism). 
The subdegrees of $G$ on
cosets of $P_1$ are given in \cite{LiebeckSaxl86}, and are not equal.

It remains to consider the action of $E_6(q)$ with point stabiliser $P_3$.  Working in the
algebraic group with the usual labelling of the root system, we have $P_3 = QL$ with
unipotent radical $Q$ and Levi subgroup $L$, where
$$
Q = \left\la U_\alpha\mid\alpha = \sum_{i=1}^6 c_i\alpha_i, c_3>0 \right\ra, \;\; L = \la U_{\pm \alpha_i},T\mid
i = 1,2,4,5,6 \ra,
$$ where $T$ is a maximal torus. Let $n_0 \in N_G(T)$ project to the longest element $w_0$
of the Weyl group $W(E_6)$. Recall that $w_0$ acts on the root system as the negative of
the graph symmetry.  We calculate the intersection $P_3\cap P_3^{n_0}$ along the lines of
\cite[2.8]{car2}.  Observe that
$$
\begin{array}{ll}
L\cap L^{n_0} & = \la U_{\pm \alpha_i},T\mid i = 1,2,4,6 \ra, \\
Q\cap L^{n_0} &= \la U_\alpha \mid \alpha = \alpha_3,\alpha_{34},\alpha_{234},\alpha_{13},\alpha_{134},\alpha_{1234}
\ra, \\
L\cap Q^{n_0} &= \la U_\alpha \mid \alpha = -\alpha_5,-\alpha_{45},-\alpha_{245},-\alpha_{56},
-\alpha_{456},-\alpha_{2456} \ra, \\
Q\cap Q^{n_0} &= 1
\end{array}
$$ where we use the notation $\alpha_{ij...} = \alpha_i+\alpha_j+ \cdots$.  It follows
that $P_3\cap P_3^{n_0} = U_{12}.(L\cap L^{n_0}) = U_{12}.(A_1A_1A_2T_2)$, where $U_{12}$
is a unipotent group of dimension 12. Taking fixed points of a Frobenius morphism and
returning to the finite group $G_0 = E_6(q)$, we see that in its action on $P_3$, there is
a subdegree equal to
\begin{align*}
|P_3\,:\,P_3\cap P_3^{n_0}|& = \frac{q^{36}(q^5-1)(q^4-1)(q^3-1)(q^2-1)^2(q-1)}
{q^{17}(q^3-1)(q^2-1)^2(q-1)^2}\\
 &= q^{19}(q^2+1)(q^5-1)/(q-1).
\end{align*}
This subdegree does not divide $|G:P_3|-1$, so this action is not
\tont.

This completes the proof of Theorem \ref{32excep}.

%
%

\section{Sporadic Groups}

In this section we prove

\begin{theorem}\label{sporadtont}
  Every almost simple {\tont} primitive permutation group with socle a sporadic
  group, is doubly transitive. 
\end{theorem}

\begin{table}[H]
\caption{Non-base two actions of large sporadic groups}
\label{largenonbase2}
\[
\begin{array}{llll}
\hline
G & \hbox{point stabiliser }H & m = {\rm hcf}(|H|,n-1) & \hbox{Comment} \\
\hline
Th  & ^3\!D_4(2).3 & 1 & \\
    & 2^5.L_5(2) & 8 & \\
Fi_{23} & 2.Fi_{22} & - & \hbox{rank 3, subdegrees }3510,28160 \\
        & \POmega_8^+(3).S_3 & - & \hbox{rank 3, subdegrees }28431,109200\\
        & \Omega_7(3)\times S_3 & - & O_3(H) \ne 1 \\
        & \Sp_8(2) & 5 & \\
        & 3^{1+8}.2^{1+6}.3^{1+2}.2S_4 & 3 & \\
        & 2^2.U_6(2).2 & 2^2\cdot 7 \cdot 11 &  \\
        & 2^{11}.M_{23} & 46 & \\
Co_1 & Co_2 & - & \hbox{rank 4, subdegrees }4600,46575,47104 \\
     & 3.Suz.2 & 11 & \\
     & Co_3 & 23 & \\
     & U_6(2).S_3 & 1 & \\
     & (A_4\times G_2(4)).2 & 13 & \\
     & 2^{2+12}.(A_8\times S_3) & 2 & \\
     & 2^{4+12}.(S_3\times 3S_6) & 2 & \\
     & 2^{11}.M_{24} & 46 & \\
     & 2^{1+8}.\Omega_8^+(2) & 2 & \\
J_4 & 2^{11}.M_{24} & 4 &  \\
    & 2^{1+12}.3.M_{22} & 6 & \\
    & 2^{10}.L_5(2) & 1 & \\
Fi_{24}' & Fi_{23} & - & \hbox{rank 3, subdegrees }31671,275264 \\
         & 2.Fi_{22}.2 & 1 & \\
         & (3\times \POmega_8^+(3).3).2 & 1 & \\
         & 3^7.\Omega_7(3) & 1 & \\
         & \Omega_{10}^-(2) & 1 & \\
         & 3^{1+10}.U_5(2).2 & 3 & \\
         & 2^{11}.M_{24} & 46 & \\
Fi_{24}  & (2\times 2^2.U_6(2)).S_3 & 11 & \\
BM & 2.^2\!E_6(2).2 & - & O_2(H)\ne 1\\
   & (2^2\times F_4(2)).2 & - & O_2(H)\ne 1 \\
   & 2^{9+16}.\Sp_8(2) & 4 & \\
   & 2^{2+10+20}.(M_{22}.2 \times S_3) & 2 & \\
   & 2^{1+22}.Co_2 & 46 & \\
   & Fi_{23} & 3 & \\
   & Th & 31 & \\
M & 2.BM & - & O_2(H)\ne 1 \\   
\hline 
\end{array}
\]
\end{table}

Let $G$ be an almost simple primitive permutation group with socle a sporadic
  group $L$ and point stabiliser $H$. The base two permutation 
  representations of 
  such groups $G$ were determined in \cite{sporadicbase2} and \cite{NNOW}. Such groups 
  are not
Frobenius groups and so are not \tont. Hence we only need to consider 
those actions which are not base two.  
The non-base two actions of all nineteen 
sporadic groups of order up to $|Ly|$ and their automorphism groups are given
in Table \ref{table:sporadics}.
We were able to compute all the subdegrees for these actions, and these 
are listed in Table \ref{table:sporadics}. 
This gives much more
information
 than we actually need, but might be 
interesting to the reader.

For the eight almost simple sporadic groups larger than $Ly$, the non-base two 
primitive actions are as in Table \ref{largenonbase2}.
In the third column of the table, with a few exceptions, we give the 
highest common factor $m$ of the numbers $|H|$ and $n-1$, where $n$ is 
the degree $|G:H|$. If $G$ is \tont, the subdegree must divide $m$. When $G=Fi_{23}$ and $H=3^{1+8}.2^{1+6}.3^{1+2}.2S_4$, it follows that the subdegree is 3. However, a simple \textsc{Magma} \cite{magma} calculation finds subdegrees of length greater than 3. In all other cases in the table, $H$ is insoluble and so the subdegree must be at least 5. However, it is clear in all these cases that $H$ has no transitive action of degree at least 5 and dividing $m$, hence $G$ cannot be \tont. In the exceptional cases in the table where 
$m$ is not given, the fourth column either gives the subdegrees, or
states that $O_p(H) \ne 1$ for some prime $p$. In the latter cases we check
that $p$ does not divide $n-1$, hence $G$ is not \tont.

This completes the proof of Theorem \ref{sporadtont}.

\begin{center}\footnotesize
\begin{longtable}{lll}
\caption{Nontrivial subdegrees of some of the sporadic almost simple groups where the
  respective action does not have a base of size $2$.}\label{table:sporadics}\\
$G$&$H$&Subdegrees\\
\hline
$M_{11}$&$M_{10}$&$10$\\
        &$L_2(11)$&$11$\\
        &$M_9:2$& 18, 36\\
        &$S_5$& 15, 20, 30\\
$M_{12}$&$M_{11}$&$11$\\
        &$A_6.2^2$& $20, 45$\\
        &$L_2(11)$&  $11^2, 55, 66$\\
        &$3^2:\AGL_2(3)$&  $12, 27, 72, 108$\\
        &$2\times S_5$& $10^2, 15, 30^2, 60^3, 120$\\  
        &$2^{1+4}:S_3$& $6, 16, 24, 32^2, 48^2, 96^3 $\\
        &$4^2:D_{12}$& $6, 16, 24, 32^2, 48^2, 96^3 $\\
$M_{12}.2$&$3^{1+2}:D_8$ &$6, 18^2, 27, 54^3, 108^6$\\
$J_1$ & $L_2(11)$ & $11, 12, 110, 132$\\
$M_{22}$&$L_3(4)$&$21$\\
        &$2^4:A_6$&$16, 60$\\
        &$A_7$&$70, 105$\\
        &$2^4:S_5$& $30, 40, 160$\\
        &$2^3:L_3(2)$&$7, 42, 112, 168$\\
        &$M_{10}$&$30, 45, 180, 360$\\
        &$L_2(11)$&$55^2, 66, 165, 330$\\
$J_2$&$U_3(3)$& $36, 63$\\
        &$3.A_6.2$& $36, 108, 135$\\
        &$2^{1+4}:A_5$&$10, 32^2, 80, 160$\\
        &$2^{2+4}:(3\times S_3)$&$12, 32, 96, 192^2$\\
        &$A_4\times A_5$&$15, 20, 24, 180, 240, 360$\\
        &$A_5\times D_{10}$&$12, 25, 50, 60^2, 100^2, 150^2, 300$\\
$J_2.2$&$L_3(2):2\times 2$&$21, 28^2, 42, 84^2, 168^5, 336^2$\\
$M_{23}$&$M_{22}$&$22$\\
        &$L_3(4):2$& $42, 210$ \\
        &$2^4:A_7$& $112, 140$\\
        &$A_8$&$15, 210, 280$\\
        &$M_{11}$&$65, 330, 792 $\\
        &$2^4:(3\times A_5):2$&$20, 60, 90, 160, 480^3$\\
$HS$&$M_{22}$&$22, 77$\\
        &$U_3(5):2$&175\\
        &$L_3(4):2$&$42, 105, 280, 672$\\
        &$S_8$&$28, 105, 336, 630$\\
        &$2^4.S_6$&$15, 32, 90, 120, 160, 192, 240, 240, 360, 960, 1440$\\
        &$4^3:L_3(2)$&$28, 64, 112, 336, 448, 896^2, 1344$\\
        &$M_{11}$&$55, 132, 165, 495, 660, 792, 1320, 1980$\\
        &$4.2^4.S_5$&$30, 80, 128, 480, 640, 960, 1536, 1920$\\
$HS.2$&$(2\times A_6.2.2).2$&$24, 30, 45, 72, 180, 288, 360^5, 720^2, 1440^4, 2880^2$\\
$J_3$&$L_2(16):2$&$85, 120, 510, 680, 1360, 1360, 2040$\\
$J_3.2$&$L_2(16):4$&$85, 120, 510, 680, 2040, 2720$\\
        &$(3\times M_{10}):2$&$80, 135, 180, 540, 720, 1080^2, 1440^3,  2160^7$\\
$M_{24}$&$M_{23}$&23\\
        &$M_{22}:2$&$44, 231$\\
        &$2^4:A_8$&$30, 280, 448$\\
        &$M_{12}:2$&$495, 792$\\
        &$2^6:3.S_6$&$90, 240, 1440$\\
        &$L_3(4):S_3$&$63, 210, 630, 1120$\\
        &$2^6:(L_3(2)\times S_3)$&$42, 56, 1008, 2688$\\
        &$L_2(23)$&$253^2, 276^4, 506^2, 759, 1012^5, 1518^3, 3036^9$ \\
$McL$&$U_4(3)$&$112, 162$\\
        &$M_{22}$&$330, 462, 1232$\\
        &$U_3(5)$&$252, 750, 2625, 3500$\\
        &$3^{1+4}:2.S_5$&$90, 1215, 2430, 11664$\\
        &$3^4:M_{10}$&$30, 60, 162, 810, 1620^3, 3645, 5832$\\
        &$L_3(4):2$&$112, 210^3, 1120, 1260, 2520^2, 3360^3, 4032$\\
        &$2.A_8$&$210, 2240, 5040, 6720, 8064$\\
        &$2^4:A_7$&$112, 140, 210, 420, 672, 1680^2, 2240, 3360^2, 5040$\\
$McL.2$&$M_{11}\times 2$&$165, 220^2, 660, 792^2, 990, 1320, 1980^4, 3960^4, 5280, 7920^9$\\
$He$&$\Sp_4(4):2$&$136^2,425, 1360$\\
        &$2^2.L_3(4).S_3$&$105, 720, 840^2, 1344, 4480$\\
        &$2^6:3.S_6$&$90, 120, 384, 960^2, 1440, 2160, 2880^2, 5760, 11520$\\
$Ru$&$\,^2F_4(2)$&$1755, 2304$\\
        &$2^6.U_3(3).2$&$63, 756, 2016^3, 16128^2, 21504, 24192, 48384, 55296$\\
        &$(2^2\times Sz(8)):3$&$455, 3640, 5824, 29120^2, 58240, 87360^2, 116480$\\
        &$2^{3+8}:L_3(2)$&$28, 672, 896, 2688, 3584, 4096, 10752, 14336^2,$\\
        &&$28672, 43008^2, 57344, 86016, 114688$\\
        &$U_3(5):2$&$126,350,2520,5250,7875,10500,12600^2, 15750^3, 18000,21000, 63000^5, 126000$\\
        &$2^{1+4+6}.S_5$&$30, 240, 480, 640, 3840^2, 4096, 5120, 7680^2, 10240, 12288, 15360, 30720, 61440^6, 122880$\\
$Suz$&$G_2(4)$&$416, 1365$\\
        &$3.U_4(3).2$&$280, 486, 8505, 13608$\\
        &$U_5(2)$& $891, 1980, 2816, 6336, 20736$\\
        &$2^{1+6}.U_4(2)$&$54, 360, 1728, 5120, 9216, 17280, 46080, 55296$\\
        &$3^5:M_{11}$&$165, 891, 2673^2, 2916, 16038^2, 17820, 40095^3, 53460$\\
        &$J_2:2$&$200, 315, 630, 1800, 3150, 12600^2, 16800, 20160, 25200^2, 50400,100800^2$\\
        &$2^{4+6}:3A_6$&$60, 480, 1536, 1920, 6144^2, 20480, 23040^2, 46080, 92160, 184320$\\
        &$(A_4\times L_3(4)):2$&$224, 315, 420, 1260^2, 1680, 2520, 15120, 20160^2, 26880,$\\
        &&$ 30240^3, 40320, 60480, 80640, 120960^2, 161280^2$ \\
        &$2^{2+8}:(A_5\times S_3)$&$30, 40, 48, 480, 640^2, 960^2, 1536, 2048, 3072, 5120^2, 7680^3,$ \\
        &&    $10240^2, 15360^5, 30720^2, 61440^4, 92160^3, 122880^4$\\
$Suz.2$&$M_{12}:(2\times 2)$&$264, 495, 792, 990, 1760, 2640, 2970, 5280, 7920, 11880^3,15840^2,19008, $\\
        && $31680^3, 47520^{10}, 63360, 95040^5, 190080^6$\\
$O'N$&$L_3(7):2$&$5586, 6384, 52136, 58653$\\
$O'N.2$&$4.L_3(4):2$&$448, 630, 1120, 2240, 4480, 20160^5, 23040^2, 40320^3, 80640^4, 161280^{14}$\\
$Co_3$&$McL:2$&$275$\\
        &$HS$ &$352, 1100, 4125, 5600$\\
        &$U_4(3).(2.2)$&$224, 324, 1680, 4536^2, 8505, 18144 $\\
        &$M_{23}$&$253, 506, 1771, 7590, 8855, 14168, 15456$\\
        &$3^5:(2\times M_{11})$&$495, 2673, 32076, 40095, 53460$\\
        &$2.\Sp_6(2)$&$630, 1920, 8960, 30240, 48384, 80640$\\
        &$U_3(5):S_3$&$525, 2625, 3500, 6000, 21000, 23625, 31500, 63000^2, 126000^2, 189000$\\
        &$3^{1+4}:4S_6$&$180, 360, 3645, 7290, 14580, 29160, 58320^4, 69984, 349920$\\
        &$2^4.A_8$&$70, 840, 896^2, 960, 1120^2, 1920^2, 2688^2, 4480^4, 6720^2, 8960, 10080^2, 13440^7, 17920,$\\
        &&$  26880^6, 32256, 40320^4, 53760^8, 80640^3, 161280^2$\\
$Co_2$&$U_6(2):2$&$891, 1408$\\
        &$2^{10}:M_{22}:2$&$462, 2464, 21120, 22528$\\
        &$McL$&$275, 2025, 7128, 15400, 22275$\\        
        &$2^{1+8}:\Sp_6(2)$&$1008, 1260, 14336, 40320$\\
        &$HS:2$&$3850, 4125, 44352, 61600, 132000, 231000$\\
        &$(2^4\times 2^{1+6}).A_8$&$15, 210, 1680, 1920, 2520, 13440^2, 20160, 35840, 161280, 344064, 430080$\\
        &$U_4(3):D_8$&$840, 1134, 1680, 8505, 9072, 19440, 181440^3, 204120^2, 217728, 408240$\\
        &$2^{4+10}.(S_5\times S_3)$&$90, 120, 160, 480, 640, 720, 2880, 3840^2, 5760^2, 7680,15360^3, 16384, 23040^2, 40960^2, 46080,$\\
        &&$ 92160^2, 184320, 245760^2, 368640^2, 737280, 983040$\\
        &$M_{23}$&$506, 1771^2, 5313, 7590, 15456, 17710, 28336^2, 30360, 53130,60720, 70840, 85008^2, 141680, $\\
        && $ 170016, 212520^4, 283360, 425040^2, 510048, 850080$\\
$Fi_{22}$&$2.U_6(2)$&$693, 2816$\\ 
        &$\Omega_7(3)$&$3159, 10920$\\
        &$\Omega_8^+(2):S_3$&$1575, 22400, 37800$\\
        &$2^{10}:M_{22}$&$154, 1024, 3696, 4928, 11264, 42240, 78848$\\
        &$2^6:\Sp_6(2)$&$135, 1260, 2304, 8640, 10080, 45360, 143360, 241920^2$\\
        &$(2\times 2^{1+8}):(U_4(2):2)$&$270, 360, 1024, 1152, 4320, 34560, 40960, 46080, 69120^2, 138240, 368640, 442368$\\
        &$U_4(3):2\times S_3$&$560, 1680, 1701, 2520, 17010, 68040, 81648, 90720^2, 136080, 544320, 612360$ \\
        &${}^2F_4(2)'$&$1755, 11700, 14976, 83200^2, 140400, 187200, 374400, 449280, 2246400$\\
        &$2^{5+8}:(S_3\times A_6)$&$48, 180, 480, 1536, 5760^2, 7680, 8640, 11520, 24576^2,69120, 73728, 81920, $\\
        &&$138240, 184320^3, 552960, 983040, 1105920$\\
        &$3^{1+6}:2^{3+4}:3^2:2$&Only some of the subdegrees $432, 1296, 2187, 5184, 8748, 10368, 34992, 46656, 69984, $\\
        &&$139968, 279936, 314928, 419904, 559872, 839808, 1259712, 2519424$\\
$Fi_{22}.2$&$3^5:(2\times U_4(2):2)$ &Only some of the subdegrees  $360, 486, 2916, 6561, 8640, 19440, 38880, 58320, 104976,$\\
        &&$131220, 262440, 524880, 699840, 1049760, 1574640$\\
        &$G_2(3):2$&Only some of the subdegrees $702, 1456, 2808, 5824, 13104, 19656, 22113, 26208, 39312, $\\
        && $ 52416, 157248, 176904, 202176, 235872, 353808, 471744, 530712, 943488, 1415232, $\\
        && $2830464, 4245696$\\
$HN$&$A_{12}$&$462, 2520^2, 10395, 16632, 30800, 69300, 166320^2, 311850, 362880$\\
        &$2.HS.2$&$1408, 2200, 5775, 35200, 123200, 277200, 354816, 739200$\\
        &$U_3(8):3$&$1539, 14364, 25536^3, 68096, 131328, 229824^2, 612864^4, 689472, 787968^2, 5515776^2$\\  
$Ly$&$G_2(5)$&$19530, 968750, 2034375, 5812500$\\
        &$3.McL:2$& $15400, 534600, 1871100, 7185024$\\
\hline

\end{longtable}
\end{center}

\section*{Acknowledgements}
The authors thank an anonymous referee for a thorough reading of the manuscript.


\def\cprime{$'$}
\providecommand{\bysame}{\leavevmode\hbox to3em{\hrulefill}\thinspace}
\providecommand{\MR}{\relax\ifhmode\unskip\space\fi MR }
\providecommand{\MRhref}[2]{%
  \href{http://www.ams.org/mathscinet-getitem?mr=#1}{#2}
}
\providecommand{\href}[2]{#2}

\end{document}